\documentclass[a4paper,report, 11pt]{article}
\usepackage{amsmath,a4wide}
\usepackage{amssymb}
\usepackage{amscd}
\usepackage{epsfig}
\usepackage{graphics}
\usepackage{multirow}
\usepackage{graphicx}

\usepackage{subfigure} 
\usepackage{epstopdf}
\usepackage{booktabs}

\usepackage{ifpdf}
\usepackage[a4paper,top=1.5in, bottom=1.5in, left=0.5in, right=0.5in]{geometry}
\usepackage{changes}
\usepackage[makeroom]{cancel}
\usepackage{bbm}

\usepackage{amsthm}
\usepackage{color}
\definecolor{amaranth}{rgb}{0.9, 0.17, 0.31}	
\definecolor{myblue}{rgb}{0.2,0,0.9}
\definecolor{auburn}{rgb}{0.43, 0.21, 0.1}
\definecolor{bittersweet}{rgb}{1.0, 0.44, 0.37}
\definecolor{blue-violet}{rgb}{0.54, 0.17, 0.89}
\usepackage[authoryear]{natbib}

\usepackage{hyperref}

\hypersetup{
	colorlinks,linkcolor=myblue,citecolor=blue-violet, filecolor=magenta,urlcolor=myblue}

\newtheorem{pro}{Proposition}[section]
\newtheorem{pr}{Problem}[section]

\newtheorem{theorem}{Theorem}[section]

\newtheorem{lem}{Lemma}[section]

\newtheorem{as}{Assumption}[section]
\newtheorem{rem}{Remark}[section]

\def\b{\beta}
\def\d{\delta}
\def\D{\Delta}
\def\e{\epsilon}
\def\g{\gamma}
\def\l{\lambda}
\def\L{\Lambda}
\def\m{\mu}

\def\s{\sigma}
\def\t{\theta}

\def\z{\zeta}
\def\k{\kappa}

\numberwithin{equation}{section}

\begin{document}
	
\title{\LARGE \textbf{A Problem of Finite-Horizon Optimal Switching and \\Stochastic Control for Utility Maximization\footnote{Zhou Yang is supported by NNSF of China (No. 12171169, 12371470). Junkee Jeon is supported by the National Research Foundation of Korea(NRF) grant funded by the Korea government(MSIT) [Grant No. RS-2023-00212648]. \\}}}

	\author{
		  Zhou Yang\footnote{E-mail: \href{mailto:yangzhou@scnu.edu.cn}{\tt yangzhou@scnu.edu.cn}\;\;\;School of Mathematical Science, South China Normal University, China. }	
		\and 
      Junkee Jeon \footnote{E-mail: \href{mailto:junkeejeon@khu.ac.kr}{\tt junkeejeon@khu.ac.kr}\;\;\;Department of Applied Mathematics, Kyung Hee University, Republic of Korea.} 
	}
	
	\date{\today}
	
	\maketitle \pagestyle{plain} \pagenumbering{arabic}
	
    \abstract{In this paper, we undertake an investigation into the utility maximization problem faced by an economic agent who possesses the option to switch jobs, within a scenario featuring the presence of a mandatory retirement date. The agent needs to consider not only optimal consumption and investment but also the decision regarding optimal job-switching. Therefore, the utility maximization encompasses features of both optimal switching and stochastic control within a finite horizon. To address this challenge, we employ a dual-martingale approach to derive the dual problem defined as a finite-horizon pure optimal switching problem. By applying a theory of the double obstacle problem with non-standard arguments, we examine the analytical properties of the system of parabolic variational inequalities arising from the optimal switching problem, including those of its two free boundaries. Based on these analytical properties, we establish a duality theorem and characterize the optimal job-switching strategy in terms of time-varying wealth boundaries. Furthermore, we derive integral equation representations satisfied by the optimal strategies and provide numerical results based on these representations.}
	
	\vspace{1.0cm} {\em Keywords} : utility maximization, optimal switching problem,  double obstacle, job-switching with costs, dual-martingale approach, a system of parabolic variational inequalities \\

\newpage

\section{Introduction}

As the importance of job-switching options in individuals' financial planning has increased, research on the consumption-investment problem with a consideration of job-switching options has recently emerged in the field of financial mathematics literature. (Refer to \citet{SHIM2014145}, \citet{SKS18}, \citet{LSS19}, \citet{JEON2023127777}, \citet{JeonShim}, and the references mentioned therein). However, none of these studies have taken into account both practical factors, namely, job-switching costs and mandatory retirement dates.

Drawing inspiration from this gap, we investigate an agent's utility maximization problem that takes into account both job-switching options and mandatory retirement. Similar to the model of \citet{JeonShim}, we assume that the agent has only two jobs (or job types) available. One of the jobs provides a higher income compared to the other, but it comes with a higher utility cost or disutility due to labor. Additionally, the agent has the flexibility to switch from their current job to another at any time before the mandatory retirement date by incurring a fixed cost. Therefore, since the agent should consider the mandatory retirement date when making decisions not only about optimal consumption and investment but also about optimal job-switching, the utility maximization problem faced by the agent becomes a non-standard stochastic problem, encompassing the characteristics of both finite-horizon optimal switching and stochastic control.

To address this non-standard stochastic problem, we employ the dual-martingale approach, as in \citet{JeonShim}. If we apply the dynamic programming principle, similar to \citet{SHIM2014145}, we end up with a system of non-linear partial differential equations (PDEs). Dealing with these equations necessitates the use of separate linearization techniques, which can be particularly challenging to execute effectively, especially in a finite-horizon context. In contrast, utilizing the dual-martingale approach developed by \citet{JeonShim} offers the advantage that the resulting dual problem is automatically formulated as a linearized finite-horizon pure optimal switching problem. 
This optimal switching problem is solely related to job-switching decisions. By adobting the theory of impulse control, we can derive a system of parabolic variational inequalities(VIs) from the finite-horizon optimal switching problem. 

While we share many aspects with \citet{JeonShim}, particularly in considering job-switching costs and employing the dual-martingale approach, a significant distinction in our model is the incorporation of the mandatory retirement date. In particular, in \citet{JeonShim} case, dealing with an infinite horizon problem makes analysis relatively easier since it allows for obtaining a closed form of the system of VIs derived from the dual problem. However, in the case of our model, dealing with a system of parabolic VIs presents a challenge, as it requires analysis primarily using PDE theory.

On the other hand, \citet{Dai2010} and \citet{Dai-et-al-2016} conducted research on finite-horizon optimal switching problems related to stock trading. In particular, their studies involved the transformation of the system of parabolic VIs derived from the optimal switching problem into a parabolic double obstacle problem. They then analyzed the solutions and corresponding free boundaries in this context. By leveraging this connection, we can also derive the corresponding parabolic double obstacle problem for our finite-horizon optimal switching problem.  However, the mathematical models in both of these articles are pure optimal switching problem without a controlled process. Conversely, our model is a mixed control problem involving optimal switching coupled with controlled processes, which is more intricate and more challenging. 

There is extensive research related to the parabolic double obstacle problem in financial mathematics. For instance, in the context of the transaction cost problem, notable studies include those by \citet{Dai-et-al-2010}, \citet{DAI20091445}, and \citet{DZ16}. When considering the reversible investment problem, you can refer to the work of \citet{CHEN2012928}, and for insights into stock trading, the works of \citet{Dai2010} and \citet{Dai-et-al-2016} are valuable resources. Regarding the European option with transaction costs, \citet{Yi} stands out as a notable reference.

However, the double obstacle problem derived from our model has more challenging aspects compared to the double obstacle problem addressed in these studies. Firstly, in the existing literature, it is common for the starting points of the free boundary to lie on the terminal boundary. However, in our problem, the starting points of both free boundaries are not on the terminal boundary but rather at infinity. When the starting point is on the terminal boundary, its location can usually be determined directly using the terminal value condition and the differential equation. However, when the starting point is at infinity, we must initially make an educated guess about the asymptote of the free boundary and then demonstrate, using a carefully constructed function, that the free boundary indeed converges to this asymptote. 

Secondly, in the previously mentioned double obstacle problems of the existing literature, the monotonicity of the free boundaries was guaranteed. The monotonicity of the free boundary typically plays a crucial role in establishing the regularity of the free boundary.  However, in our specific problem, the free boundary is neither monotonic nor unbounded. As a result, proving the regularity of the free boundary in this context is exceptionally challenging. Our approach involves several steps: First, we establish that the free boundary is locally bounded. Then, we demonstrate that we can control the derivative of the difference between the solution and the obstacle with respect to time by utilizing the derivative of the difference with respect to spatial variables within the local domain. Next, we consider the VI in the local domain and employ a twisting transformation. This transformation converts the original VI into another VI, effectively twisting the free boundary, which is non-monotonic in the original problem, into a monotonic one in the new problem. By accomplishing this, we can apply classical methods to prove the regularity of the free boundary. Moreover, we establish the presence of a critical time beyond which no free boundary exists for transitioning to a higher-income job. Agents switch to such jobs to boost the present value of future income, even if it incurs expenses. Nevertheless, beyond this critical time, costs outweigh the increase in future income's present value, leading agents to never switch to higher-income positions.

Based on the theoretical results related to the  solution of our problem's double obstacle problem and the associated two free boundaries, we can recover the solutions to the system of VIs derived from the optimal switching problem. By utilizing the two free boundaries of the double obstacle problem, we construct an optimal switching strategy and, through a verification theorem, demonstrate that the solutions to the system of VIs ultimately become the solutions to the finite-horizon optimal switching problem. Finally, we establish the validity of the duality theorem and leverage it to derive optimal job-switching, consumption, and investment strategies in feedback form. Furthermore, we derive the  coupled integral equations satisfied by the two free boundaries of the double obstacle problem and use them to provide an integral equation representation of the optimal strategy. We employ \citet{Huang}'s well-established recursive integration method to numerically solve the integral equations and analyze the impact of job-switching costs and mandatory retirement dates on the optimal strategy through graphical illustrations.

There is a vast literature related to the finite-horizon optimal switching problem (see \citet{Hamad2007}, \citet{Djehiche2009}, \citet{DJEHICHE}, \citet{HAMADENE2010403}, \citet{Asri}, \citet{Martyr}, \citet{BOUFOUSSI2023126947}, and reference therein). These studies primarily focus on using the backward stochastic differential equation (BSDE) theory to establish the existence of solutions and optimal switching strategies or emphasize proving the existence of viscosity solutions for the associated system of VIs. However, our research differs in that we place emphasis not only on the existence, uniqueness, and regularity of strong solutions for the system of VIs derived from the optimal switching problem but also on the analysis of the related free boundaries.

Furthermore, our research is also related to \citet{YK}'s study, which explored the optimal retirement decision problem while considering the mandatory retirement date. In our model, when the cost of switching to a job that offers higher income is set to infinity, it reduces to a finite-horizon optimal retirement problem with only one irreversible retirement decision. \citet{YKS21}, \citet{CJW22}, and \citet{PW-sicon} each expanded upon \citet{YK}'s research by considering non-Markovian environments, partial information, and drift ambiguity, respectively. Our research also holds the potential for extension by considering these factors.

Our paper is structured as follows: In Section \ref{sec:model}, we explain the utility maximization problem for agents considering mandatory retirement dates and job-switching costs. Section \ref{sec:optimization} defines the finite-horizon pure optimal switching problem in utility maximization using the dual-martingale approach and derives the associated system of parabolic VIs. Section \ref{sec:double-obstacle} provides not only the existence and uniqueness of the solutions to the double obstacle problem related to our model but also analytical properties of the associated two free boundaries. Section \ref{sec:strategy} establishes the verification theorem and duality theorem for the optimal switching problem. In Section \ref{sec:numerical}, we numerically provide the agent's optimal strategy under the constant relative risk aversion(CRRA) utility function. Finally, Section \ref{sec:conclusion} concludes the paper.

\section{Model}\label{sec:model}

We consider an agent whose objective is to maximize the following expected utility: 
\begin{equation}
    {\bf U}:=\mathbb{E}\left[\int_0^T e^{-\b t}\left(\left(U_1(t,c_t)-L_0\right){\rm\bf I}_{\{\delta_t={\cal D}_0\}}+\left(U_1(t,c_t)-L_1 \right){\rm\bf I}_{\{\delta_t={\cal D}_1\}}\right)dt + e^{-\b T}U_2(T,W_T)\right],
\end{equation}
where $\b>0$ is the constant subjective discount rate,  $c_t\ge 0$ is the  rate of consumption at time $t$, $W_T$ is her/his wealth at time $T$, ${\bf  I}_D$ is the indicator function of set $D$,  $\delta_t\in\{{\cal D}_0, {\cal D}_1 \}$ denotes the agent's job at $t$,  $L_i>0\ (i=0,1)$ is a constant utility cost from labor at job ${\cal D}_i$, $U_1(\cdot, \cdot)$ is the felicity function of consumption, $T>0$ is a constant mandatory retirement date, and $U_2(\cdot, \cdot)$ is the felicity function of bequest (or terminal wealth). For simplicity, we assume that there exist two job to the agent, that is, $\delta_t$ takes one of the two values, ${\cal D}_0$ and ${\cal D}_1$. Moreover, the agent is free to switch between the two jobs at any time before the mandatory retirement date $T$. However, when switching from job $\mathcal{D}_i$ to job $\mathcal{D}_{1-i}$ $(i=0,1)$, a constant cost of $\zeta_i>0$ is incurred.  The agent receives a constant income of $\epsilon_i$ at job ${\cal D}_i$ ($i=0,1$). Job ${\cal D}_1$ is assumed to receive a higher income than job ${\cal D}_0$, but as a result, it incurs higher utility costs (or disutility). Under these circumstances, we assume that
$$
    0<L_0<L_1 \;\;\mbox{and}\;\;0\le \e_0 <\e_1.
$$
\begin{rem}
If we set $\epsilon_0$ and $\zeta_1$ to 0, our model can be viewed as a partially reversible retirement problem that takes into account mandatory retirement. Furthermore, if we set $\epsilon_0$ to 0, $\zeta_0$ to infinity, and $\zeta_1$ to 0, our model reduces to the optimal retirement problem described in \citet{YK}'s work.
\end{rem}

We assume that there exist two financial assets trading in the economy, a riskless asset and a  risky asset, whose prices at $t$ are denoted by $S_{0,t}$ and $S_{1,t}$, respectively.  The asset prices satisfy the dynamics: 
$$
\dfrac{dS_{0,t}}{S_{0,t}} = {r} dt \;\;\;\mbox{and}\;\;\;\dfrac{dS_{1,t}}{S_{1,t}}= {\mu} dt + \s dB_t, 
$$
where ${r}>0$ is the constant  risk-free rate, ${\mu}>0$ is the  constant  drift of the  risky asset price, $\s>0$ is the  constant  volatility of the risky asset returns, and $B$ is a standard Brownian motion on  $(\Omega, \mathcal{F}_\infty, \mathbb{P}).$ We will denote the augmented filtration generated by the Brownian motion $B$ by $\mathbb{F}=\{\mathcal{F}_t\}_{t\ge 0}$. 


Despite the existence of job switching costs, the primary reason for an agent to switch from job ${\cal D}_0$ to ${\cal D}_1$ is to increase the present value of future income flows.  However, if $\epsilon_0(1-e^{-rT})/r$ is greater than or equal to $\epsilon_1(1-e^{-rT})/r -\zeta_0$, there exists no motive to return to job ${\cal D}_1$ (or from ${\cal D}_0$ to ${\cal D}_1$). Thus, we can make the following assumption on the switching costs:
\begin{as}\label{as:cost}
    $$
     0 < \zeta_0 < \dfrac{(1-e^{-rT})}{r}(\epsilon_1 - \epsilon_0)\;\;\mbox{and}\;\; \zeta_1>0.
    $$
\end{as}


We will now describe the agent's wealth dynamics. Let $\Upsilon = \{\Upsilon_t\}_{t=0}^T$  denote the job indicator process such that 
$\Upsilon_t = {\rm\bf I}_{\{\d_t = {\cal D}_1 \}}$ for $t \geq 0$.
That is,
$$
\Upsilon_t =
\begin{cases}
1, \;\;\;&\mbox{if}\;\;\delta_t = {\cal D}_1,\\
0,\;\;\;&\mbox{if}\;\;\delta_t = {\cal D}_0\\
\end{cases}
~~\mbox{for}~~ t \geq 0.
$$ 

Let us denote the agent's  investment in the risky asset at time $t$ by $\pi_t$ (in dollar amount). In the presence of switching costs, the agent's wealth $W_t^{c,\pi,\Upsilon}$ corresponding to $(c,\pi,\Upsilon)$  follows the  dynamics:
\begin{align}
\label{eq:wealth-dynamics0}
dW_t^{c,\pi,\Upsilon} =& \left[rW_t^{c,\pi,\Upsilon} + (\mu-r)\pi_t - c_t + \left(\e_0(1-\Upsilon_t)+\e_1\Upsilon_t\right) \right]dt+\s \pi_t dB_t\\-&\zeta_0(\Delta \Upsilon_t)^+ -\zeta_1(\Delta \Upsilon_t)^-,\nonumber
\end{align}
where $\Delta \Upsilon_t = \Upsilon_{t+}-\Upsilon_{t}$ and $(\Delta \Upsilon_t)^\pm =\max\{\pm \Delta \Upsilon_t, 0\}$. Here, note that, if the job switching occurs from ${\cal D}_0$ to ${\cal D}_1$ at time $t$, then $\Delta \Upsilon_t =1$, whereas if a job switching occurs from ${\cal D}_1$ to  ${\cal D}_0$, then $\Delta \Upsilon_t =-1$. 

For a initial job $\Upsilon_0 = j \in \{0, 1\}$ and wealth $W_0 = w$, we call a strategy $(c,\pi,\Upsilon)$ admissible if it satisfies 
\begin{itemize}
    \item[(i)]  a job indicator process  $\Upsilon$ is $\mathbb{F}$-adapted, finite variation, c\`{a}gl\`{a}d process with values in $\{0,1\}$,
    \item[(ii)]  $c \ge 0$ and $\pi$ are $\mathbb{F}$-progressively measurable processes satisfying the following integrability conditions:
	\begin{align*}
	\int_0^T c_s ds <\infty \;\;\mbox{and}\;\;\int_0^T \pi_s^2 ds<\infty\;\;\mbox{a.s.},
	\end{align*}
    \item[(iii)] the agent is assumed to face the following wealth constraint (a dynamic budget constraint):
	\begin{equation}\label{eq:wealth_range}
	\begin{aligned}
	    W_t^{c,\pi,\Upsilon} \geq&  \left\{\left(-\frac{\epsilon_1(1-e^{-r(T-t)})}{r}+\zeta_0\right){\bf I}_{\{t\in[0,T-T_1)\}}+\left(-\dfrac{\e_0(1-e^{-r(T-t)})}{r}\right){\bf I}_{\{t\in[T-T_1,T]\}}\right\}(1-\Upsilon_t)\\-&\frac{\epsilon_1(1-e^{-r(T-t)})}{r}\Upsilon_t ~\mbox{for all}~ t \in[0,T] ~\mbox{a.s.},
	\end{aligned}
	\end{equation}
 where $T_1$ is defined as 
\begin{equation}   \label{de:T1}
  T_1:=-\frac{1}{r}\ln\left(1-\frac{r\zeta_0}{\e_1-\e_0}\right).
\end{equation}
\end{itemize}
We will denote the set of all admissible strategies  at $(j, w)$ by $\mathcal {A}(j, w)$. 

\begin{rem}
    Let us provide an additional explanation for condition (iii) in the context of an admissible strategy. Suppose we are in a situation where the agent's job is ${\cal D}_0$ at time $t$. According to the definition of $T_1$, if $t$ falls within the interval $[0, T - T_1)$, it implies that if the agent continues to remain in job ${\cal D}_0$ until time $T$, the present value of income up to $T$, denoted as $\epsilon_0(1-e^{-r(T-t)})/r$, will be less than $\epsilon_1(1-e^{-r(T-t)})/r -\zeta_0$. Despite taking into account the associated costs, there exists an incentive to switch to job ${\cal D}_1$ in order to secure a higher present value of income. On the contrary, if $t$ lies within the interval $[T-T_1, T]$, then because $\epsilon_0(1-e^{-r(T-t)})/r$ is greater than or equal to $\epsilon_1(1-e^{-r(T-t)})/r - \zeta_0$, the agent will never switch to job ${\cal D}_1$ and will continue to remain in job ${\cal D}_0$ until $T$. When we combine these two scenarios, we can conclude that the agent's wealth must always satisfy the natural borrowing constraint \eqref{eq:wealth_range}.

\end{rem}

As in \citet{YK}, we impose the following assumption on the felicity function $U_i$.
\begin{as}\label{as:utility}
    The utility functions $U_i(t,c)\in C^\infty([0,T]\times(0,\infty))$, $i=1,2$, takes value in $\mathbb{R}$, are strictly increasing and strictly concave with $c$, and satisfy the following conditions: for any $t\in[0,T]$
    \begin{equation}
        \lim_{c\to0+}\partial_c U_i(t,c)=+\infty,\;\;\lim_{c\to\infty}\partial_c U_i(t,c)=0,\;\;\mbox{and}\;\;\limsup_{c\to\infty}\max_{t\in[0,T]}\partial_c U_i(t,c)c^{k_i} \le C_i, 
    \end{equation}
    where $C_i$ and $k_i$ are positive constants depending only $U_i$. 
\end{as}

It follows from Assumption \ref{as:cost} that $T_1>0$, and without loss of generalization, we suppose that $T$ is large enough such that $T>T_1$. So, we have that
$$
0<T_1<T.
$$
In the perspective of admissibility, we can impose the following assumption for the initial wealth $w$:
\begin{as}\label{as:initial:as}
    For given $\Upsilon_0=j\in\{0,1\}$, the initial wealth $W_0=w$ satisfies 
    \begin{equation}\label{eq:initial:as}
            w > \left(-\frac{\epsilon_1(1-e^{-rT})}{r}+\z_0\right)(1-j)-\frac{\epsilon_1(1-e^{-rT})}{r}j.
    \end{equation}
\end{as}

Now, we can state our optimization problem as follows:
\begin{pr}\label{pr:main-cost} For the given initial job $\Upsilon_0=j \in \{0, 1\}$ and initial wealth $W_0=w$ satisfying \eqref{eq:initial:as}, we consider the following agent's maximization problem:
\begin{equation}
    \label{eq:orgvalueftn}
	V(j,w): = \sup_{(c,\pi,\Upsilon)\in {\cal A}(j, w)} \mathbb{E}\left[\int_0^T e^{-\b t}\left(\left(U_1(t,c_t)-L_0\right)(1-\Upsilon_t)+\left(U_1(t,c_t)-L_1 \right)\Upsilon_t \right)dt + e^{-\b T}U_2(T,W_T^{c,\pi,\Upsilon})\right].
\end{equation}
\end{pr}
Our strategy to tackle Problem \ref{pr:main-cost} is as follows: First, we utilize the dual-martingale approach developed by \citet{JeonShim} to derive the dual problem defined as a finite-horizon pure optimal switching problem (Section \ref{sec:optimization}). Then, we consider the parabolic double obstacle problem arising from the dual problem and analyze the analytical properties of the associated free boundaries using PDE theory (Section \ref{sec:double-obstacle}). Based on these results, we recover the solution to the optimal switching problem from the double obstacle problem and establish a duality theorem to fully characterize the optimal strategy (Section \ref{sec:strategy}).

\section{Optimization Problem}\label{sec:optimization}

We will utilize a dual approach developed by \citet{JeonShim}. Let $\t:=(\m-r)/\s$, the risk premium on the return of the risky asset for one unit of standard deviation, or the Sharpe ratio. We define the stochastic discount factor, ${\cal H}_t$:
$$
{\cal H}_t \equiv \exp\left(-\left\{r+\dfrac{1}{2}\t^2 \right\}t-\t B_t\right).
$$

The main advantage to utilize a dual approach is that we do not need to consider the portfolio choice in the perspective of solving our optimization problem. For this to be established in our problem, once the optimal consumption and job switching strategies are designed, the existence of a portfolio process that supports the strategies should be guaranteed. In line with this, the following proposition is a key to utilize the duality approach to Problem \ref{pr:main-cost}:
\begin{pro}\label{pro:static-budget-cost} Let $\Upsilon_0=j \in \{0, 1\}$ and $w$ satisfying \eqref{eq:initial:as} be given. 
	\begin{itemize}
		\item[(a)] For any triple $(c,\pi,\Upsilon)\in{\cal A}(j,w)$, the following static budget constraint holds:
		\begin{equation}
		\label{eq:static_budget-cost}
		\mathbb{E}\left[\int_0^T{\cal H}_t \left(c_t - \epsilon_0 (1-\Upsilon_t) - \epsilon_1 \Upsilon_t\right)dt +\sum_{0 \leq t \le T}[\z_0{\cal H}_t(\Delta \Upsilon_t)^+ + \z_1{\cal H}_t(\Delta \Upsilon_t)^- ]+ {\cal H}_T W_T^{c,\pi,\Upsilon} \right]\le w.
		\end{equation}
		\item[(b)] If a pair $(c,\Upsilon)$ of the consumption and job switching strategy with a non-negative random variable $\mathfrak{B}$ satisfy the following equation:
        \begin{equation*}
            \mathbb{E}\left[\int_0^T{\cal H}_t \left(c_t - \epsilon_0 (1-\Upsilon_t) - \epsilon_1 \Upsilon_t\right)dt +\sum_{0 \leq t \le T}[\z_0{\cal H}_t(\Delta \Upsilon_t)^+ + \z_1{\cal H}_t(\Delta \Upsilon_t)^- ] + {\cal H}_T \mathfrak{B}\right] =  w
        \end{equation*}
          Then, there exists a portfolio process $\pi$ such that $(c,\pi,\Upsilon)\in{\cal A}(j,w)$. The corresponding wealth $W^{c,\pi,\Upsilon}$ with $W_0^{c,\pi,\Upsilon}=w$ and $W_T^{c,\pi,\Upsilon}=\mathfrak{B}$  satisfies
		\begin{equation}
		\label{eq:static_budget-cost-t}
		{\cal H}_tW_t^{c,\pi,\Upsilon}=\mathbb{E}_t\left[\int_t^\infty{\cal H}_s \left(c_s - \epsilon_0 (1-\Upsilon_s) - \epsilon_1 \Upsilon_s\right)ds +\sum_{t \leq s \le T}[\z_0{\cal H}_s(\Delta \Upsilon_s)^+ + \z_1{\cal H}_s(\Delta \Upsilon_s)^- ] + {\cal H}_T \mathfrak{B}\right]
		\end{equation}
  for any $t\in[0,T]$ a.s. in $\Omega$.
	\end{itemize}
\end{pro}
\begin{proof}
	See Proposition 3.1 in \citet{JeonShim}.
\end{proof}

From the static budget constraint \eqref{eq:static_budget-cost} for Problem \ref{pr:main-cost}, for any $(c,\pi,\Upsilon)\in{\cal A}(j,w)$, we deduce that 
\begin{footnotesize}
\begin{equation}\label{eq:Lagrangian-1}
    \begin{aligned}
    & \mathbb{E}\left[\int_0^T e^{-\b t}\left(\left(U_1(t,c_t)-L_0\right)(1-\Upsilon_t)+\left(U_1(t,c_t)-L_1 \right)\Upsilon_t \right)dt + e^{-\b T}U_2(T,W_T^{c,\pi,\Upsilon})\right]\\
    \le& \mathbb{E}\left[\int_0^T e^{-\b t}\left(\left(U_1(t,c_t)-L_0\right)(1-\Upsilon_t)+\left(U_1(t,c_t)-L_1 \right)\Upsilon_t \right)dt + e^{-\b T}U_2(T,W_T^{c,\pi,\Upsilon})\right]\\
    +&\l \left(w-\mathbb{E}\left[\int_0^T{\cal H}_t \left(c_t - \epsilon_0 (1-\Upsilon_t) - \epsilon_1 \Upsilon_t\right)dt +\sum_{0 \leq t \le T}[\z_0{\cal H}_t(\Delta \Upsilon_t)^+ + \z_1{\cal H}_t(\Delta \Upsilon_t)^- ]+ {\cal H}_T W_T^{c,\pi,\Upsilon} \right]\right)\\
    \le &\mathbb{E}\left[\int_0^T e^{-\b t}\left(\left(U_1(t,c_t)-{\cal Y}_t^\l c_t +{\cal Y}_t^\l \e_0-L_0\right)(1-\Upsilon_t)+\left(U_1(t,c_t)-{\cal Y}_t^\l c_t +{\cal Y}_t^\l \e_1-L_1 \right)\Upsilon_t \right)dt\right]\\
    +&\mathbb{E}\left[e^{-\b T}\left(U_2(T,W_T^{c,\pi,\Upsilon})-{\cal Y}_T^{\l}W_T^{c,\pi,\Upsilon}\right)\right]-\mathbb{E}\left[\sum_{0 \leq t \le T}e^{-\b t}[\z_0{\cal Y}_t^\l (\Delta \Upsilon_t)^+ + \z_1{\cal Y}_t^{\l}(\Delta \Upsilon_t)^- ]\right]+\l w,
    \end{aligned}
\end{equation}
\end{footnotesize}
where $\l>0$ is a Lagrangian multiplier for the static budget constraint \eqref{eq:static_budget-cost}, and ${\cal Y}_t^\l:= \l e^{\b t} {\cal H}_t$.

For each $i=1,2$, we define the convex conjugate function $\widetilde{U}_i$ of ${U}_i$ by 
\begin{equation}
    \widetilde{U}_i(t,y) = \sup_{y>0}\left(U_i(t,c)-yc\right).
\end{equation}
By the first-order condition, we deduce that for $i=1,2,$ 
\begin{equation*}
    \widetilde{U}_i(t,y) = U_i(t,{\cal I}_i(t,y))-{\cal I}_i(t,y)y,
\end{equation*}
where ${\cal I}_i(t,\cdot)$ are the inverse function of $\partial_c U_i(t,\cdot)$. 

Hence, it follows from \eqref{eq:Lagrangian-1} that 
\begin{equation}\label{eq:Lagrangian-2}
        \mathbb{E}\left[\int_0^T e^{-\b t}\left(\left(U_1(t,c_t)-L_0\right)(1-\Upsilon_t)+\left(U_1(t,c_t)-L_1 \right)\Upsilon_t \right)dt + e^{-\b T}U_2(T,W_T^{c,\pi,\Upsilon})\right]
   \le {\cal J}(j,\l;\Upsilon)+\l w,
\end{equation}
where ${\cal J}(j,\l;\Upsilon)$ is defined as 
\begin{align*}
        {\cal J}(j,\l;\Upsilon):=&\mathbb{E}\left[\int_0^T e^{-\b t}\left(\left(\widetilde{U}_1(t,{\cal Y}_t^\l) +{\cal Y}_t^\l \e_0-L_0\right)(1-\Upsilon_t)+\left(\widetilde{U}_1(t,{\cal Y}_t^\l) +{\cal Y}_t^\l \e_1-L_1 \right)\Upsilon_t \right)dt\right]\\ 
   +&\mathbb{E}\left[e^{-\b T}\widetilde{U}_2(T,{\cal Y}_T^\l)\right]-\mathbb{E}\left[\sum_{0 \leq t \le T}e^{-\b t}[\z_0{\cal Y}_t^\l (\Delta \Upsilon_t)^+ + \z_1{\cal Y}_t^{\l}(\Delta \Upsilon_t)^- ]\right],
\end{align*}
and the inequality in \eqref{eq:Lagrangian-2} holds when the consumption and terminal wealth at $T$ are given by 
\begin{equation}\label{eq:candidate-optimal}
    c_t ={\cal I}_1(t,{\cal Y}_t^\l)\;\;\mbox{and}\;\; W_T={\cal I}_2(T,{\cal Y}_T^\l).
\end{equation}

For $j=0,1$, let $\Phi_j$ be the set all $\mathbb{F}$-adapted, finite variation, c\`{a}gl\`{a}d process taking values in $\{0,1\}$ and starting at $j$. From \eqref{eq:Lagrangian-2}, we can easily deuce the following {\it weak duality}: for $j=0,1,$
\begin{equation}\label{eq:weak-duality}
    V(j,w)\le \inf_{\l>0}\left[\sup_{\Upsilon\in\Phi_j}{\cal J}(j,\l;\Upsilon)+\l w\right].
\end{equation}

We will show that the inequality in \eqref{eq:weak-duality} holds as the equality later. Moreover, the weak-duality allows us to define the following dual problem.
\begin{pr}
    For given $\l>0$ and $j\in\{0,1\}$, we consider the following maximization problem:
    \begin{equation}
        J(j,\l) : = \sup_{\Upsilon\in \Phi_j}{\cal J}(j,\l;\Upsilon).
    \end{equation}
    We call $J(j,\l)$ as the {\it dual value function}.
\end{pr}

Note that 
\begin{align*}
    &\mathbb{E}\left[\int_0^T e^{-\b t}\left(\left(\widetilde{U}_1(t,{\cal Y}_t^\l) +{\cal Y}_t^\l \e_0-L_0\right)(1-\Upsilon_t)+\left(\widetilde{U}_1(t,{\cal Y}_t^\l) +{\cal Y}_t^\l \e_1-L_1 \right)\Upsilon_t \right)dt\right]\\=&\mathbb{E}\left[\int_0^T e^{-\b t}\left(\left({\cal Y}_t^\l \e_0-L_0\right)(1-\Upsilon_t)+\left({\cal Y}_t^\l \e_1-L_1 \right)\Upsilon_t \right)dt\right]+\mathbb{E}\left[\int_0^Te^{-\b t}\widetilde{U}_1(t,{\cal Y}_t^\l)dt\right].
\end{align*}
Thus, we can decompose the dual problem into two sub-problems as follows:
\begin{equation}
    \begin{aligned}
        J(j,\l) =& \sup_{\Upsilon\in \Phi_j}{\cal J}_j(\l;\{\Upsilon_t\}_{t=0}^T)=J_R(\l)+J_S(j,\l),
    \end{aligned}
\end{equation}
where 
\begin{align}\label{eq:OSP}
    J_R(\l):=&\mathbb{E}\left[\int_0^Te^{-\b t}\widetilde{U}_1(t,{\cal Y}_t^\l)dt +e^{-\b T}\widetilde{U}_2(T,{\cal Y}_T^\l)\right],\\ 
   J_S(j,\l):=& \sup_{\Upsilon\in \Phi_j}\left\{ \mathbb{E}\left[\int_0^T e^{-\b t}\left(\left({\cal Y}_t^\l \e_0-L_0\right)(1-\Upsilon_t)+\left({\cal Y}_t^\l \e_1-L_1 \right)\Upsilon_t \right)dt\right]\right.\\-&\left.\mathbb{E}\left[\sum_{0 \leq t \le T}e^{-\b t}[\z_0{\cal Y}_t^\l (\Delta \Upsilon_t)^+ + \z_1{\cal Y}_t^{\l}(\Delta \Upsilon_t)^- ]\right]\right\}. \nonumber
\end{align}
Here, $J_R(\lambda)$ represents the dual problem of the unconstrained problem, and $J_S(j, \lambda)$ is the finite-horizon pure optimal switching problem specifically related to the decision of job-switching.

Throughout the paper, we will use the following notations:
\begin{equation*}
    {\cal N}_T:=[0,T)\times(0,\infty)\;\;\mbox{and}\;\;\overline{{\cal N}}_T:=[0,T]\times(0,\infty).
\end{equation*}

Consider the following PDE arising from the unconstrained problem in $J_R(\l)$:
\begin{equation}\label{eq:PDE-unconstrained}
    \begin{cases}
        &\partial_t {\cal Q}_R(t,\l) + {\cal L}{\cal Q}_R(t,\l)  + \widetilde{U}_1(t,\l)=0\;\;\mbox{in}\;\;{\cal N}_T;\vspace{2mm}\\
        &{\cal Q}_R(T,\l) = \widetilde{U}_2(T,\l)\;\;\mbox{for all}\;\;\l>0,
    \end{cases}
\end{equation}
where the differential operator ${\cal L}$ is defined as
\begin{equation*}
    {\cal L}:=\dfrac{\t^2}{2}\l^2\partial_{\l\l} + (\beta -r )\l\partial_{\l}- \beta.
\end{equation*}
Then, we can state the following proposition.
\begin{pro}\label{pro:unconstrained} The following statements are true. 
\begin{itemize}
    \item[(a)] The PDE \eqref{eq:PDE-unconstrained} has a unique solution $ {{\cal Q}}_R\in C^\infty(\overline{{\cal N}}_T)$.
    \item[(b)] For any $(t,\l)\in \overline{{\cal N}}_T$, 
    $$
    \partial_{\l} {\cal Q}_R <0\;\;\mbox{and}\;\;\partial_{\l\l}{\cal Q}_R > 0.
    $$
    \item[(c)] There exist positive constant $C_R$ and $K_R$ such that for all $(t,\l)\in \overline{\cal N}_T$ 
    $$
    \left|\partial_{\l}{\cal Q}_R(t,\l)\right| \le C_R (1+ \l^{-K_R}).
    $$
    Moreover, 
    $$
    \lim_{\l \to 0+}\partial_{\l}{\cal Q}_R(t,\l) =-\infty\;\;\mbox{and}\;\;\lim_{\l \to \infty} \partial_{\l}{\cal Q}_R(t,\l) =0.
    $$
\end{itemize}    
\end{pro}
\begin{proof}
    See the proof of Theorem 2 in \citet{YK}.
\end{proof}
Utilizing the properties of ${\cal Q}_R$ in Proposition \ref{pro:unconstrained}, we can easily show that 
\begin{equation}
    J_R(\l) = {\cal Q}_R(0,\l)\;\;\mbox{for all}\;\;\l>0.
\end{equation}

Moreover, since the dual process ${\cal Y}_t^\l$ follows the geometric Brownian motion, it follows that 
\begin{equation}
    \begin{aligned}
        J_R(y)=&\mathbb{E}\left[\int_0^Te^{-\b t}\widetilde{U}_1(t,{\cal Y}_t^\l)dt +e^{-\b T}\widetilde{U}_2(T,{\cal Y}_T^\l)\right]\\
        =&\int_0^T e^{-\b t} \int_0^\infty \frac{\widetilde{U}_1(t,\xi)}{\sqrt{2\pi}\xi}\exp\left\{-\dfrac{(\ln{\xi}-\ln{y}-(\b-r-\frac{1}{2}\t^2)t)^2}{2\t^2 t}\right\}d\xi dt  \\
        +& e^{-\b T}\int_0^\infty \frac{\widetilde{U}_2(T,\xi)}{\sqrt{2\pi}\xi}\exp\left\{-\dfrac{(\ln{\xi}-\ln{y}-(\b-r-\frac{1}{2}\t^2)T)^2}{2\t^2 T}\right\}d\xi.\quad 
    \end{aligned}
\end{equation}

Now, we will investigate the optimal switching problem in $J_S(j,\l)$ for $j=0,1.$  By a standard theory of impulse control (see \citet{BL}), we consider the following Hamilton-Jacobi-Bellman(HJB) equation which takes the form of the system of parabolic VIs: 
\begin{equation}
    \label{eq:HJB-1}
		\max\{\partial_t {\cal Q}_0+\mathcal{L}\mathcal{Q}_0+\epsilon_0 \l-L_0, \mathcal{Q}_1-\mathcal{Q}_0-\zeta_0 \l \}=0\;\;\mbox{in}\;\; {\cal N}_T,
\end{equation}
and
\begin{equation}
    \label{eq:HJB-2}
		\max\{\partial_t {\cal Q}_1+\mathcal{L}\mathcal{Q}_1+\epsilon_1 \l-L_1, \mathcal{Q}_0-\mathcal{Q}_1-\zeta_1 \l \}=0\;\;\mbox{in}\;\; {\cal N}_T,
\end{equation}
with 
\begin{equation} \label{eq:HJB-initial-value}
    {\cal Q}_0(T,\l) ={\cal Q}_1(T,\l)=0,\;\;\forall\;\l\in(0,+\infty).
\end{equation}

We will prove the verification theorem (Theorem \ref{verification theorem}) in Section 5, which shows that 
\begin{equation*}
    J_S(j,\l) = {\cal Q}_j (0, \l),\;\;\forall\;j=0,1;\;\l\in(0,+\infty).
\end{equation*}
To achieve this, we first analyze the system of VIs \eqref{eq:HJB-1},  \eqref{eq:HJB-2}, and \eqref{eq:HJB-initial-value}. 

Let 
\begin{equation}
    {\cal P}(t,\l)= {\cal Q}_1(t,\l) - {\cal Q}_0(t,\l),\;\;\forall\;(t,\lambda)\in\overline{\cal N}_T.
\end{equation}

Then, we can observe that ${\cal P}$ satisfies the following double obstacle problem: 
\begin{equation}
    \label{eq:VI-P}\textbf{}
    \begin{cases}
        \partial_t {\cal P} + {\cal L}{\cal P} +(\e_1-\e_0)\l -(L_1-L_0)=0&\mbox{in}\;\;
        \{(t,\lambda)\in{\cal N}_T:-\zeta_1\lambda <{\cal P}(t,\lambda)< \zeta_0\lambda\}; \vspace{2mm}\\  
         \partial_t {\cal P} + {\cal L}{\cal P} +(\e_1-\e_0)\l -(L_1-L_0)\le 0&\mbox{in}\;\;
        \{(t,\lambda)\in{\cal N}_T:{\cal P}(t,\lambda)=-\zeta_1\lambda\};\vspace{2mm}\\  
         \partial_t {\cal P} + {\cal L}{\cal P} +(\e_1-\e_0)\l -(L_1-L_0)\ge 0&\mbox{in}\;\;
        \{(t,\lambda)\in{\cal N}_T:{\cal P}(t,\lambda)=\zeta_0\lambda\};\vspace{2mm}\\ 
        -\zeta_1\lambda\leq {\cal P}\leq \zeta_0\lambda\;\;\mbox{on}\;\;\overline{\cal N}_T;
        &{\cal P}(T,\l)=0,\;\;\l\in(0,+\infty).
    \end{cases}
\end{equation}

On the other hand, if we obtain the unique strong solution ${\cal P}$ of VI \eqref{eq:VI-P}, we can define $({\cal Q}_0,{\cal Q}_1)$ by the unique strong solution of the following PDE,  
\begin{equation}\label{de:Q0-and-Q1}
    \begin{cases}
    &-\left(\partial_t {\cal Q}_1 +{\cal L}{\cal Q}_1 +\e_1 \l -L_1\right)=\left[-\partial_t {\cal P}-{\cal L}{\cal P}-(\e_1-\e_0)\l +(L_1-L_0)\right]^+; \vspace{2mm}\\
    &-\left(\partial_t {\cal Q}_0 +{\cal L}{\cal Q}_0 +\e_0 \l -L_0\right)=\left[-\partial_t {\cal P}-{\cal L}{\cal P}-(\e_1-\e_0)\l +(L_1-L_0)\right]^-;\vspace{2mm}\\ 
    &{\cal Q}_0(T,\l) = {\cal Q}_1(T,\l)=0,\quad\l\in(0,+\infty),
    \end{cases}
\end{equation}
then ${\cal Q}_0$ and ${\cal Q}_1$ satisfy \eqref{eq:HJB-1} and  \eqref{eq:HJB-2}, respectively. 

\begin{theorem}\label{th-relationship-Q-and-P}
If ${\cal P}\in W^{1,2}_{p,loc}({\cal N}_T)\cap C(\overline{\cal N}_T)$ is the unique strong solution of VI \eqref{eq:VI-P} with any $p\geq1$, then ${\cal Q}_0$ and ${\cal Q}_1$ defined in \eqref{de:Q0-and-Q1} satisfy \eqref{eq:HJB-1} and  \eqref{eq:HJB-2}, respectively. Moreover, $0\leq {\cal Q}_0,{\cal Q}_1\leq C_T\l$ in $\overline{\cal N}_T$, where $C_T$ is a constant independent of $(t,\l)$. 
\end{theorem}
\begin{proof}
Since ${\cal P}\in W^{1,2}_{p,loc}({\cal N}_T)$ is the unique strong solution of VI \eqref{eq:VI-P}, we know that 
\begin{eqnarray*}
    &&f_1:=\left[-\partial_t {\cal P}-{\cal L}{\cal P}-(\e_1-\e_0)\l +(L_1-L_0)\right]^+
    =\left[(L_1-L_0)-(\e_1-\e_0+r\zeta_1)\l\right] {\rm\bf I}_{\{{\cal P}=-\zeta_1\l\}}\in L^p_{loc}({\cal N}_T);
    \\[2mm]
    &&f_0:=\left[-\partial_t {\cal P}-{\cal L}{\cal P}-(\e_1-\e_0)\l +(L_1-L_0)\right]^-
    =\left[(\e_1-\e_0-r\zeta_0)\l-(L_1-L_0)\right] {\rm\bf I}_{\{{\cal P}=\zeta_0\l\}}\in L^p_{loc}({\cal N}_T)
\end{eqnarray*}
for any $p\geq1$. So, the theory for PDEs  in \citet{Li96}  implies that  \eqref{de:Q0-and-Q1} has a unique strong solution $({\cal Q}_0,{\cal Q}_1)$ such that ${\cal Q}_0,{\cal Q}_1\in  W^{1,2}_{p,loc}({\cal N}_T)\cap C(\overline{\cal N}_T)$ with any $p\geq1$.  

Denote that $\overline{{\cal Q}}=e^{\max(k,0)T-kt}\l\geq0$ with $k=\e_1-r$. It is not difficult to check that $f_1,f_0\geq0$ and 
\begin{eqnarray*}
&&-\left(\partial_t \overline{{\cal Q}} +{\cal L}\overline{{\cal Q}}+\e_1 \l -L_1\right)
=\left(k+r-\e_1\right)e^{\max(k,0)T-kt}\l+ L_1\geq f_1,
\\[2mm] 
&&-\left(\partial_t \overline{{\cal Q}} +{\cal L}\overline{{\cal Q}}+\e_0 \l -L_0\right)
=\left(k+r-\e_0\right)e^{\max(k,0)T-kt}\l+ L_0\geq f_0,
\end{eqnarray*}
where we have used the fact that $\e_1>\e_0$ and $r\zeta_i>0,i=0,1$. So, the comparison principle for PDEs implies that 
$$
0\leq {\cal Q}_0,\;{\cal Q}_1\leq e^{\max(\e_1-r,0)T-(\e_1-r)t}\l\;\;\mbox{in}\;\;\overline{\cal N}_T.
$$

From \eqref{de:Q0-and-Q1}, we know that ${\cal Q}_1-{\cal Q}_0$ satisfies PDE
\begin{equation*}
    \begin{cases}
    &-\left[\partial_t ({\cal Q}_1-{\cal Q}_0) +{\cal L}({\cal Q}_1-{\cal Q}_0) +(\e_1-\e_0)\l -(L_1-L_0)\right]=-\partial_t {\cal P}-{\cal L}{\cal P}-(\e_1-\e_0)\l +(L_1-L_0);\vspace{2mm}\\ 
    &{\cal Q}_1(T,\l)- {\cal Q}_0(T,\l)=0= {\cal P}(T,\l),\quad\l\in(0,+\infty).
    \end{cases}
\end{equation*}
Hence, the uniqueness of the strong solution of the above PDE implies that ${\cal Q}_1-{\cal Q}_0={\cal P}$ satisfies VI \eqref{eq:VI-P}, and $-\zeta_1\l\leq {\cal Q}_1-{\cal Q}_0\le \zeta_0\l$, and 
\begin{equation*}
    \begin{cases}
    \partial_t {\cal Q}_1 +{\cal L}{\cal Q}_1 +\e_1 \l -L_1=-f_1\le0,\;\;
    \partial_t {\cal Q}_0 +{\cal L}{\cal Q}_0 +\e_0 \l -L_0=-f_0=0&\mbox{if}\;{\cal P}(t,\l)=-\zeta_1\l;\vspace{2mm}\\ 
    \partial_t {\cal Q}_1 +{\cal L}{\cal Q}_1 +\e_1 \l -L_1=-f_1=0,\;\;
    \partial_t {\cal Q}_0 +{\cal L}{\cal Q}_0 +\e_0 \l -L_0=-f_0=0&\mbox{if}\;-\zeta_1\l<{\cal P}(t,\l)<\zeta_0\l;\vspace{2mm}\\     
    \partial_t {\cal Q}_1 +{\cal L}{\cal Q}_1 +\e_1 \l -L_1=-f_1=0,\;\;
    \partial_t {\cal Q}_0 +{\cal L}{\cal Q}_0 +\e_0 \l -L_0=-f_0\le0&\mbox{if}\;{\cal P}(t,\l)=\zeta_0\l.
    \end{cases}
\end{equation*}
So, ${\cal Q}_0$ and ${\cal Q}_1$ defined in \eqref{de:Q0-and-Q1} satisfy \eqref{eq:HJB-1} and \eqref{eq:HJB-2}, respectively.
\end{proof}

\section{Double Obstacle Problem}\label{sec:double-obstacle}

In this section, we focus on the analysis of the double obstacle problem \eqref{eq:VI-P} derived in the previous section. We investigate aspects such as the uniqueness, existence, and regularity of the solution to the double obstacle problem and its corresponding two free boundaries.

Let 
\begin{equation}\label{tran-P-u}
u(\tau,x)={\cal P}(t,\lambda)/\lambda,\quad 
x=\ln\lambda,\quad \tau=T-t,
\end{equation}
then $u$ satisfies the following VI
\begin{equation}
    \label{eq:VI-u}
    \begin{cases}
        \partial_\tau u - {\cal L}_xu -(\e_1-\e_0)+(L_1-L_0)e^{-x}=0\;\;\;&\mbox{in}\;\;
        \{(\tau,x)\in{\cal U}_T:-\zeta_1 <u(\tau,x)< \zeta_0\}; \vspace{2mm}\\
        \partial_\tau u - {\cal L}_xu -(\e_1-\e_0)+(L_1-L_0)e^{-x}\ge 0\;\;\;&\mbox{in}\;\;
        \{(\tau,x)\in{\cal U}_T:u(\tau,x)=-\zeta_1\}; \vspace{2mm}\\
        \partial_\tau u - {\cal L}_xu -(\e_1-\e_0)+(L_1-L_0)e^{-x}\le 0&\mbox{in}\;\;
        \{(\tau,x)\in{\cal U}_T:u(\tau,x)=\zeta_0\}; \vspace{2mm}\\
        -\zeta_1\le u \le \zeta_0\;\;\mbox{in}\;\;\overline{\cal U}_T; &u(0,x)=0,\;\;x\in  \mathbb{R} 
    \end{cases}
\end{equation}
with 
\begin{equation}
    \label{differential operator for widetilde L}
		 {\cal L}_x:=\dfrac{\t^2}{2}\partial_{xx}+ \left(\beta -r+{\frac{\t^2}{2}}\right)\partial_{x}- r,\qquad 
   {\cal U}_T:=(0,T]\times \mathbb{R},\quad \overline{\cal U}_T:=[0,T]\times \mathbb{R}.
\end{equation}

\subsection{The Existence and Uniqueness of the Strong Solution}

In this section, we establish the existence and uniqueness of the strong solution of VI \eqref{eq:VI-u}. We firstly consider the bounded problems \eqref{eq:VI-un}, which converge to VI \eqref{eq:VI-u}.
\begin{equation}
    \label{eq:VI-un}
    \begin{cases}
        \partial_\tau u_n - {\cal L}_xu_n-
        (\e_1-\e_0)+(L_1-L_0)e^{-x}=0\;\;\;&\mbox{in}\;\;
        \{(\tau,x)\in{\cal U}_T^n:-\zeta_1 <u_n(\tau,x)< \zeta_0\}; \vspace{2mm}\\
        \partial_\tau u_n - {\cal L}_xu_n -(\e_1-\e_0)+(L_1-L_0)e^{-x}\ge 0\;\;\;&\mbox{in}\;\;
        \{(\tau,x)\in{\cal U}_T^n:u_n(\tau,x)=-\zeta_1\}; \vspace{2mm}\\
        \partial_\tau u_n - {\cal L}_xu_n -(\e_1-\e_0)+(L_1-L_0)e^{-x}\le 0&\mbox{in}\;\;
        \{(\tau,x)\in{\cal U}_T^n:u_n(\tau,x)=\zeta_0\}; \vspace{2mm}\\
        -\zeta_1\le u_n \le \zeta_0\;\;\mbox{in}\;\;\overline{\cal U}^n_T;\qquad 
        &u_n(0,x)=0,\,x\in[-n,n]; \vspace{2mm}\\
        u_n(\tau,n)=\varphi_+(\tau),\qquad 
        u_n(\tau,-n)=\varphi_{-,n}(\tau),&\tau\in[0,T],
    \end{cases}
\end{equation}
where $n\in \mathbb{Z}_+$, and ${\cal U}_T^n:=(0,T]\times (-n,n),\,\overline{\cal U}_T^n:=[0,T]\times [-n,n]$, and 
\begin{equation}    \label{de:varphi+}
   \varphi_+(\tau)
   :=\frac{\e_1-\e_0}{r}\left(1-e^{-r\tau}\right)
   {\rm\bf I}_{\{\tau\leq T_1\}}
   +\zeta_0{\rm\bf I}_{\{\tau>T_1\}},\quad
   \varphi_{-,n}(\tau):=-L_1e^{n}\tau{\rm\bf I}_{\{\tau\leq \zeta_1e^{-n}/L_1\}}-\zeta_1{\rm\bf I}_{\{\tau>\zeta_1e^{-n}/L_1\}}
\end{equation}
for any $\tau\in[0,T]$, here $T_1$ is defined in \eqref{de:T1}.

In order to achieve the existence and uniqueness of the strong solution of VI \eqref{eq:VI-un}, we use the penalty problems \eqref{eq:PDE-unep} to converge to the bounded problems \eqref{eq:VI-un},
\begin{equation}
    \label{eq:PDE-unep}
    \begin{cases}
        \partial_\tau u_{\varepsilon,n} - {\cal L}_xu_{\varepsilon,n}-(\e_1-\e_0)+(L_1-L_0)e^{-x}-\beta_{1,\varepsilon}(u_{\varepsilon,n}+\zeta_1)
        -\beta_{2,\varepsilon}(u_{\varepsilon,n}-\zeta_0)=0
        \;\;&\mbox{in}\;\;
        {\cal U}_T^n; \vspace{2mm}\\
        u_{\varepsilon,n}(0,x)=0,\;x\in  \mathbb{R};\qquad\qquad\qquad
        u_{\varepsilon,n}(\tau,n)=\varphi_+(\tau),\quad
        u_{\varepsilon,n}(\tau,-n)=\varphi_{-,n}(\tau),\;&\tau\in[0,T],
    \end{cases}
\end{equation}
where the penalty functions $\beta_{1,\varepsilon},\beta_{2,\varepsilon}\in C^2(\mathbb{R})$ are defined as
\begin{equation*}
\begin{aligned}
	&\beta_{1,\varepsilon}\geq0;\\
	&\beta_{1,\varepsilon}(\xi)=0,\quad \text{if }\xi\geq \varepsilon;\\
	&\beta_{1,\varepsilon}(\xi)=L_1e^n\left(1-{2\frac{\xi}{\varepsilon}}\right),\;\;\text{if }\xi\leq0;\\
	&\beta'_{1,\varepsilon}\leq0;\quad \beta^{\prime\prime}_{1,\varepsilon}\geq0,
\end{aligned}\qquad\qquad
\begin{aligned}
	&\beta_{2,\varepsilon}\leq0;\\
	&\beta_{2,\varepsilon}(\xi)=0,\quad \text{if }\xi\leq -\varepsilon;\\
	&\beta_{2,\varepsilon}(\xi)=(\e_0-\e_1)\left(1+{2\frac{\xi}{\varepsilon}}\right),\;\;\text{if }\xi\geq0;\\
	&\beta'_{2,\varepsilon}\leq0;\quad \beta^{\prime\prime}_{2,\varepsilon}\leq0.
\end{aligned}
\end{equation*}

In the first, we show some estimates about $ u_{\varepsilon,n}$ via the following lemma.

\begin{lem}\label{lemma for penalty problem}
For any $0<\varepsilon<\min(\zeta_0,\zeta_1),\,n\in \mathbb{Z}_+$, the unique classical solution $u_{\varepsilon,n}\in C(\overline{\cal U}_T^n)\cap C^{1,2}({\cal U}_T^n)$ of PDE \eqref{eq:PDE-unep} has the following estimates 
$$
\varphi_{-,n}\leq u_{\varepsilon,n}\leq \varphi_+,\;\; 
\partial_x u_{\varepsilon,n}\geq0,\;\; 
-L_1e^{-\beta(\tau-\zeta_1e^{-n}/L_1)-x}
\leq \partial_\tau u_{\varepsilon,n}
\leq (\e_1-\e_0)e^{-r\tau}\;\;\mbox{in}\;\;\overline{\cal U}_T^n.
$$
Moreover, if $n$ is large enough, then we have the following estimates,
\begin{equation}\label{cone estimation}    
 -C\partial_x u_{\varepsilon,n}\leq \tau\partial_\tau u_{\varepsilon,n}\leq 
 Ce^{x}\partial_x u_{\varepsilon,n}\;\;\mbox{in}\;\;\overline{\cal U}_T^{n/2},
\end{equation}
where $C$ is a constant independent of $\tau,x,n$ and $\varepsilon$.
\end{lem}

\begin{proof} From the theory for PDE in \citet{Li96}, we know that PDE \eqref{eq:PDE-unep} has a unique classical solution $u_{\varepsilon,n}\in C(\overline{\cal U}_T^n)\cap C^{1,2}({\cal U}_T^n)$. We firstly compute that 
\begin{eqnarray*}
&&\partial_\tau \varphi_{-,n} - {\cal L}_x\varphi_{-,n}
-(\e_1-\e_0)+(L_1-L_0)e^{-x}
-\beta_{1,\varepsilon}(\varphi_{-,n}+\zeta_1)
-\beta_{2,\varepsilon}(\varphi_{-,n}-\zeta_0)
\\[2mm]
&=&\left[\varphi^\prime_{-,n}+r\varphi_{-,n}
-\beta_{1,\varepsilon}(\varphi_{-,n}+\zeta_1)\right]
-\left[(\e_1-\e_0)-(L_1-L_0)e^{-x}\right]
\\[2mm]
&\leq&-L_1e^{n}(1+r\tau){\rm\bf I}_{\{\tau\leq \zeta_1e^{-n}/L_1\}}
-\left[r\zeta_1+\beta_{1,\varepsilon}(0)\right]{\rm\bf I}_{\{\tau>\zeta_1e^{-n}/L_1\}}
+(L_1-L_0)e^{n}<0\;\;\mbox{in}\;\;{\cal U}_T^n.
\end{eqnarray*}
Moreover, it is clear that $\varphi_{-,n}\in W^{1,2}_{p,loc}({\cal U}^n_T)\cap C(\overline{\cal U}_T^n)$ for any $p\geq1$, and 
$$
 \varphi_{-,n}(\tau)=u_{\varepsilon,n}(\tau,-n)\leq 0=u_{\varepsilon,n}(0,x)\leq\varphi_{+}(\tau)
 =u_{\varepsilon,n}(\tau,n),\;\forall\;\tau\in[0,T],\;x\in[-n,n].
$$
Applying the comparison theorem for PDEs to PDEs of $\varphi_{-,n}$ and $u_{\varepsilon,n}$ in the domain ${\cal U}_T^n$, we know that
$$
 u_{\varepsilon,n}\geq \varphi_{-,n}\;\;\mbox{in}\;\;
 \overline{\cal U}_T^n.
$$

On the other hand, we can check that 
\begin{eqnarray*}
&&\partial_\tau \varphi_{+} - {\cal L}_x\varphi_{+}
-(\e_1-\e_0)+(L_1-L_0)e^{-x}
-\beta_{1,\varepsilon}(\varphi_{+}+\zeta_1)
-\beta_{2,\varepsilon}(\varphi_{+}-\zeta_0)
\\[2mm]
&=&\left[\varphi^\prime_++r\varphi_+
-\beta_{2,\varepsilon}(\varphi_{+}-\zeta_0)\right]
-\left[(\e_1-\e_0)-(L_1-L_0)e^{-x}\right]
\\[2mm]
&>&(\e_1-\e_0){\rm\bf I}_{\{\tau\leq T_1\}}
+\left[r\zeta_0-\beta_{2,\varepsilon}(0)\right]{\rm\bf I}_{\{\tau>T_1\}}
-(\e_1-\e_0)\geq0\;\;\mbox{in}\;\;{\cal U}_T^n.
\end{eqnarray*}
Moreover, it is clear that $\varphi_+\in W^{1,2}_{p,loc}({\cal U}^n_T)\cap C(\overline{\cal U}_T^n)$ for any $p\geq1$, and 
$$
 \varphi_{+}(\tau)=u_{\varepsilon,n}(\tau,n)\geq 0=u_{\varepsilon,n}(0,x)\geq \varphi_{-,n}(\tau)
 =u_{\varepsilon,n}(\tau,-n),\;\forall\;\tau\in[0,T],\;x\in[-n,n].
$$
Applying the comparison theorem for PDEs to PDEs of $\varphi_+$ and $u_{\varepsilon,n}$ in the domain ${\cal U}_T^n$, we know that
$$
 u_{\varepsilon,n}\leq \varphi_+\;\;\mbox{in}\;\;
 \overline{\cal U}_T^n.
$$

Since 
$$
 u_{\varepsilon,n}(\tau,-n)=\varphi_{-,n}(\tau)\leq u_{\varepsilon,n}(\tau,x)\leq \varphi_+(\tau)=u_{\varepsilon,n}(\tau,n),\;\;
 u_{\varepsilon,n}(0,x)=0,\;\;
 \forall\;\tau\in[0,T],\;x\in[-n,n],
$$
we deduce that 
$$
 \partial_x u_{\varepsilon,n}(\tau,-n)\geq0,\;\;
 \partial_x u_{\varepsilon,n}(\tau,n)\geq0,\;\;
 \partial_x u_{\varepsilon,n}(0,x)=0,\;\;
 \forall\;\tau\in[0,T],\;x\in[-n,n].
$$
Moreover, differentiating PDE \eqref{eq:PDE-unep} with respect to $x$, we know that $\partial_x u_{\varepsilon,n}$ satisfies
\begin{equation}\label{eq:PDE-partial-x-unep}
 \partial_\tau (\partial_x u_{\varepsilon,n})- {\cal L}_x (\partial_x u_{\varepsilon,n})
-\beta^\prime_{1,\varepsilon} (u_{\varepsilon,n}+\zeta_1)(\partial_x u_{\varepsilon,n})
-\beta^\prime_{2,\varepsilon}(u_{\varepsilon,n}-\zeta_0)(\partial_x u_{\varepsilon,n})=(L_1-L_0)e^{-x}
\geq0.
\end{equation}
So, the comparison theorem for PDEs implies that 
$$
 \partial_x u_{\varepsilon,n}\geq0\;\;\mbox{in}
 \;\;\overline{\cal U}_T^n.
$$

Differentiating PDE \eqref{eq:PDE-unep} with respect to $\tau$, we deduce that $\partial_\tau u_{\varepsilon,n}$ satisfies
\begin{equation}\label{eq:PDE-partial-tau-unep}
    \begin{cases}
    \partial_\tau (\partial_\tau u_{\varepsilon,n})- {\cal L}_x (\partial_\tau u_{\varepsilon,n})
    -\beta^\prime_{1,\varepsilon} (u_{\varepsilon,n}+\zeta_1)(\partial_\tau u_{\varepsilon,n})
    -\beta^\prime_{2,\varepsilon}(u_{\varepsilon,n}-\zeta_0)(\partial_\tau u_{\varepsilon,n})=0; 
    \vspace{2mm}\\
    \partial_\tau u_{\varepsilon,n}(0,x)={\cal L}_x0+(\e_1-\e_0)-(L_1-L_0)e^{-x}+\beta_{1,\varepsilon}(\zeta_1)+\beta_{2,\varepsilon}(-\zeta_0)
    =(\e_1-\e_0)-(L_1-L_0)e^{-x};
    \vspace{2mm}\\
    \partial_\tau u_{\varepsilon,n}(\tau,-n)
    =\varphi^\prime_{-,n}(\tau)
    =-L_1e^{n}{\rm\bf I}_{\{\tau\leq \zeta_1e^{-n}/L_1\}},\quad
    \partial_\tau u_{\varepsilon,n}(\tau,n)
    =\varphi^\prime_+(\tau)
    =(\e_1-\e_0)e^{-r\tau}{\rm\bf I}_{\{\tau\leq T_1\}}.
    \end{cases}
\end{equation}

Temporarily denote 
$$ 
 \underline{v}=-L_1e^{-\beta(\tau-\zeta_1e^{-n}/L_1)-x}<0,
$$
then it is not difficult to check that 
\begin{equation*}
    \begin{cases}
    \partial_\tau \underline{v}- {\cal L}_x \underline{v}
    -\beta^\prime_{1,\varepsilon} (u_{\varepsilon,n}+\zeta_1)\underline{v}
    -\beta^\prime_{2,\varepsilon}(u_{\varepsilon,n}-\zeta_0)\underline{v}
    \leq \partial_\tau \underline{v}- {\cal L}_x \underline{v}=0; 
    \vspace{2mm}\\
    \underline{v}(0,x)\leq \partial_\tau u_{\varepsilon,n}(0,x);\qquad 
    \underline{v}(\tau,n)
    <0\leq \partial_\tau u_{\varepsilon,n}(\tau,n);
    \qquad 
    \underline{v}(\tau,-n)\le
    \partial_\tau u_{\varepsilon,n}(\tau,-n).
    \end{cases}
\end{equation*}
So, the comparison theorem for PDEs implies that 
\begin{equation}\label{lower bound-partial-tau-unep}
 \partial_\tau u_{\varepsilon,n}\geq\underline{v}
 =-L_1e^{-\beta(\tau-\zeta_1e^{-n}/L_1)-x}\;\;\mbox{in}
 \;\;\overline{\cal U}_T^n.
\end{equation}

On the other hand, we temporarily denote 
$$ 
 \overline{v}=(\e_1-\e_0)e^{-r\tau}>0,
$$
then it is not difficult to check that 
\begin{equation*}
    \begin{cases}
    \partial_\tau \overline{v}- {\cal L}_x \overline{v}
    -\beta^\prime_{1,\varepsilon} (u_{\varepsilon,n}+\zeta_1)\overline{v}
    -\beta^\prime_{2,\varepsilon}(u_{\varepsilon,n}-\zeta_0)\overline{v}
    \leq \partial_\tau \overline{v}- {\cal L}_x \overline{v}=0; 
    \vspace{2mm}\\
    \overline{v}(0,x)>\partial_\tau u_{\varepsilon,n}(0,x);\qquad 
    \overline{v}(\tau,n)\geq \partial_\tau u_{\varepsilon,n}(\tau,n);
    \qquad 
    \overline{v}(\tau,-n)> 0\geq 
    \partial_\tau u_{\varepsilon,n}(\tau,-n).
    \end{cases}
\end{equation*}
So, the comparison theorem for PDEs implies that 
\begin{equation}\label{upper bound-partial-tau-unep}
 \partial_\tau u_{\varepsilon,n}\leq\overline{v}
 =(\e_1-\e_0)e^{-r\tau}\;\;\mbox{in}
 \;\;\overline{\cal U}_T^n.
\end{equation}

From PDE \eqref{eq:PDE-partial-tau-unep}, we know that
\begin{equation}\label{eq:PDE-tau-partial-tau-unep}
    \begin{cases}
    \partial_\tau (\tau\partial_\tau u_{\varepsilon,n})- {\cal L}_x (\tau\partial_\tau u_{\varepsilon,n})
    -\beta^\prime_{1,\varepsilon} (u_{\varepsilon,n}+\zeta_1)(\tau\partial_\tau u_{\varepsilon,n})
    -\beta^\prime_{2,\varepsilon}(u_{\varepsilon,n}-\zeta_0)(\tau\partial_\tau u_{\varepsilon,n})
    =\partial_\tau u_{\varepsilon,n}; 
    \vspace{2mm}\\
    (\tau\partial_\tau u_{\varepsilon,n})(0,x)=0;\;
    (\tau\partial_\tau u_{\varepsilon,n})(\tau,n)
    =(\e_1-\e_0)e^{-r\tau}\tau{\rm\bf I}_{\{\tau\leq T_1\}};\;
    \partial_\tau u_{\varepsilon,n}(\tau,-n)
    =-L_1e^{n}\tau{\rm\bf I}_{\{\tau\leq \zeta_1e^{-n}/L_1\}}.
    \end{cases}
\end{equation}

In order to achieve \eqref{cone estimation} , we construct an auxiliary function $\psi$ as follows
\begin{equation}\label{def: psi}
  \psi(\tau,x;x_0)=e^{k\tau}(e^{2|x-x_0|}-2|x-x_0|-1)\geq0\;\;
  \mbox{with}\;\; k=4\beta+3r+6\t^2+2.
\end{equation}
It is easy to check that $\psi(\cdot;x_0)\in W^{1,2}_{p,loc}({\cal U}^n_T)\cap C(\overline{\cal U}_T^n)$ for any $p\geq1$, and 
\begin{eqnarray}\nonumber 
 &&\partial_\tau\psi-\mathcal{L}_x\psi
 -\beta^\prime_{1,\varepsilon}(u_{\varepsilon,n}+\zeta_1)\psi
 -\beta^\prime_{2,\varepsilon}(u_{\varepsilon,n}-\zeta_0)\psi \geq\partial_\tau\psi-\mathcal{L}_x\psi
 \\[2mm]\nonumber 
 &=&e^{k\tau}
 \left[(k+r)(e^{2|x-x_0|}-2|x-x_0|-1)-2\theta^2
 e^{2|x_0-x|}-\left(2\beta-2r+\theta^2\right){\rm sign}(x-x_0)\left(e^{2|x_0-x|}-1\right)\right]
 \\[2mm]\nonumber
 &\geq&e^{k\tau}
 \left[\frac{k+r}{2}\left(e^{2|x-x_0|}-4|x-x_0|-2\right) +\left(\frac{k+r}{2}-2\beta-2r-3\theta^2\right)e^{2|x-x_0|}\right]
 \\[2mm]\label{eq:PDE-psi}
 &\geq& e^{k\tau}\left[-(k+r){\rm\bf I}_{\{|x-x_0|\leq 1\}}+e^{2|x-x_0|}\right].
\end{eqnarray}

Temporally denote
$$
 v=\tau\partial_\tau u_{\varepsilon,n}
 +K\partial_x u_{\varepsilon,n}
 +e^{-\max(x_0,0)}\psi(\cdot;x_0),\quad K=\frac{L_1+k+r}{L_1-L_0}e^{1+kT},
$$
then combining \eqref{eq:PDE-tau-partial-tau-unep}, \eqref{eq:PDE-partial-x-unep}, \eqref{eq:PDE-psi} and \eqref{lower bound-partial-tau-unep}, we know that $v$ satisfies
\begin{equation*}
\left\{\begin{aligned}	
  &\partial_\tau v
  -\mathcal{L}_xv
  -\beta^\prime_{1,\varepsilon}(u_{\varepsilon,n}+\zeta_1)v
  -\beta^\prime_{2,\varepsilon}(u_{\varepsilon,n}-\zeta_0)v  
  \\[2mm]&\qquad \qquad 
  \geq \partial_\tau u_{\varepsilon,n}
  +K(L_1-L_0)e^{-x}-(k+r)e^{k\tau-\max(x_0,0)}{\rm\bf I}_{\{|x-x_0|\leq 1\}}
  \\[2mm]&\qquad\qquad  
  \geq -L_1e^{-\beta(\tau-\zeta_1e^{-n}/L_1)-x}+K(L_1-L_0)e^{-x}-(k+r)e^{k\tau-\max(x_0,0)}{\rm\bf I}_{\{|x-x_0|\leq 1\}}\geq0;
  \\[2mm]
  &v(0, x)\geq0;\qquad v(\tau,n)\geq0;
  \\[2mm]
  &v(\tau,-n)\geq -L_1 e^{n}\tau{\rm\bf I}_{\{\tau\leq \zeta_1e^{-n}/L_1\}}
  +e^{-\max(x_0,0)}\left(e^{2|n+x_0|}-2|n+x_0|-1\right)
  \geq-\zeta_1+e^{-n/2}\left(e^{n}-n-1\right)
  \geq0
\end{aligned}\right.
\end{equation*}
provided that $x_0\in[-n/2,n/2]$ and $n$ is large enough.  So, the comparison principle for PDEs implies that $v\geq0$ in $\overline{\cal U}^n_T$ and
$$
 \tau\partial_\tau u_{\varepsilon,n}(\tau,x_0)\geq 
 -K\partial_x u_{\varepsilon,n}(\tau,x_0),\;\;\forall\;\;(\tau,x_0)\in 
 [0,T]\times[-n/2,n/2].
$$

On the other hand, temporally denote
$$
 v=\tau\partial_\tau u_{\varepsilon,n}-Ke^{x_0}\partial_x u_{\varepsilon,n}
 -(\e_1-\e_0)\psi,\quad K=\frac{(k+r)(\e_1-\e_0)}{L_1-L_0}e^{1+kT},
$$
then combining \eqref{eq:PDE-tau-partial-tau-unep}, \eqref{eq:PDE-partial-x-unep}, \eqref{eq:PDE-psi}  and \eqref{upper bound-partial-tau-unep}, we know that $w$ satisfies
\begin{equation*}
\left\{\begin{aligned}	
  &\partial_\tau v
  -\mathcal{L}_xv
  -\beta^\prime_{1,\varepsilon}(u_{\varepsilon,n}+\zeta_1)v
  -\beta^\prime_{2,\varepsilon}(u_{\varepsilon,n}-\zeta_0)v
  \\[2mm]&\qquad 
  \leq (\e_1-\e_0)e^{-r\tau}-K(L_1-L_0)e^{x_0-x}
  -(\e_1-\e_0)e^{k\tau}\left[e^{2|x-x_0|}-(k+r){\rm\bf I}_{\{|x-x_0|\leq 1\}}\right]\leq0;
  \\[2mm]
  &v(0,x)\leq0;\qquad v(\tau,-n)\leq0;
  \\[2mm]
  &v(\tau,n)\leq (\e_1-\e_0)\tau e^{-r\tau}{\rm\bf I}_{\{\tau\leq T_1\}}
  -(\e_1-\e_0)\left(e^{2|n-x_0|}-2|n-x_0|-1\right)
  \leq0
\end{aligned}\right.
\end{equation*}
provided that $x_0\in[-n,n/2]$ and $n$ is large enough.  So, the comparison principle implies that $v\leq0$ in $\overline{\cal U}^n_T$ and
$$
 \tau\partial_\tau u_{\varepsilon,n}(\tau,x_0)\leq 
 Ke^{x_0}\partial_x u_{\varepsilon,n}(\tau,x_0),\;\;\forall\;\;(\tau,x_0)\in 
 [0,T]\times[-n,n/2].
$$
\end{proof}

Thanking to the estimates in Lemma \ref{lemma for penalty problem}, by the standard argument (see \citet{F1,Yi,JKY2023}), we have the following lemma and theorem.

\begin{lem}\label{lemma for VI in bounded domain}
For any $n\in \mathbb{Z}_+$, the unique strong solution $u_{n}\in W^{1,2}_p({\cal U}_T^n)\cap C(\overline{\cal U}_T^n)$ with any $p\geq1$ of VI \eqref{eq:VI-un} has the following estimates 
$$
\varphi_{-,n}\leq u_{n}\leq \varphi_+,\;\;
\partial_x u_{n}\geq0,\;\;
-L_1e^{-\beta(\tau-\zeta_1e^{-n}/L_1)-x}
\leq \partial_\tau u_{n}
\leq (\e_1-\e_0)e^{-r\tau}\;\;\mbox{in}\;\;\overline{\cal U}_T^n.
$$
Moreover, if $n$ is large enough, then we have the following estimates,
$$
 -C\partial_x u_{n}\leq \tau\partial_\tau u_n\leq 
 Ce^{x}\partial_x u_{n}\;\;\mbox{in}\;\;\overline{\cal U}_T^{n/2},
$$
where $C$ is a constant independent of $\tau,x$ and $n$.
\end{lem}

\begin{theorem}\label{th-u}
VI \eqref{eq:VI-u} has a unique strong solution $u\in W^{1,2}_{p,loc}({\cal U}_T)\cap C(\overline{\cal U}_T)$ with any $p\geq1$, which has the following estimates 
$$
-\zeta_1\leq u\leq \varphi_+,\; 
\partial_x u\geq0,\;
-L_1e^{-\beta\tau-x}
\leq \partial_\tau u
\leq (\e_1-\e_0)e^{-r\tau},\;
 -C\partial_x u\leq \tau \partial_\tau u\leq Ce^{x}\partial_x u\;\mbox{in}\;\overline{\cal U}_T,
$$
where $C$ is a constant independent of $\tau,x$, and $\varphi_+$ is defined in \eqref{de:varphi+}.
\end{theorem}

\subsection{The Properties of the Free Boundaries}

Thanking to the fact $\partial_x u\geq0$ in Theorem \ref{th-u}, we can define the free boundaries as follows:
\begin{equation}   \label{de:x0x1}
  x_0(\tau)=\sup\{x:u(\tau,x)<\zeta_0\};\quad 
  x_1(\tau)=\sup\{x:u(\tau,x)=-\zeta_1\},\;\;\forall\;\;
  \tau\in(0,T],
\end{equation}
and it is clear that 
\begin{equation}\label{de:domain-tau-x}
\begin{aligned}
 &{\cal RR}_0:=\{(\tau,x)\in{\cal U}_T:u(\tau,x)=\zeta_0\}=\{(\tau,x):x\geq x_0(\tau):0<\tau\leq T\};\\[2mm]
 &{\cal WR}:=\{(\tau,x)\in{\cal U}_T:-\zeta_1<u(\tau,x)<\zeta_0\}=\{(\tau,x):x_1(\tau)<x<x_0(\tau) :0<\tau\leq T\};\\[2mm]
 &{\cal RR}_1:=\{(\tau,x)\in{\cal U}_T:u(\tau,x)=-\zeta_1\}=\{(\tau,x):x\leq x_1(\tau):0<\tau\leq T\}.
\end{aligned}
\end{equation}

In order to study the properties of the free boundaries $x_0(\tau)$ and $x_1(\tau)$, we establish some estimates about $u$ via the following lemmas.
\begin{lem}\label{lem-boundary estimate1} 
The unique strong solution $u$ of VI \eqref{eq:VI-u} has the following estimates 
$$
 u>-\zeta_1\;\;\mbox{in}\;\;(0,T]\times(X_1,+\infty),\qquad 
 u<\zeta_0\;\;\mbox{in}\;\;(0,T]\times(-\infty,X_2)\cup [0,T_1]\times\mathbb{R},
$$
where  $T_1$ is defined in \eqref{de:T1}, and 
\begin{equation}   \label{de:X1X2}
  X_1:=\ln{\frac{L_1-L_0}{\e_1-\e_0+r\zeta_1}},\qquad 
  X_2:=\ln{\frac{L_1-L_0}{\e_1-\e_0-r\zeta_0}}> X_1.
\end{equation}
\end{lem}
\begin{proof}Temporarily denote
$$
\underline{u}=(-\zeta_1+\varepsilon^4)-\varepsilon^2(x-X^\varepsilon_1-\varepsilon)^2{\rm\bf I}_{\{x<X^\varepsilon_1+\varepsilon\}},\qquad 
X^\varepsilon_1=\ln{\frac{L_1-L_0}{\e_1-\e_0+r\zeta_1-r\varepsilon^4-\varepsilon}}>X_1
$$
with small enough positive $\varepsilon$. It is not difficult to check that  $\underline{u}\in W^{1,2}_{p,loc}([0,T]\times[X^\varepsilon_1,+\infty))$ with any $p\geq1$, and in the domain $(0,T]\times(X^\varepsilon_1,+\infty)$,  $\underline{u}$ satisfies
\begin{equation*}
    \begin{cases}
        \partial_\tau \underline{u}- {\cal L}_x \underline{u}-(\e_1-\e_0)+(L_1-L_0)e^{-x} =\left[-r(\zeta_1-\varepsilon^4)-(\e_1-\e_0)+(L_1-L_0)e^{-x}\right]
        +\varepsilon^2{\rm\bf I}_{\{x<X^\varepsilon_1+\varepsilon\}}\left[\t^2\right.
        \vspace{2mm}\\
        \qquad\qquad\qquad\left.+(2\b-2r+\t^2)(x-X^\varepsilon_1-\varepsilon)-r(x-X^\varepsilon_1-\varepsilon)^2\right]
        \leq-\varepsilon+\varepsilon^2{\rm\bf I}_{\{x<X^\varepsilon_1+\varepsilon\}}(\t^2+2r\varepsilon)\leq0;
        \vspace{2mm}\\
       \underline{u}(0,x)\leq -\zeta_1+\varepsilon^4\leq 0=u(0,x);\qquad \underline{u}(\tau,X^\varepsilon_1)=-\zeta_1\leq u(\tau,X^\varepsilon_1);\qquad 
       \underline{u}\leq -\zeta_1+\varepsilon^4< \zeta_0
    \end{cases}
\end{equation*}
provided that $\varepsilon$ is small enough. Applying the comparison principle for VIs to VIs of $\underline{u}$ and $u$ in the domain $(0,T]\times(X^\varepsilon_1,+\infty)$, we know that 
$$
 u \geq \underline{u}\;\mbox{in}\;[0,T]\times[X^\varepsilon_1,+\infty)\;\;
 \mbox{and}\;\;u\geq \underline{u}=-\zeta_1+\varepsilon^4>-\zeta_1\;\mbox{in}\;[0,T]\times[X^\varepsilon_1+\varepsilon,+\infty). 
$$
Thanking to the fact that $X^\varepsilon_1+\varepsilon\rightarrow X_1$ as $\varepsilon$ tends to $0^+$, we deduce that 
$$
  u>-\zeta_1\;\mbox{in}\;(0,T]\times(X_1,+\infty). 
$$

On the other hand, if we temporarily denote
$$
\overline{u}=(\zeta_0-\varepsilon^4)+\varepsilon^2(x-X^\varepsilon_2+\varepsilon)^2{\rm\bf I}_{\{x>X^\varepsilon_2-\varepsilon\}},\qquad 
X^\varepsilon_2=\ln{\frac{L_1-L_0}{\e_1-\e_0-r\zeta_0+r\varepsilon^4+\varepsilon}}<X_2
$$
with small enough positive $\varepsilon$, then we can check that  $\overline{u}\in W^{1,2}_{p,loc}([0,T]\times(-\infty,X^\varepsilon_2])$ with any $p\geq1$, and in the domain $(0,T]\times(-\infty,X^\varepsilon_2)$,  $\overline{u}$ satisfies
\begin{equation*}
    \begin{cases}
        \partial_\tau \overline{u}- {\cal L}_x \overline{u}-(\e_1-\e_0)+(L_1-L_0)e^{-x} =\left[r(\zeta_0-\varepsilon^4)-(\e_1-\e_0)+(L_1-L_0)e^{-x}\right]
        -\varepsilon^2{\rm\bf I}_{\{x>X^\varepsilon_2-\varepsilon\}}\left[\t^2\right.
        \vspace{2mm}\\
        \qquad\qquad\qquad\left.+(2\b-2r+\t^2)(x-X^\varepsilon_2+\varepsilon)-r(x-X^\varepsilon_2+\varepsilon)^2\right]
        \geq\varepsilon-\varepsilon^2{\rm\bf I}_{\{x>X^\varepsilon_2-\varepsilon\}}(\t^2+2\b\varepsilon+\t^2\varepsilon)\geq0;
        \vspace{2mm}\\
       \overline{u}(0,x)\geq \zeta_0-\varepsilon^4\geq 0=u(0,x);\qquad \overline{u}(\tau,X^\varepsilon_2)=\zeta_0\geq u(\tau,X^\varepsilon_2);\qquad 
       \overline{u}\geq \zeta_0-\varepsilon^4>-\zeta_1
    \end{cases}
\end{equation*}
provided that $\varepsilon$ is small enough. Applying the comparison principle for VIs to VIs for $\overline{u}$ and $u$ in the domain $(0,T]\times(-\infty,X^\varepsilon_2)$, we know that 
$$
 u \leq \overline{u}\;\mbox{in}\;[0,T]\times(-\infty,X^\varepsilon_2]\;\;
 \mbox{and}\;\;u\leq \overline{u}=\zeta_0-\varepsilon^4<\zeta_0\;\mbox{in}\;(0,T]\times(-\infty,X^\varepsilon_2-\varepsilon]. 
$$
Thanking to the fact that $X^\varepsilon_2-\varepsilon\rightarrow X_2$ as $\varepsilon$ tends to $0^+$, we deduce that 
$$
  u<\zeta_0\;\mbox{in}\;(0,T]\times(-\infty,X_2). 
$$

Moreover, 
according to Theorem \ref{th-u}, we know that
$$
 u\leq \varphi_+\;\mbox{in}\;[0,T_1]\times\mathbb{R}\quad
 \mbox{and}\quad u\leq \varphi_+< \varphi_+(T_1)=\zeta_0\;\mbox{in}\;[0,T_1)\times\mathbb{R}.
$$
So, in the domain $(0,T_1)\times(X_1,+\infty)$, $u$ satisfies $-\zeta_1<u<\zeta_0$ and PDE
$$
 \partial_\tau u- {\cal L}_x u-(\e_1-\e_0)+(L_1-L_0)e^{-x}=0,\;\;u(0,x)=0.
$$
Then the theory for PDEs implies that $u\in C^\infty([0,T_1]\times[X_2,+\infty))$. Moreover, in the domain $[0,T_1]\times[X_2,+\infty)$, $\,\varphi_+\in C^\infty([0,T_1]\times[X_2,+\infty))$, and satisfies 
\begin{equation*}
    \begin{cases}
        \partial_\tau \varphi_+- {\cal L}_x \varphi_+-(\e_1-\e_0)+(L_1-L_0)e^{-x}=(L_1-L_0)e^{-x}>0
       =\partial_\tau u- {\cal L}_x u-(\e_1-\e_0)+(L_1-L_0)e^{-x};
        \vspace{2mm}\\
       \varphi_+(0,x)=u(0,x);\quad 
       \varphi_+(\tau,X_2)\geq u(\tau,X_2).
    \end{cases}
\end{equation*}
Applying the strong maximum principle for PDEs, we deduce that 
$$
 u<\varphi_+\;\;\mbox{in}\;\;(0,T_1]\times(X_2,+\infty),\qquad 
 u<\varphi_+(T_1,x)=\zeta_0\;\;\mbox{in}\;\;[0,T_1]\times(X_2,+\infty).
$$
Combining $\partial_x u\geq 0$, we know that 
$$
 u<\zeta_0\;\;\mbox{in}\;\;[0,T_1]\times \mathbb{R}.
$$
\end{proof}

\begin{lem}\label{lem-boundary estimate2}
For any $\varepsilon\in(0,T-T_1)$, there exists a constant $K_\varepsilon>0$ such that 
 $$u=-\zeta_1\;\mbox{in}\;[\varepsilon,T]\times(-\infty,-K_\varepsilon],\quad  
 u=\zeta_0\;\mbox{in}\;[T_1+\varepsilon,T]\times[K_\varepsilon,+\infty).$$
\end{lem}
\begin{proof}

Temporarily denote
$$
 \overline{u}(\tau,x)=-\left({\zeta_1\tau\over \varepsilon}{\rm\bf I}_{\{\tau\leq \varepsilon\}}+\zeta_1{\rm\bf I}_{\{\tau> \varepsilon\}}\right)
 +\left(x+{2\over \varepsilon}\right)^2{\rm\bf I}_{\{x\geq -2/\varepsilon\}},\;\;\forall\;\;
 (\tau,x)\in [0,T]\times(-\infty,-1/\varepsilon].
$$
Then it is not difficult to check that   $\overline{u}\in W^{1,2}_{p,loc}([0,T]\times(-\infty,-1/\varepsilon])$ with any $p\geq1$, and in the domain $[0,T]\times(-\infty,-1/\varepsilon]$, $\overline{u}$ satisfies
\begin{equation*}
    \begin{cases}
        \partial_\tau \overline{u}- {\cal L}_x \overline{u}-(\e_1-\e_0)+(L_1-L_0)e^{-x}
        =\left[-{\zeta_1\over \varepsilon}(1+r\tau){\rm\bf I}_{\{\tau\leq \varepsilon\}}-r\zeta_1{\rm\bf I}_{\{\tau> \varepsilon\}}-(\e_1-\e_0)+(L_1-L_0)e^{-x}\right]
        \vspace{2mm}\\
        \hspace{6.4cm}-\left[\t^2+(2\b-2r+\t^2)\left(x+{2\over \varepsilon}\right)-r\left(x+{2\over \varepsilon}\right)^2\right]{\rm\bf I}_{\{x\geq -2/\varepsilon\}}        
        \vspace{2mm}\\
        \hspace{6cm}\geq\frac{L_1-L_0}{2}e^{1/\varepsilon}-\left(\t^2+{2\b+\t^2\over \varepsilon}\right){\rm\bf I}_{\{x\geq -2/\varepsilon\}}\geq0;
        \vspace{2mm}\\
       \overline{u}(0,x)\geq 0=u(0,x);\quad 
       \overline{u}(\tau,-1/\varepsilon)\geq-\zeta_1+{1\over \varepsilon^2}\geq\zeta_0\geq u(\tau,-1/\varepsilon);\quad 
       \overline{u}\geq \overline{u}(\varepsilon,-2/\varepsilon)=-\zeta_1
    \end{cases}
\end{equation*}
provided that $\varepsilon$ is small enough. Applying the comparison principle for VIs to VIs for $\overline{u}$ and $u$ in the domain $(0,T]\times(-\infty,-1/\varepsilon)$, we deduce that 
$$
 \overline{u}\geq u\;\mbox{in}\;[0,T]\times(-\infty,-1/\varepsilon]\quad 
 \mbox{and}\quad u(\tau,x)=\overline{u}(\tau,x)=-\zeta_1,\;
 \forall\;(\tau,x)\in [\varepsilon,T]\times(-\infty,-2/\varepsilon],
$$
which implies that for any $\varepsilon\in(0,T)$, there exists a constant $K_\varepsilon>0$ such that 
 $$u=-\zeta_1\;\mbox{in}\;[\varepsilon,T]\times(-\infty,-K_\varepsilon].$$

Temporarily denote
$$
 \underline{u}=\left[{1\over r}(1-\varepsilon)(\e_1-\e_0)\left(1-e^{-r\tau}\right){\rm\bf I}_{\{\tau\leq{T}^\varepsilon_1\}}+\zeta_0{\rm\bf I}_{\{\tau>{T}^\varepsilon_1\}}\right]-\varepsilon^5\left(x-{2\over \varepsilon^3}\right)^2{\rm\bf I}_{\{x\leq 2/\varepsilon^3\}}\;\;\mbox{in}\;\;[0,T]\times[1/\varepsilon^3,+\infty),
$$
where $\varepsilon$ is a positive constant small enough, and 
$$
 {T}^\varepsilon_1=-{1\over r}\ln\left[1-{r\zeta_0\over (1-\varepsilon)(\e_1-\e_0)}\right]>T_1.
$$
Then it is not difficult to check that $\underline{u}\in W^{1,2}_{p,loc}([0,T]\times[1/\varepsilon^3,+\infty))$ with any $p\geq1$, and in the domain $[0,T]\times[1/\varepsilon^3,+\infty)$, $\underline{u}$ satisfies
$$
 \partial_\tau\underline{u}\geq0;\qquad 
 \underline{u}(\tau,x)\leq \underline{u}({T}^\varepsilon_1,2/\varepsilon^3)=\zeta_0;
 \qquad 
 \underline{u}(0,x)\leq 0;\qquad 
 \underline{u}(\tau,1/\varepsilon^3)\leq \zeta_0-1/\varepsilon<-\zeta_1,
$$
and
\begin{eqnarray*}
    &&\partial_\tau \underline{u} - {\cal L}_x\underline{u} -(\e_1-\e_0)+(L_1-L_0)e^{-x}
    \\[2mm]
    &=&\left[-\varepsilon(\e_1-\e_0){\rm\bf I}_{\{\tau\leq{T}^\varepsilon_1\}}
    +(r\zeta_0-\e_1+\e_0){\rm\bf I}_{\{\tau>{T}^\varepsilon_1\}}
    +(L_1-L_0)e^{-x}\right]
    +\varepsilon^5
    \left[\t^2+\left(2\beta -2r+\t^2\right)\left(x-{2\over \varepsilon^3}\right)\right.
    \\[2mm]
    &&\left.-r\left(x-{2\over \varepsilon^2}\right)^3\right]{\rm\bf I}_{\{x\leq 2/\varepsilon^3\}}
    \\[2mm]
    &\leq&-{\varepsilon\over2}(\e_1-\e_0)+\varepsilon^5
    \left[\t^2-2r\left(x-{2\over \varepsilon^3}\right)\right]{\rm\bf I}_{\{x\leq 2/\varepsilon^3\}}
    \leq -{\varepsilon\over2}(\e_1-\e_0)
    +\varepsilon^5\left(\t^2+\frac{2r}{\varepsilon^3}\right)\leq0
\end{eqnarray*}
provided that $\varepsilon$ is small enough, where we have used the fact that $r>0$ and $r\zeta_0-\e_1+\e_0<0$. Applying the comparison principle for VIs to VIs for $\underline{u}$ and $u$ in the domain $[0,{T}]\times[1/\varepsilon^3,+\infty)$, we deduce that 
$$
 u\geq\underline{u}\;\mbox{in}\;[0,T]\times[1/\varepsilon^3,+\infty)
 \quad 
 \mbox{and}\quad u=\underline{u}=\zeta_0\;\mbox{in}\;[{T}^\varepsilon_1,T]\times[2/\varepsilon^3,+\infty).
$$ 
Note that ${T}^\varepsilon_1\rightarrow T_1$ as $\varepsilon$ tends to $0^+$, then we know that  for any $\varepsilon>0$, there exists a $K_\varepsilon>0$ such that 
 $$u=\zeta_0\;\mbox{in}\;[T_1+\varepsilon,T]\times[K_\varepsilon,+\infty).$$
\end{proof}

\begin{figure}[ht]
	\centering
	\subfigure{\includegraphics[scale=0.8]{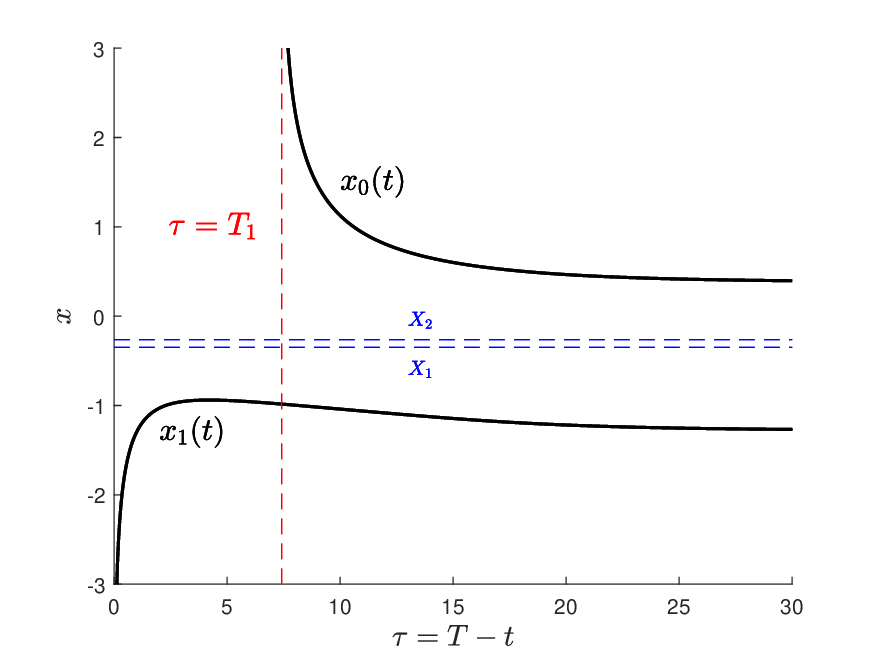}}
	\caption{The two boundaries $x_0(t)$ and $x_1(t)$ in $(\tau,x)$-domain.\label{free boundary in tau-x}}
\end{figure}

\begin{theorem}\label{theorem for free boundaires}
The free boundaries $x_1(\cdot)\in C^\infty(0,T]$ and $x_0(\cdot)\in C^\infty(T_1,T]$, and $u\in C^\infty(\overline{{\cal RR}_1})\cap  C^\infty(\overline{{\cal WR}})\cap  C^\infty(\overline{{\cal RR}_0}),\partial_x u,\partial_\tau u\in C(\overline{\cal U}_T)$. Moreover, 
$$
 -\infty<x_1(\cdot)<X_1,\qquad x_0(\cdot)>X_2,\qquad x_0(\tau)=+\infty,\;\forall\;\tau\in[0,T_1],\qquad 
 x_0(\tau)<+\infty,\;\forall\;\tau\in(T_1,T],
$$
$$
 \lim\limits_{\tau\rightarrow0^+}x_1(\tau)=-\infty,\qquad 
 \lim\limits_{\tau\rightarrow T_1^+}x_0(\tau)=+\infty,
$$
$$
 \lim\limits_{x\rightarrow (x_1(\tau))^+}(\partial_{xx} u+\partial_{x} u)(\tau,x)>0,\;\forall\;\tau\in(0,T],\quad
 \lim\limits_{x\rightarrow (x_0(\tau))^-}(\partial_{xx} u+\partial_{x} u)(\tau,x)<0,\;\forall\;\tau\in(T_1,T],
$$
where $T_1$ is defined in \eqref{de:T1}, and $X_1,X_2$ are defined in \eqref{de:X1X2} (see Figure \ref{free boundary in tau-x}).
\end{theorem}  
\begin{proof} We divide the proof into 4 steps.\\
\noindent {\bf Step 1.} According to Lemma \ref{lem-boundary estimate1}, we deduce that 
$$
 x_1(\tau)\leq X_1,\;\forall\;\tau\in[0,T],\qquad 
 x_0(\tau)=+\infty,\;\forall\;\tau\in[0,T_1],\qquad 
 x_0(\tau)\geq X_2,\;\forall\;\tau\in[0,T].
$$

Next, we prove that $x_1(\tau)<X_1$ for any $\tau\in(0,T]$. Otherwise, there exists a $\tau_0\in(0,T]$ such that $x_1(\tau_0)=X_1$ and $u(\tau_0,X_1)=-\zeta_1$. Since $x_1(\tau)\leq X_1<X_2\leq x_0(\tau)$, we know that in the domain $[0,T]\times (X_1,X_2)$, $u\in C^\infty((0,T]\times (X_1,X_2))$ satisfies the following PDEs
$$
 \partial_\tau u-{\cal L}_xu=(\e_1-\e_0)-(L_1-L_0)e^{-x}>(\e_1-\e_0)-(L_1-L_0)e^{-X_1}
 =\partial_\tau (-\zeta_1)-{\cal L}_x (-\zeta_1).
$$
So, the strong principle for PDEs implies that $u>-\zeta_1$ in $(0,T]\times (X_1,X_2)$. Moreover, the Hopf's lemma (see \cite{Evans}) implies that 
$$\partial_x u(\tau_0,x_1(\tau_0))=\partial_x u(\tau_0,X_1)>0.$$
On the other hand, since $u\in W^{1,2}_{p,loc}({\cal U}_T)$ for any $p\geq1$, the embedding theorem implies that $\partial_x u\in C(\overline{\cal U}_T)$ and 
\begin{eqnarray} \label{partial-x-u-x1x0} 
\partial_xu (\tau,x_1(\tau))=\lim\limits_{x\rightarrow(x_1(\tau))^-}\partial_x u(\tau,x)=\partial_x(-\zeta_1)=0,\,\forall\,\tau\in(0,T];\quad 
 \partial_xu (\tau,x_0(\tau))=0,\,\forall\,\tau\in(T_1,T].
\end{eqnarray}
Hence, we have obtain a contradiction, and proved that $x_1(\tau)<X_1$ for any $\tau\in(0,T]$. 

Now, we prove that $x_0(\tau)>X_2$ for any $\tau\in(0,T]$. Otherwise, there exists a $\tau_0\in(0,T]$ such that $x_0(\tau_0)=X_2$ and $u(\tau_0,X_2)=\zeta_0$. As above, we know that in the domain $[0,T]\times (X_1,X_2)$, $u\in C^\infty((0,T]\times (X_1,X_2))$ satisfies the following PDEs
$$
 \partial_\tau u-{\cal L}_xu=(\e_1-\e_0)-(L_1-L_0)e^{-x}<(\e_1-\e_0)-(L_1-L_0)e^{-X_2}
 =\partial_\tau \zeta_0-{\cal L}_x \zeta_0.
$$
So, the strong principle for PDEs implies that $u<\zeta_0$ in $(0,T]\times (X_1,X_2)$, and the Hopf's lemma implies that $\partial_x u(\tau_0,X_2)>0$, which contradicts with the second equality in  \eqref{partial-x-u-x1x0}. Hence, we have proved that $x_0(\tau)> X_2$ for any $\tau\in(0,T]$. 

\noindent {\bf Step 2.} By  the first equality and second  equality in Lemma \ref{lem-boundary estimate2}, we derive that 
$$
 x_1(\tau)>-\infty,\;\forall\;\tau\in(0,T],\qquad 
 x_0(\tau)<+\infty,\;\forall\;\tau\in(T_1,T].
$$

\noindent {\bf Step 3.} We claim that 
\begin{equation}\label{eq:starting point-x1}
    \lim\limits_{\tau\rightarrow0^+}x_1(\tau)=-\infty.
\end{equation}
Otherwise, there exists a sequence $\{\tau_n\}_{n=1}^\infty$ such that $\tau_n\rightarrow 0^+$, and $x_1(\tau_n)\rightarrow x_0>-\infty$ as $n$ tends to infinite. Combining the fact that $u\in C(\overline{\cal U}_T)$, we know that
$$ 
 u(0,x_0)=\lim\limits_{n\rightarrow+\infty}u(\tau_n,x_1(\tau_n))
 =\lim\limits_{n\rightarrow+\infty}(-\zeta_1)=-\zeta_1,
$$
which contradicts with the initial condition $u(0,x)=0$. So, we have shown \eqref{eq:starting point-x1}. 

Moreover, We claim that 
\begin{equation}\label{eq:starting point-x0}
    \lim\limits_{\tau\rightarrow T_1^+}x_0(\tau)=+\infty.
\end{equation}
Otherwise, there exists a sequence $\{\tau_n\}_{n=1}^\infty$ such that $\tau_n\rightarrow T_1^+$, and $x_0(\tau_n)\rightarrow x_0<+\infty$ as $n$ tends to infinite. Combining the fact that $u\in C(\overline{\cal U}_T)$, we know that
$$ 
u(T_1,x_0)=\lim\limits_{n\rightarrow+\infty}u(\tau_n,x_0(\tau_n))=
 \lim\limits_{n\rightarrow+\infty}\zeta_0=\zeta_0,
$$
which contradicts with the second  inequality in Lemma \ref{lem-boundary estimate1}. So, we have shown \eqref{eq:starting point-x0}. 

\noindent {\bf Step 4.} We improve the regularity of the free boundaries $x_1(\cdot),x_0(\cdot)$ and the strong solution $u$. 

It is clear $u=-\zeta_1$ in ${\cal RR}_1$ and $u\in C^\infty(\overline{{\cal RR}_1})$, and $u=\zeta_0$ in ${\cal RR}_0$ and $u\in C^\infty(\overline{{\cal RR}_0})$, and $-\zeta_1<u<\zeta_0$ in ${\cal WR}$. Moreover, from the interior estimates for PDEs (see \cite{Li96}), we know that 
$u\in C^\infty({\cal WR})$.

From the last two inequalities in Theorem \ref{th-u}, we know that for any $\varepsilon$, there exists a positive constant  $C_\varepsilon$ such that 
\begin{eqnarray} \label{es-cone}   
 -C_\varepsilon\partial_xu\leq \partial_\tau u\leq C_\varepsilon e^x\partial_xu\;\;\mbox{a.e. in}\;\;
 [\varepsilon,T]\times \mathbb{R}.
\end{eqnarray}
Combining \eqref{partial-x-u-x1x0}, we know that  $\partial_\tau u\in C(\overline{\cal U}_T)$ and 
$$
 \partial_\tau u(\tau,x_1(\tau))=\lim\limits_{(t,x)\rightarrow(\tau,x_1(\tau))}\partial_\tau u(t,x)=0,\quad 
  \partial_\tau u(\tau,x_0(\tau))=\lim\limits_{(t,x)\rightarrow(\tau,x_0(\tau))}\partial_\tau u(t,x)=0.
$$
So, from the differential equation in \eqref{eq:VI-u} and $u\in C^\infty({\cal WR})$, we know that 
$\partial_{xx} u\in C(\overline{\cal WR})$, and 
\begin{eqnarray*}   
 \lim\limits_{x\rightarrow (x_1(\tau))^+}(\partial_{xx} u+\partial_{x} u)(\tau,x)&=&\dfrac{2}{\t^2}\left[(L_1-L_0)e^{-x_1(\tau)}-(\e_1-\e_0+r\zeta_1)\right]
 \\[2mm]
 &>&\dfrac{2}{\t^2}\left[(L_1-L_0)e^{-X_1}-(\e_1-\e_0+r\zeta_1)\right]=0,\;\forall\;\tau\in(0,T],
 \\[2mm]   
 \lim\limits_{x\rightarrow (x_0(\tau))^-}(\partial_{xx} u+\partial_{x} u)(\tau,x)&=&\dfrac{2}{\t^2}\left[(L_1-L_0)e^{-x_0(\tau)}-(\e_1-\e_0-r\zeta_0)\right]<0,\;\forall\;\tau\in(T_1,T].
\end{eqnarray*}
Moreover, from Lemma \ref{lem-boundary estimate2}, we deduce that there exists a constant $K_\varepsilon$ such that 
\begin{eqnarray}\label{es-x0x1}
 -K_\varepsilon<x_1(\tau)<X_1<X_2<x_0(\tau),\;\forall\;\tau\in [\varepsilon,T],\qquad 
 X_2<x_0(\tau)<K_\varepsilon,\;\forall\;\tau\in [T_1+\varepsilon,T].
\end{eqnarray}

In order to improve the regularity of $x_1(\cdot)$, we consider the VI of $u$ in the domain $\widetilde{\cal U}_{1,x}=\{(\tau,x):x>-K_\varepsilon-C_\varepsilon(T-\tau),\varepsilon<\tau\leq T\}$. We apply the following transformation
$$ 
 w_1(\tau,y)=u(\tau,x),\qquad x=y+C_\varepsilon\tau,
$$
then  $w_1$ satisfies the following VI,
\begin{equation*}
    \begin{cases}
        \partial_\tau w_1 - {\cal L}_{1,y}w_1 -(\e_1-\e_0)+(L_1-L_0)e^{-y-C_\varepsilon\tau}=0\;\;\;\mbox{in}\;\;
        \{(\tau,y)\in\widetilde{\cal U}_{1,y}:-\zeta_1 <w_1(\tau,y)< \zeta_0\}; \vspace{2mm}\\
        \partial_\tau w_1 - {\cal L}_{1,y}w_1 -(\e_1-\e_0)+(L_1-L_0)e^{-y-C_\varepsilon\tau}\ge 0\;\;\;\mbox{in}\;\;
        \{(\tau,y)\in\widetilde{\cal U}_{1,y}:w_1(\tau,x)=-\zeta_1\}; \vspace{2mm}\\
        \partial_\tau w_1 - {\cal L}_{1,y}w_1 -(\e_1-\e_0)+(L_1-L_0)e^{-y-C_\varepsilon\tau}\le 0\;\;\mbox{in}\;\;\;
        \{(\tau,y)\in\widetilde{\cal U}_{1,y}:w_1(\tau,x)=\zeta_0\}; \vspace{2mm}\\
        -\zeta_1\leq w_1\leq \zeta_0;\quad         w_1(\varepsilon,y)=u(\varepsilon,y+C_\varepsilon\varepsilon),\;y\in[K_{1,\varepsilon},+\infty);\quad        w_1(\tau,K_{1,\varepsilon})=u(\tau,K_{1,\varepsilon}+C_\varepsilon\tau),\;
        \tau\in [\varepsilon,T] 
    \end{cases}
\end{equation*}
with 
\begin{equation*} 
		 {\cal L}_{1,y}:=\dfrac{\t^2}{2}\partial_{yy}+ \left(\beta -r+{\frac{\t^2}{2}}+C_\varepsilon\right)\partial_{y}- r,\quad 
   \widetilde{\cal U}_{1,y}:=(\varepsilon,T]\times(K_{1,\varepsilon},+\infty),\quad 
   K_{1,\varepsilon}:=-K_\varepsilon-C_\varepsilon T.
\end{equation*}
From the transformation and  \eqref{es-x0x1}, we deduce that the free boundaries
$$ 
  \{(\tau,x_1(\tau)):\tau\in(\varepsilon,T]\}\subset \widetilde{\cal U}_{1,x},\qquad 
  \{(\tau,y_1(\tau)):\tau\in(\varepsilon,T]\}\subset \widetilde{\cal U}_{1,y},\;\;
  y_1(\tau):=x_1(\tau)-C_\varepsilon\tau.
$$
Moreover, by the transformation and \eqref{es-cone}, we know that 
$$
 \partial_\tau (w_1-(-\zeta_1))=\partial_\tau u+C_\varepsilon\partial_x u\geq0,\;\;
 \partial_y (w_1-(-\zeta_1))=\partial_x u\geq0\;\;
 \mbox{in}\;\;\widetilde{\cal U}_{1,y}.
$$ 
Note that $-\zeta_1$ is the lower obstacle. So, by the standard method (\cite{F2}), we know that the free boundary in $(\tau,y)-$Coordinate System $y_1(\cdot)\in C^\infty((\varepsilon,T])$. Hence, the free boundary in $(\tau,x)-$Coordinate System $x_1(\cdot)\in C^\infty((\varepsilon,T])$, too. Since $\varepsilon$ is arbitrary, we deduce that $x_1(\cdot)\in C^\infty((0,T)$.

In order to improve the regularity of $x_0(\cdot)$, we consider the VI of $u$ in the domain $\widetilde{\cal U}_{2,x}=\{(\tau,x):x<K_\varepsilon+\widetilde{C}_\varepsilon(T-\tau),T_1+\varepsilon<\tau\leq T\}$, and apply the following transformation
$$ 
 w_2(\tau,y)=u(\tau,x),\quad x=y-\widetilde{C}_\varepsilon\tau,\quad \widetilde{C}_\varepsilon=C_\varepsilon e^{K_\varepsilon},
$$
where $C_\varepsilon$ and $K_\varepsilon$ are respectively the constants in \eqref{es-cone} and \eqref{es-x0x1}, then  $w_2$ satisfies the following VI,
\begin{equation*}
    \begin{cases}
        \partial_\tau w_2 - {\cal L}_{2,y}w_2 -(\e_1-\e_0)+(L_1-L_0)e^{-y+\widetilde{C}_\varepsilon\tau}=0\;\;\;\mbox{in}\;\;
        \{(\tau,y)\in\widetilde{\cal U}_{2,y}:-\zeta_1 <w_2(\tau,y)< \zeta_0\}; \vspace{2mm}\\
        \partial_\tau w_2 - {\cal L}_{2,y}w_2 -(\e_1-\e_0)+(L_1-L_0)e^{-y+\widetilde{C}_\varepsilon\tau}\ge 0\;\;\;\mbox{in}\;\;
        \{(\tau,y)\in\widetilde{\cal U}_{2,y}:w_2(\tau,y)=-\zeta_1\}; \vspace{2mm}\\
        \partial_\tau w_2 - {\cal L}_{2,y}w_2 -(\e_1-\e_0)+(L_1-L_0)e^{-y+\widetilde{C}_\varepsilon\tau}\le 0\;\;\mbox{in}\;\;\;
        \{(\tau,y)\in\widetilde{\cal U}_{2,y}:w_2(\tau,y)=\zeta_0\}; \vspace{2mm}\\
        -\zeta_1 \leq w_2 \leq \zeta_0;\qquad w_2(T_1+\varepsilon,y)=u(T_1+\varepsilon,y-\widetilde{C}_\varepsilon(T_1+\varepsilon)),\;y\in(-\infty, K_{2,\varepsilon}]; \vspace{2mm}\\        w_2(\tau,K_{2,\varepsilon})=u(\tau,K_{2,\varepsilon}-\widetilde{C}_\varepsilon\tau),\,
        \tau\in [T_1+\varepsilon,T] 
    \end{cases}
\end{equation*}
with 
\begin{equation*} 
		 {\cal L}_{2,y}:=\dfrac{\t^2}{2}\partial_{yy}+ \left(\beta -r+{\frac{\t^2}{2}}-\widetilde{C}_\varepsilon\right)\partial_{y}- r,\quad 
   \widetilde{\cal U}_{2,y}:=(T_1+\varepsilon,T]\times(-\infty, K_{2,\varepsilon}),\quad 
   K_{2,\varepsilon}:=K_\varepsilon+\widetilde{C}_\varepsilon T.
\end{equation*}
From the transformation and  \eqref{es-x0x1}, we deduce that the free boundaries
$$ 
  \{(\tau,x_0(\tau)):\tau\in(T_1+\varepsilon,T]\}\subset \widetilde{\cal U}_{2,x},\;\;\;
  \{(\tau,y_0(\tau)):\tau\in(T_1+\varepsilon,T]\}\subset \widetilde{\cal U}_{2,y},\;\;
  y_0(\tau):=x_0(\tau)+\widetilde{C}_\varepsilon\tau.
$$
Moreover, by the transformation and \eqref{es-cone}, we know that 
$$
 \partial_\tau (w_2-\zeta_0)=\partial_\tau u-\widetilde{C}_\varepsilon\partial_x u\leq
 (\partial_\tau u- C_\varepsilon e^{x}\partial_x u){\rm\bf I}_{\{x\leq K_\varepsilon\}}+(\partial_\tau\zeta_0 -\widetilde{C}_\varepsilon\partial_x\zeta_0){\rm\bf I}_{\{x>K_\varepsilon\}}\leq0\;\;\mbox{in}\;\;\widetilde{\cal U}_{2,y}; 
$$ 
and 
$$
 \partial_y (w_2-\zeta_0)=\partial_x u\geq0\;\;\mbox{in}\;\;\widetilde{\cal U}_{2,y}.
$$
 Note that $\zeta_0$ is the upper obstacle. So, by the standard method, we know that the free boundary in $(\tau,y)-$Coordinate System $y_0(\cdot)\in C^\infty((T_1+\varepsilon,T])$. Hence, the free boundary in $(\tau,x)-$Coordinate System $x_0(\cdot)\in C^\infty((T_1+\varepsilon,T])$, too. Since $\varepsilon$ is arbitrary, we know that $x_0(\cdot)\in C^\infty((T_1,T])$. Moreover, by the standard theory for PDEs, we can deduce that $u\in C^\infty(\overline{\cal WR})$.
\end{proof}

\begin{figure}[ht]
	\centering
	\subfigure[$\Lambda_0(t)$]{\includegraphics[scale=0.6]{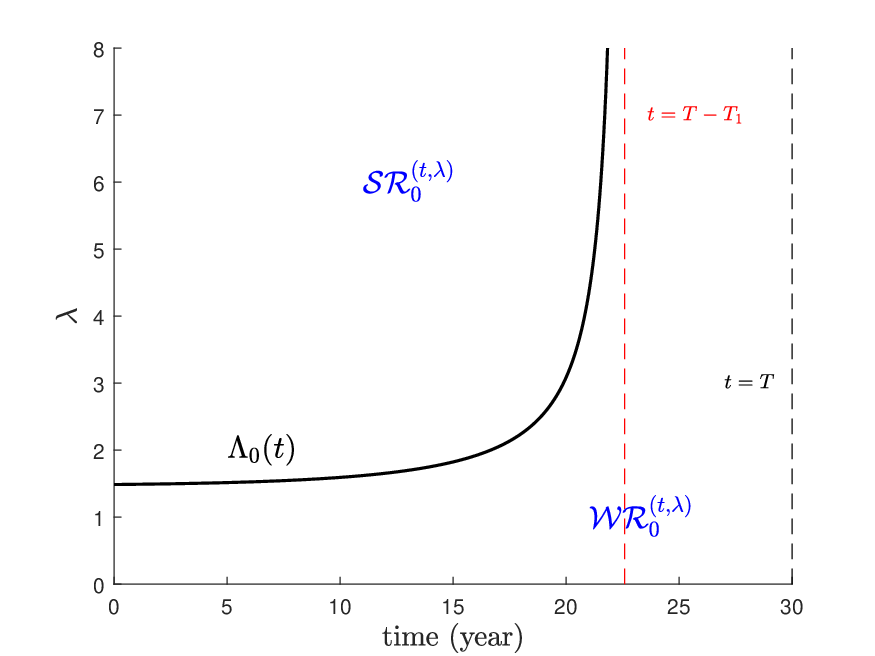}}
	\subfigure[$\Lambda_1(t)$]{\includegraphics[scale=0.6]{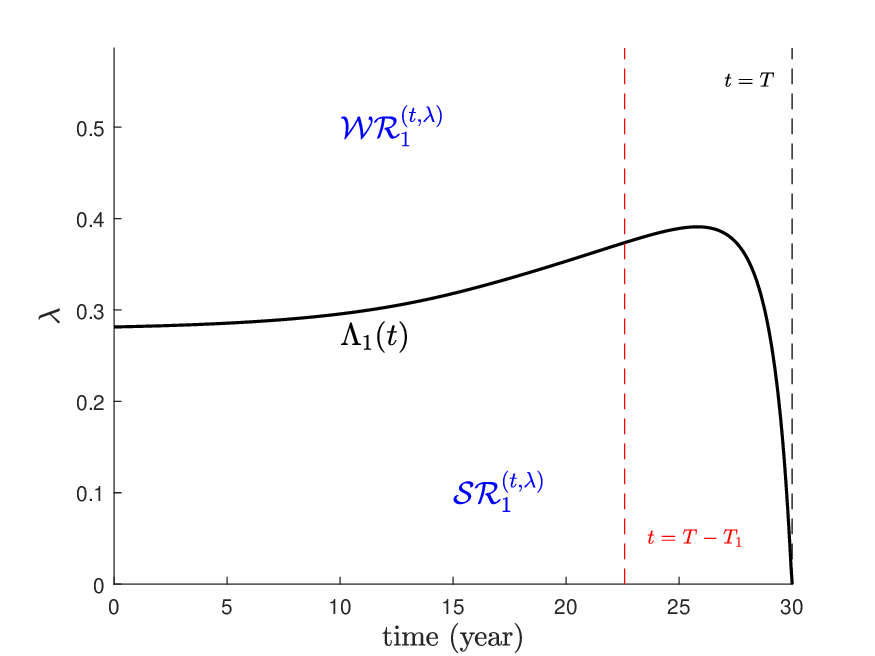}}
	\caption{The four regions $\mathcal{WR}_0^{(t,\l)}$, $\mathcal{SR}_0^{(t,\l)}$, $\mathcal{WR}_1^{(t,\l)}$, and $\mathcal{SR}_1^{(t,\l)}$ in $(t,\l)$-domain \label{fig:regions}}
\end{figure}

Thanks to Theorem \ref{theorem for free boundaires}, we can obtain the following theorem.
\begin{theorem}\label{property-Q}
If we define $\L_0(t)$ and $\L_1(t)$ by
\begin{equation}\label{eq:L0L1}
    \L_0(t):= e^{x_0(T-t)}\;\;\mbox{and}\;\;\L_1(t):=e^{x_1(T-t)},\;\;\;\forall\;\;t\in[0,T],
\end{equation}
where the functions $x_0(\cdot)$ and $x_1(\cdot)$ are defined in \eqref{de:x0x1}, then $\L_0$ and $\L_1$ satisfy the following properties. 
\begin{itemize}
    \item[(a)] \begin{equation}   \label{de:w0w1}
  \L_0(\tau)=\sup\{\l>0:{\cal P}(t,\l)<\zeta_0\};\quad\L_1(\tau)=\sup\{\l>0:{\cal P}(t,\l)=-\zeta_1\} ,\;\;\forall\;\;
  \tau\in(0,T],
\end{equation}
and
\begin{equation}\label{de:domain-t-lambda}
    \begin{aligned}
        \mathcal{WR}_0^{(t,\l)}:=&\{(t,\l)\in{\cal N}_T:({\cal Q}_1-{\cal Q}_0)(t,\l)<\zeta_0\l\}
        =\{(t,\l)\in{\cal N}_T:0<\l < \L_0(t)\};\vspace{2mm}\\
        \mathcal{SR}_0^{(t,\l)}:=&\{(t,\l)\in{\cal N}_T:({\cal Q}_1-{\cal Q}_0)(t,\l)=\zeta_0\l\}=\{(t,\l)\in{\cal N}_T:\l\geq \L_0(t)\};\vspace{2mm}\\
        \mathcal{WR}_1^{(t,\l)}:=&\{(t,\l)\in{\cal N}_T:({\cal Q}_1-{\cal Q}_0)(t,\l)>-\zeta_1\l\}=\{(t,\l)\in{\cal N}_T:\;\l>\L_1(t)\};\vspace{2mm}\\
        \mathcal{SR}_1^{(t,\l)}:=&\{(t,\l)\in{\cal N}_T:({\cal Q}_1-{\cal Q}_0)(t,\l)=-\zeta_1\l\}=\{(t,\l)\in{\cal N}_T:  0<\l\le \L_1(t)\},
    \end{aligned}
\end{equation}
where $({\cal Q}_0,{\cal Q}_1)$ is the unique strong solution of \eqref{de:Q0-and-Q1}, and satisfy VIs \eqref{eq:HJB-1} and \eqref{eq:HJB-2} (see Figure \ref{fig:regions}). 
    \item[(b)] $\L_0(\cdot)\in C^\infty([0,T-T_1))$ and $\L_1(\cdot)\in C^\infty([0,T))$.
    \item[(c)] $\L_0(\cdot)=+\infty$ in $[T-T_1,T]$, and $\L_0(\cdot)<+\infty$ in $[0,T-T_1)$, and $0<\L_1(\cdot)\leq e^{X_1}<e^{X_2}\leq\L_0(\cdot)$ in $[0,T)$,  and 
    $$
    \lim_{t\to (T-T_1)^-}\L_0(t) = +\infty \;\;\mbox{and}\;\;
    \lim_{t\to T^-} \L_1(t)=0.
    $$
\end{itemize}

Moreover, ${\cal Q}_0,\,{\cal Q}_1\in W^{1,2}_{p,loc}({\cal N}_T)\cap C^{0,1}(\overline{\cal N}_T)$ with $p\geq1$, and 
$${\cal Q}_0\in C^\infty\left(\;\overline{\mathcal{WR}_0^{(t,\l)}}\;\right)\cap C^\infty\left(\;\overline{\mathcal{SR}_0^{(t,\l)}}\;\right),\;\;
{\cal Q}_1\in C^\infty\left(\;\overline{\mathcal{WR}_1^{(t,\l)}}\;\right)\cap C^\infty\left(\;\overline{\mathcal{SR}_1^{(t,\l)}}\setminus\{(T,0)\}\;\right),$$ and satisfy the following properties.
\begin{itemize}
    \item[(d)] 
\begin{equation}\label{expression-Q0}
     {\cal Q}_0=\dfrac{1-e^{-r(T-t)}}{r}\e_0\l -\dfrac{1-e^{-\b(T-t)}}{\b}L_0\;\;\mbox{in}\;\;
     [T-T_1,T]\times(0,+\infty).
\end{equation}
    \item[(e)]$|\partial_{\l}{\cal Q}_0|,  |\partial_{\l}{\cal Q}_0|\leq C,\,\partial_{\l}{\cal Q}_0\geq0$ in $\overline{\cal N}_T$, where $C$ is a large enough constant.
    \item[(f)] $\partial_{\l\l}{\cal Q}_0>0$ in ${\cal N}_{T-T_1}$, and $\partial_{\l\l}{\cal Q}_0=0$ in $[T-T_1,T]\times(0,+\infty)$; and $\partial_{\l\l}{\cal Q}_1>0$ in ${\cal N}_{T-T_1}\cup\mathcal{WR}_1^{(t,\l)}$, and $\partial_{\l\l}{\cal Q}_1=0$ in $[T-T_1,T]\times(0,+\infty)\cap\mathcal{SR}_1^{(t,\l)}$.
\end{itemize}
\end{theorem}

\begin{proof}
From the transformation \eqref{tran-P-u} and \eqref{de:x0x1}, \eqref{de:domain-tau-x}, and the results in Theorem \ref{theorem for free boundaires} and \ref{th-relationship-Q-and-P}, we know all properties of $\L_0$ and $\L_1$, and $\partial_t {\cal P}, \partial_\l {\cal P}\in C(\overline{\cal N}_T)$, and 
$$
 {\cal P}\in C^\infty\left(\;\overline{\mathcal{WR}_0^{(t,\l)}}\cap\overline{\mathcal{WR}_1^{(t,\l)}}\;\right)\cap C^\infty\left(\;\overline{\mathcal{SR}_0^{(t,\l)}}\;\right)\cap C^\infty\left(\;\overline{\mathcal{SR}_1^{(t,\l)}}\setminus\{(T,0)\}\;\right).
$$ 
From PDE \eqref{de:Q0-and-Q1}, we know that $({\cal Q}_0,{\cal Q}_1)$ is the unique strong solution of the following PDE, 
\begin{equation*}
    \begin{cases}
    &-\left(\partial_t {\cal Q}_0 +{\cal L}{\cal Q}_0 \right)=\widetilde{f}_0
    :=\left[(\e_1-\e_0-r\zeta_0)\l-(L_1-L_0)\right] {\rm\bf I}_{\mathcal{SR}_0^{(t,\l)}}+(\e_0 \l -L_0)\geq \e_0 \l -L_0;\vspace{2mm}\\ 
    &-\left(\partial_t {\cal Q}_1 +{\cal L}{\cal Q}_1\right)=\widetilde{f}_1:=\left[(L_1-L_0)-(\e_1-\e_0+r\zeta_1)\l\right] {\rm\bf I}_{\mathcal{SR}_1^{(t,\l)}}+(\e_1 \l -L_1)\geq\e_1 \l -L_1; \vspace{2mm}\\
    &{\cal Q}_0(T,\l) = {\cal Q}_1(T,\l)=0,\quad\l\in(0,+\infty),
    \end{cases}
\end{equation*}
Since the boundaries of $\mathcal{SR}_0^{(t,\l)}$ and $\mathcal{SR}_1^{(t,\l)}$ are respectively $\L_0(\cdot)\in C^\infty([0,T))$ and $\L_1(\cdot)\in C^\infty([0,T-T_1))$, and 
$$
\widetilde{f}_0\in C^\infty\left(\;\overline{\mathcal{WR}_0^{(t,\l)}}\;\right)\cap C^\infty\left(\;\overline{\mathcal{SR}_0^{(t,\l)}}\;\right)\cap 
L^p_{loc}({\cal N}_T),\quad 
\widetilde{f}_1\in C^\infty\left(\;\overline{\mathcal{WR}_1^{(t,\l)}}\;\right)\cap 
C^\infty\left(\;\overline{\mathcal{SR}_1^{(t,\l)}}\;\right)\cap 
L^p_{loc}({\cal N}_T).
$$ 
Applying the theory for PDEs, we deduce that the regularity results about ${\cal Q}_0$ and ${\cal Q}_1$. 

Note that $\mathcal{SR}_0^{(t,\l)}\subset{\cal N}_{T-T_1}$. So, we know that in the domain $[T-T_1,T]\times(0,+\infty)$, $\widetilde{f}_0= \e_0 \l -L_0$, and  ${\cal Q}_0$  satisfies that
\begin{equation*}
    \begin{cases}
    -\left(\partial_t {\cal Q}_0 +{\cal L}{\cal Q}_0 \right)=\e_0 \l -L_0&\mbox{in}\;\;[T-T_1,T]\times(0,+\infty);\vspace{2mm}\\
    {\cal Q}_0(T,\l) =0,\quad&\l\in(0,+\infty).
    \end{cases}
\end{equation*}
It is not difficult to check that ${\cal Q}_0$ takes the form of \eqref{expression-Q0}.

Let $\widetilde{f}^\varepsilon_0$ and $\widetilde{f}^\varepsilon_1$ be the mollification of $\widetilde{f}_0$ and $\widetilde{f}_1$,  respectively, which satisfy
$$
 \widetilde{f}^\varepsilon_0\rightarrow \widetilde{f}_0\;\mbox{in}\;\;L^p({\cal N}_T^n)\cap C^{1+\alpha/2,2+\alpha}\left(\;{\overline{{\cal N}_T^n\cap \mathcal{WR}_0^{(t,\l)}}}\;\right)\cap C^{1+\alpha/2,2+\alpha}\left(\;{\overline{{\cal N}_T^n\cap \mathcal{SR}_0^{(t,\l)}}}\;\right);
$$
$$
  \widetilde{f}^\varepsilon_1\rightarrow \widetilde{f}_1\;\mbox{in}\;\;L^p({\cal N}_T^n)\cap C^{1+\alpha/2,2+\alpha}\left(\;{\overline{{\cal N}_T^n\cap \mathcal{WR}_1^{(t,\l)}}}\;\right)\cap C^{1+\alpha/2,2+\alpha}\left(\;{\overline{{\cal N}_T^n\cap \mathcal{SR}_1^{(t,\l)}}}\;\right)
$$
for any $n\in\mathbb{Z}_+$ and $p\geq1,0<\alpha<1$. Since $\L_0>\L_1$ and  \eqref{de:domain-t-lambda}, we know that $\widetilde{f}_0$ and $\widetilde{f}_1$ are convex, and
$$    
  \partial_{\l\l}\widetilde{f}_0,\;\partial_{\l\l}\widetilde{f}_1= 0\;\mbox{in}\;{\cal N}_T\setminus(\{(\tau,\L_0(\tau))\}\cup\{(\tau,\L_1(\tau))\}),\qquad 
  0\leq \partial_\lambda \widetilde{f}_0\leq\e_1,\;
  -r\zeta_1\leq \partial_\lambda \widetilde{f}_1\leq \e_1\;\mbox{in}\;\overline{{\cal N}}_T.
$$ 
So, we know that $\widetilde{f}^\varepsilon_0$ and $\widetilde{f}^\varepsilon_1$ satisfy
$$
 |\partial_\lambda \widetilde{f}^\varepsilon_0|,\;
 |\partial_\lambda \widetilde{f}^\varepsilon_1|\leq \e_1+r\zeta_1,\qquad 
 \partial_\lambda \widetilde{f}^\varepsilon_0,\;\partial_{\l\l}\widetilde{f}^\varepsilon_0,\;\partial_{\l\l}\widetilde{f}^\varepsilon_1\geq 0.
$$
Let $({\cal Q}^\varepsilon_0,{\cal Q}^\varepsilon_1)$ be the unique classical solution of the following PDE, 
\begin{equation*}
    \begin{cases}
    -\partial_t {\cal Q}^\varepsilon_0-{\cal L}{\cal Q}^\varepsilon_0=\widetilde{f}^\varepsilon_0;\;\; 
    -\partial_t {\cal Q}^\varepsilon_1-{\cal L}{\cal Q}^\varepsilon_1=\widetilde{f}^\varepsilon_1&\mbox{in}\;\;{\cal N}_T;\vspace{2mm}\\
    {\cal Q}^\varepsilon_0(T,\l) = {\cal Q}^\varepsilon_1(T,\l)=0,&\l\in(0,+\infty)
    \end{cases}
\end{equation*}

Applying the comparison principle to the PDEs for $\partial_{\l} {\cal Q}^\varepsilon_i, \partial_{\l\l} {\cal Q}^\varepsilon_i,\,i=0,1$, we can deduce that
$$
 |\partial_\lambda {\cal Q}^\varepsilon_0|,\;
 |\partial_\lambda {\cal Q}^\varepsilon_1|\leq (\e_1+r\zeta_1)/r,\qquad 
  \partial_\lambda {\cal Q}^\varepsilon_0,\;\partial_{\l\l}{\cal Q}^\varepsilon_0,\;\partial_{\l\l}{\cal Q}^\varepsilon_1\geq 0.
$$

Moreover, from the theory for PDEs, we know that 
$$
 {\cal Q}^\varepsilon_0\rightarrow {\cal Q}_0\;\mbox{in}\;\;W^{2,1}_p({\cal N}_T^n)\cap C^{2+\alpha/2,4+\alpha}\left(\;{\overline{{\cal N}_T^n\cap \mathcal{WR}_0^{(t,\l)}}}\;\right)\cap C^{2+\alpha/2,4+\alpha}\left(\;{\overline{{\cal N}_T^n\cap \mathcal{SR}_0^{(t,\l)}}}\;\right);
$$
$$
  {\cal Q}^\varepsilon_1\rightarrow {\cal Q}_1\;\mbox{in}\;\;W^{2,1}_p({\cal N}_T^n)\cap  C^{2+\alpha/2,4+\alpha}\left(\;{\overline{{\cal N}_T^n\cap \mathcal{WR}_1^{(t,\l)}}}\;\right)\cap C^{2+\alpha/2,4+\alpha}\left(\;{\overline{{\cal N}_T^n\cap \mathcal{SR}_1^{(t,\l)}}}\;\right)
$$
for any $n\in\mathbb{Z}_+$ and $p\geq1,0<\alpha<1$. Hence, we know that
$$
 |\partial_\lambda {\cal Q}_0|,\;
 |\partial_\lambda {\cal Q}_1|\leq C,\qquad 
  \partial_\lambda {\cal Q}_0,\;\partial_{\l\l}{\cal Q}_0,\;\partial_{\l\l}{\cal Q}_1\geq 0.
$$

Moreover, from transformation \eqref{tran-P-u} and Theorem \ref{theorem for free boundaires}, we know that 
$$
 \lim\limits_{\l\rightarrow (\L_1(t))^+}\partial_{\l\l} ({\cal Q}_1-{\cal Q}_0)(t,\l)= \lim\limits_{\l\rightarrow (\L_1(t))^+}\partial_{\l\l} {\cal P}(t,\l)=
 \lim\limits_{x\rightarrow (x_1(\tau))^+}\frac{1}{\l}(\partial_{xx} u+\partial_{x} u)(\tau,x)>0,\,\forall\,t\in[0,T),
$$
$$
 \lim\limits_{\l\rightarrow (\L_0(t))^-}\partial_{\l\l} ({\cal Q}_1-{\cal Q}_0)(t,\l)=
 \lim\limits_{x\rightarrow (x_0(\tau))^-}\frac{1}{\l}(\partial_{xx} u+\partial_{x} u)(\tau,x)<0,\,\forall\,t\in[0,T-T_1).
$$
So, we know that 
$$
  \lim\limits_{\l\rightarrow (\L_1(t))^+}\partial_{\l\l} {\cal Q}_1(t,\l)>0,\,\forall\,t\in[0,T),\qquad 
  \lim\limits_{\l\rightarrow (\L_0(t))^-}\partial_{\l\l} {\cal Q}_0(t,\l)>0,\,\forall\,t\in[0,T-T_1).
$$
From \eqref{expression-Q0}, we know that $\partial_{\l\l} {\cal Q}_0=0$ in $[T-T_1,T]\times(0,+\infty)$. Hence, in the domain $\mathcal{WR}_0^{(t,\l)}\cap{\cal N}_{T-T_1},\,\partial_{\l\l} {\cal Q}_0$ satisfies the following PDE,
\begin{equation*}
    \begin{cases}
    &-\partial_t (\partial_{\l\l} {\cal Q}_0)-\dfrac{\t^2}{2}\l^2\partial_{\l\l}(\partial_{\l\l} {\cal Q}_0) - (\beta -r+2\t^2 )\l\partial_{\l}(\partial_{\l\l} {\cal Q}_0)- (\beta -2r+\t^2 )(\partial_{\l\l} {\cal Q}_0)=0; \vspace{2mm}\\
    &(\partial_{\l\l} {\cal Q}_0)(T-T_1,\l) =0,\;\forall\;\l\in[0,+\infty);\qquad 
    (\partial_{\l\l} {\cal Q}_0)(t,\L_0(t))>0,\;\forall\,t\in[0,T-T_1).
    \end{cases}
\end{equation*}
Applying the strong principle for PDEs to the above PDE in the domain $\mathcal{WR}_0^{(t,\l)}\cap{\cal N}_{T-T_1}$, we deduce that $\partial_{\l\l} {\cal Q}_0>0$ in $\mathcal{WR}_0^{(t,\l)}$. Repeating the similar argument for $\partial_{\l\l} {\cal Q}_1$ in the domain $\mathcal{WR}_1^{(t,\l)}$, we can deduce that $\partial_{\l\l} {\cal Q}_1>0$ in $\mathcal{WR}_1^{(t,\l)}$. 

Moreover, from transformation \eqref{tran-P-u} and Theorem \ref{theorem for free boundaires}, we know that 
$$
 ({\cal Q}_1-{\cal Q}_0)(t,x)={\cal P}(t,x)=
 (\partial_{xx} u+\partial_{x} u)(\tau,x)=0\;\;\mbox{in}\;\;
 \mathcal{SR}_0^{(t,\l)}\cup\mathcal{SR}_1^{(t,\l)}.
$$
So, we deduce Conclusion (d).
\end{proof}

\begin{lem}\label{lem:limit-Q0Q1}
    The following limiting behaviors of $\partial_{\l}{\cal Q}_0$ and $\partial_{\l}{\cal Q}_1$ hold:
    \begin{footnotesize}
        \begin{equation}
        \begin{aligned}
            &\lim_{\l \to 0^+}\partial_{\l}{\cal Q}_0(t,\l)= \dfrac{1-e^{-r(T-t)}}{r}\e_0,\;\;
           \lim_{\l \to +\infty}\partial_{\l}{\cal Q}_1(t,\l)= \dfrac{1-e^{-r(T-t)}}{r}\e_1,\;\;
           \lim_{\l \to 0^+}\partial_{\l}{\cal Q}_1(t,\l)= \dfrac{1-e^{-r(T-t)}}{r}\e_0-\zeta_1,\\
           &\lim_{\l \to +\infty}\partial_{\l}{\cal Q}_0(t,\l)=\begin{cases}
             \dfrac{1-e^{-r(T-t)}}{r}\e_1-\zeta_0\;\;&\mbox{for}\;\;t\in[0,T-T_1);\\
            \dfrac{1-e^{-r(T-t)}}{r}\e_0\;\;&\mbox{for}\;\;t\in[T-T_1,T].
        \end{cases}
        \end{aligned}
    \end{equation}
    \end{footnotesize}
\end{lem}
\begin{proof}
From Theorem \ref{property-Q}, we know that $\partial_\l{\cal Q}_0$ and $\partial_\l{\cal Q}_1$ satisfy
\begin{equation*}
    \begin{cases}
    &-\partial_t (\partial_{\l} {\cal Q}_0)-{\cal L}_\l^\prime(\partial_{\l} {\cal Q}_0)-\e_0=0\;\;\mbox{in}\;\;
    [0,T)\times(0,e^{X_2});\vspace{2mm}\\
    &-\partial_t (\partial_{\l} {\cal Q}_1)-{\cal L}_\l^\prime(\partial_{\l} {\cal Q}_1)-\e_1=0\;\;\mbox{in}\;\;
    [0,T)\times(e^{X_1},+\infty);\vspace{2mm}\\
    & |\partial_{\l}{\cal Q}_0(t,e^{X_2})|, |\partial_{\l}{\cal Q}_1(t,e^{X_1})|\leq C,\;\;t\in[0,T];\qquad \partial_{\l}{\cal Q}_0(T,\l)=\partial_{\l}{\cal Q}_1(T,\l)=0,\;\;\l\in(0,+\infty),
    \end{cases}
\end{equation*}
where 
$$
{\cal L}_\l^\prime:=\dfrac{\t^2}{2}\l^2\partial_{\l\l}+ (\beta -r+\t^2 )\l\partial_{\l}-r.
$$
Temporarily denote 
$$
  \overline{v}_0=\dfrac{1-e^{-r(T-t)}}{r}\e_0
  +C e^{-k_+X_2}x^{k_+},\qquad 
  \underline{v}_0=\dfrac{1-e^{-r(T-t)}}{r}\e_0
  -\left(C+\dfrac{\e_0}{r}\right) e^{-k_+X_2}x^{k_+},
$$
and
$$
  \overline{v}_1=\dfrac{1-e^{-r(T-t)}}{r}\e_1
  +C e^{-k_-X_1}x^{k_-},\qquad 
  \underline{v}_1=\dfrac{1-e^{-r(T-t)}}{r}\e_1
  -\left(C+\dfrac{\e_1}{r}\right) e^{-k_-X_1}x^{k_-},
$$
where $k_+$ and $k_-$ are respectively the positive and negative characteristic roots of ${\cal L}_\l^\prime$, i.e., the positive and negative roots of the following algebra equation
$$
 \dfrac{\t^2}{2}k(k-1)+ (\beta -r+\t^2 )k-r=0.
$$

It is not difficult to check that 
\begin{equation*}
    \begin{cases}
    &-\partial_t \overline{v}_0-{\cal L}_\l^\prime \overline{v}_0-\e_0=0,\quad -\partial_t \underline{v}_0-{\cal L}_\l^\prime \underline{v}_0-\e_0=0 \;\;\mbox{in}\;\;
    [0,T)\times(0,e^{X_2});\vspace{2mm}\\
    &-\partial_t \overline{v}_1-{\cal L}_\l^\prime \overline{v}_1-\e_1=0,\quad -\partial_t \underline{v}_1-{\cal L}_\l^\prime \underline{v}_1-\e_1=0 \;\;\mbox{in}\;\;
    [0,T)\times(e^{X_1},+\infty);\vspace{2mm}\\
    & \overline{v}_0(t,e^{X_2}),\, \overline{v}_1(t,e^{X_1})\geq C,\quad 
    \underline{v}_0(t,e^{X_2}),\, \underline{v}_1(t,e^{X_1})\leq -C,\;\;t\in[0,T];\vspace{2mm}\\
    & \overline{v}_0(T,\l),\, \overline{v}_1(T,\l)\geq 0,\qquad 
    \underline{v}_0(T,\l),\, \underline{v}_1(T,\l)\leq 0,\;\;\l\in(0,+\infty).
    \end{cases}
\end{equation*}
So, the comparison theory for PDEs implies that 
$$
 \underline{v}_0\leq {\cal Q}_0\leq \overline{v}_0\;\;\mbox{in}\;\;
 [0,T)\times(0,e^{X_2}),\quad  
 \underline{v}_1\leq {\cal Q}_1\leq \overline{v}_1\;\;\mbox{in}\;\;
 [0,T)\times(e^{X_1},+\infty),
$$
and 
$$
 \lim_{\l \to 0^+}\partial_{\l}{\cal Q}_0(t,\l)= \dfrac{1-e^{-r(T-t)}}{r}\e_0,\qquad 
 \lim_{\l \to +\infty}\partial_{\l}{\cal Q}_1(t,\l)= \dfrac{1-e^{-r(T-t)}}{r}\e_1.
$$
Since $\partial_{\l}{\cal Q}_1(t,\l)=\partial_{\l}{\cal Q}_0(t,\l)-\zeta_1 \;\;\mbox{for}\;\;0<\l \le \L_1(t)$, it is clear that 
\begin{equation*}
    \lim_{\l \to 0+}\partial_{\l}{\cal Q}_1(t,\l)= \dfrac{1-e^{-r(T-t)}}{r}\e_0-\zeta_1.
\end{equation*}

From the expression of ${\cal Q}_0$ in Conclusion (d) in Theorem \ref{property-Q}, we know that 
$$
 \lim_{\l \to \infty}\partial_{\l}{\cal Q}_0(t,\l)=\dfrac{1-e^{-r(T-t)}}{r}\e_0,\;\;\forall\;\;t\in[T-T_1,T].
$$
Moreover, since $\L_0(t)<+\infty$ for $t\in[0,T-T_1)$,  we know that
$\partial_\l{\cal Q}_0(t,\l)=\partial_\l{\cal Q}_1(t,\l)-\zeta_0$ for $\l \ge \L_0(t),t\in[0,T-T_1)$, and
$$\lim_{\l \to \infty}\partial_{\l}{\cal Q}_0(t,\l)=\dfrac{1-e^{-r(T-t)}}{r}\e_1-\zeta_0,\;\;\forall\;\;t\in[0,T-T_1).
$$
\end{proof}

\section{Optimal Strategies}\label{sec:strategy}

\subsection{Verification Theorem}

By Theorems \ref{th-relationship-Q-and-P}, \ref{th-u} and \ref{property-Q}, the PDE in \eqref{de:Q0-and-Q1} has a unique strong solution $({\cal Q}_0,{\cal Q}_1)$ such that ${\cal Q}_0,{\cal Q}_1\in W^{1,2}_{p,loc}({\cal N}_T)\cap C^{0,1}(\overline{\cal N}_T)$ with any $p\geq1$.  Moreover, $({\cal Q}_0, {\cal Q}_1)$ satisfies the system of VIs in \eqref{eq:HJB-1} and \eqref{eq:HJB-2}.

Now, we will demonstrate that the unique solutions $Q_0(0,\lambda)$ and $Q_1(0,\lambda)$ are equivalent to the solutions $J_S(0,\lambda)$ and $J_S(1,\lambda)$ of the optimal switching problem. To establish this equivalence, we first describe a candidate for the optimal job-switching strategy using the two free boundaries, $\L_0$ and $\L_1$.

Recall that the definition of four regions in Theorem \ref{property-Q}.
\begin{align*}
    \mathcal{WR}_0^{(t,\l)}=&\{(t,\l)\in{\cal N}_T:0<\l < \L_0(t)\},\qquad 
        \mathcal{SR}_0^{(t,\l)}=\{(t,\l)\in{\cal N}_T: \l\geq \L_0(t)\},\vspace{2mm}\\
        \mathcal{WR}_1^{(t,\l)} =&  \{(t,\l)\in{\cal N}_T:\;\l>\L_1(t)\},\;\;\;\;\;\;\;\;\;\quad
        \mathcal{SR}_1^{(t,\l)}=\{(t,\l)\in{\cal N}_T:  0<\l\le \L_1(t)\}.
\end{align*}

If the agent has an initial point $(0,\lambda)$ located within the region $\mathcal{SR}_j^{(t,\l)}$ ($j=0,1$), the agent immediately switches to job ${\cal D}_{1-j}$ and incurs a cost of $\zeta_j$. When the initial point $(0, \l)$ is located within the $\mathcal{WR}_1^{(t,\l)}$ region, as the dual process $\mathcal{Y}_t^\lambda$ decreases (indicating an increase in wealth) and touches the free boundary $\Lambda_1(t)$, the agent promptly switches to job ${\cal D}_0$, incurring a cost of $\zeta_1$ in order to reduce the disutility associated with labor. On the contrary, if the agent's initial point $(0, \l)$ is situated within $\mathcal{WR}_0^{(t,\l)}$, when the process ${\cal Y}_t^\lambda$ increases (indicating a decrease in wealth) and touches the free boundary $\Lambda_1(t)$, the agent switches to job $\mathcal{D}_1$, incurring a cost of $\zeta_0$ for the sake of higher income. However, if the time remaining until the mandatory retirement date is less than $T_1$, the agent {jobbing ${\cal D}_1$} never switches jobs. This is because the job cost exceeds the increase in the present value of income in such a scenario.

Now, let us proceed to express the explanation above in a mathematically rigorous manner. For a given initial condition $(j,\lambda)\in\{0,1\}\times(0,\infty)$, let us denote $\widehat{\Upsilon}^{(j,\l)}$ by
\begin{equation}
    \label{eq:optimal_job_state}
	\widehat{\Upsilon}_t^{(j,\l)}:= j{\bf 1}_{[0,\tau_1]}(t) + \sum_{n\ge 1} \left(1-\widehat{\Upsilon}^{(j,\l)}_{\tau_{n}}\right){\bf 1}_{(\tau_{n},\tau_{n+1}]}(t),
\end{equation}
where  $\widehat{\Upsilon}^{(j,\l)}_{\tau_{n+1}} \ne \widehat{\Upsilon}^{(j,\l)}_{\tau_{n}}  $ for $n\ge 0$ with $\tau_0=0$,
\begin{equation*}
	\tau_1 = \left\{\begin{array}{l} T\wedge\inf\{s > 0 : {\cal Y}_s^\l \ge  \Lambda_0(s)\} ~~\mbox{if}~ j = 0 ~\mbox{and}~ \lambda < \Lambda_0(0),\\
		0 ~~\mbox{if}~ j = 0 ~\mbox{and}~ \lambda \geq \Lambda_0(0),\\
		0 ~~\mbox{if}~ j = 1 ~\mbox{and}~ \lambda \leq \Lambda_1(0),\\
		T\wedge\inf\{s > 0 : {\cal Y}_s^\l \le \Lambda_1(s)\}~~\mbox{if}~ j = 1 ~\mbox{and}~ \lambda > \Lambda_1(0),\\ \end{array}\right.
\end{equation*}
and
\begin{equation*}
	\tau_{n+1} = \left\{\begin{array}{l} T\wedge\inf\{s > \tau_{n} : {\cal Y}_s^\l \ge \Lambda_0(s)\} ~~\mbox{if}~ \widehat{\Upsilon}^{(j,\l)}_{\tau_{n}} = 0\\
		T\wedge\inf\{s > \tau_{n} : {\cal Y}_s^\l \le  \Lambda_1(s)\}~~\mbox{if}~ \widehat{\Upsilon}^{(j,\l)}_{\tau_{n}} = 1\\ \end{array}\right. ~~\mbox{for}~~n \ge 1.
\end{equation*}
\vspace{2mm}

In the next lemma, we show that the constructed job-switching strategy $\widehat{\Upsilon}^{(j,\l)}$ belongs to the class $\Phi_j$.
\begin{lem}\label{lem:finite_number}
For given $j\in\{0,1\}$ and $\l_1>\l>\l_2>0$, 
    $$\widehat{\Upsilon}^{(j,\l)}\in \Phi_j.$$
    Moreover, there exists a constant $C_{\l_1}$ such that 
    $$   \mathbb{E}\left[\sup\limits_{\l\in[\l_2,\l_1]}\sum_{0\le t\le T}e^{-\b t}\left[\zeta_0{\cal Y}_t^\l (\D \widehat{\Upsilon}_t^{(j,\l)})^++\zeta_1{\cal Y}_t^\l (\D \widehat{\Upsilon}_t^{(j,\l)})^-\right]\right]\le C_{\l_1}.
    $$
\end{lem}
\begin{proof}
    For given $j\in\{0,1\}$ and $\l\in[\l_2,\l_1]$, let us temporarily denote $\eta^{(j,\l)}$ by
    $$
    \eta_t^{(j,\l)}:= j{\bf 1}_{[0,\hat{\tau}_1]}(t) + \sum_{n\ge 1} \left(1-\eta^{(j,\l)}_{\hat{\tau}_{n}}\right){\bf 1}_{(\hat{\tau}_{n},\hat{\tau}_{n+1}]}(t)
    $$ 
    where $\eta_{\hat{\tau}_{n+1}}^{(j,\l)}\ne \eta_{\hat{\tau}_{n}}^{(j,\l)}$ for $n\ge 0$, $\hat{\tau}_0=0$, 
    \begin{equation*}
        \hat{\tau}_1=
        \left\{\begin{array}{l} T\wedge\inf\{s > 0 : {\cal Y}_s^\l \ge  e^{X_2}\} ~~\mbox{if}~ j = 0 ~\mbox{and}~ \lambda < e^{X_2},\\
		0 ~~\mbox{if}~ j = 0 ~\mbox{and}~ \lambda \geq e^{X_2},\\
		0 ~~\mbox{if}~ j = 1 ~\mbox{and}~ \lambda \leq e^{X_1},\\
		T\wedge\inf\{s > 0 : {\cal Y}_s^\l \le e^{X_1}\}~~\mbox{if}~ j = 1 ~\mbox{and}~ \lambda > e^{X_1},\\ \end{array}\right.
    \end{equation*}
    and
\begin{equation*}
	\hat{\tau}_{n+1} = \left\{\begin{array}{l} T\wedge\inf\{s > \hat{\tau}_{n} : {\cal Y}_s^\l \ge e^{X_2}\} ~~\mbox{if}~ \eta^{(j,\l)}_{\hat{\tau}_{n}} = 0\\
		T\wedge\inf\{s > \hat{\tau}_{n} : {\cal Y}_s^\l \le  e^{X_1}\}~~\mbox{if}~ \eta^{(j,\l)}_{\hat{\tau}_{n}} = 1\\ \end{array}\right. ~~\mbox{for}~~n \ge 1.
\end{equation*}
Recall that $X_1$ and $X_2$ are defined in \eqref{de:X1X2}.

For each the sequence of stopping time $\{\tau_n\}_{n\ge0}$ and $\{\hat{\tau}_n\}_{n\ge0}$, let us define the number of switching $N$ and $\widehat{N}$ by
$$
N=\inf\{n:\;\tau_n =T\}\;\;\mbox{and}\;\;\widehat{N}=\inf\{n:\;\hat{\tau}_n=T\},
$$
respectively.

Note that 
\begin{align*}
     &\mathbb{E}\left[\sup\limits_{\l\in[\l_2,\l_1]}\sum_{0\le t\le T}e^{-\b t}\left[\zeta_0{\cal Y}_t^\l (\D \widehat{\Upsilon}_t^{(j,\l)})^++\zeta_1{\cal Y}_t^\l (\D \widehat{\Upsilon}_t^{(j,\l)})^-\right]\right]\\
   = &\mathbb{E}\left[\sup\limits_{\l\in[\l_2,\l_1]}\sum_{0\le t\le T}e^{-\b t}\left[\zeta_0{\cal Y}_t^\l (\D \widehat{\Upsilon}_t^{(j,\l)})^++\zeta_1{\cal Y}_t^\l (\D \widehat{\Upsilon}_t^{(j,\l)})^-\right]\mathbb{E}\left[e^{-\frac{\t^2}{2}(T-t)-\t (B_T-B_t)}\right]\right]\\
   =&\mathbb{E}\left[\sup\limits_{\l\in[\l_2,\l_1]}\sum_{0\le t\le T}e^{-\b t}\left[\zeta_0{\cal Y}_t^\l (\D \widehat{\Upsilon}_t^{(j,\l)})^++\zeta_1{\cal Y}_t^\l (\D \widehat{\Upsilon}_t^{(j,\l)})^-\right]\mathbb{E}\left[e^{-\frac{\t^2}{2}(T-t)-\t (B_T-B_t)}\mid {\cal F}_t\right]\right]\\
   =&\mathbb{E}\left[\sup\limits_{\l\in[\l_2,\l_1]}\l\sum_{0\le t\le T}e^{-r t}\mathbb{E}\left[e^{-\frac{\t^2}{2}T-\t B_T}\left[\zeta_0 (\D \widehat{\Upsilon}_t^{(j,\l)})^++\zeta_1  (\D \widehat{\Upsilon}_t^{(j,\l)})^-\right]\mid {\cal F}_t\right]\right]\\
   =&\mathbb{E}\left[e^{-\frac{\t^2}{2}T-\t B_T}\sup\limits_{\l\in[\l_2,\l_1]}\l\sum_{0\le t\le T}e^{-r t} \left[\zeta_0 (\D \widehat{\Upsilon}_t^{(j,\l)})^++\zeta_1  (\D \widehat{\Upsilon}_t^{(j,\l)})^-\right] \right].
\end{align*}

Let us define an equivalent martingale measure $\mathbb{Q}$ by
\begin{equation}
    \dfrac{d\mathbb{Q}}{d\mathbb{P}} = e^{-\frac{\t^2}{2}T-\t B_T}.
\end{equation}
By Girsanov's theorem , $B_t^\mathbb{Q}:=B_t+\t t$ for $t\in[0,T]$ is the standard Brownian motion under the $\mathbb{Q}$-measure.

Then, we have
\begin{align*}
   &\mathbb{E}\left[\sup\limits_{\l\in[\l_2,\l_1]}\sum_{0\le t\le T}e^{-\b t}\left[\zeta_0{\cal Y}_t^\l (\D \widehat{\Upsilon}_t^{(j,\l)})^++\zeta_1{\cal Y}_t^\l (\D \widehat{\Upsilon}_t^{(j,\l)})^-\right]\right]\\
   \le & \mathbb{E}^\mathbb{Q}\left[\sup\limits_{\l\in[\l_2,\l_1]}\l\sum_{0\le t\le T}e^{-r t}\left[\zeta_0 (\D \widehat{\Upsilon}_t^{(j,\l)})^++\zeta_1  (\D \widehat{\Upsilon}_t^{(j,\l)})^-\right]\right] \\
   \le &\mathbb{E}^\mathbb{Q}\left[\sup\limits_{\l\in[\l_2,\l_1]}\l \left[(\zeta_0+\zeta_1)N\right]\right].
\end{align*}

Since $\Lambda_1(t)\le e^{X_1} < e^{X_2} \le \Lambda_0(t)$ for all $t\in[0,T-T_1)$ and $\Lambda_0(t)=\infty$ for $t\in[T-T_1,T]$, it is clear that 
$$
\mathbb{E}^\mathbb{Q}[N]\le \mathbb{E}^\mathbb{Q}[\widehat{N}].
$$
Note that 
$$
 {\cal Y}_t^\l = \l e^{(\b-r-\frac{1}{2}\t^2)t - \t B_t}=\l e^{(\b-r+\frac{1}{2}\t^2)t -\t B_t^\mathbb{Q}}=\l e^{(\b-r+\t^2)t}e^{-\frac{1}{2}\t^2 t-\t B_t^\mathbb{Q}}\,.
$$
Since
$$
\mathbb{E}_t^\mathbb{Q}\left[{\cal Y}_s^\l\right]={\cal Y}_t^\l\mathbb{E}_t^\mathbb{Q}\left[ e^{(\beta-r+\theta^2)(s-t)}\exp\left(-\t \left(B^\mathbb{Q}_s-B^\mathbb{Q}_t\right)-\frac{\t^2}{2}(s-t)\right)\right]
=e^{(\beta-r+\theta^2)(s-t)}{\cal Y}_t^\l,\;\forall\;0\leq t\leq s\leq T,
$$
the dual process ${\cal Y}^\l$ is a $\mathbb{F}$-sub-martingale when $(\b-r+\t^2)>0$ and a $\mathbb{F}$-super-martingale when $(\b-r+\t^2)\le 0$. 

By substituting ${\cal Y}^\l$ when $(\b-r+\t^2)>0$ and $-{\cal Y}^\l$ when $(\b-r+\t^2)\le 0$ into Theorem 3.8 on pp. 13-14 in \citet{KS2} instead of $X$, we easily deduce that 
\begin{equation*}
    \mathbb{E}^\mathbb{Q}[\widehat{N}]\le  \dfrac{2\mathbb{E}^\mathbb{Q}[{\cal Y}_T^\l]+e^{X_2}}{e^{X_2}-e^{X_1}} \le  \dfrac{2\mathbb{E}^\mathbb{Q}[{\cal Y}_T^{\l_1}]+e^{X_2}}{e^{X_2}-e^{X_1}} <\infty.
\end{equation*}
Hence, we can conclude that $\mathbb{E}^\mathbb{Q}[N]<\infty, N<\infty$ a.s. in $\Omega$, and $\widehat{\Upsilon}^{(j,\l)}\in\Phi_j$. Moreover, 
$$   \mathbb{E}\left[\sup\limits_{\l\in[\l_2,\l_1]}\sum_{0\le t\le T}e^{-\b t}\left[\zeta_0{\cal Y}_t^\l (\D \widehat{\Upsilon}_t^{(j,\l)})^++\zeta_1{\cal Y}_t^\l (\D \widehat{\Upsilon}_t^{(j,\l)})^-\right]\right]<C_{\l_1}.
$$
\end{proof}

By the construction of $\widehat{\Upsilon}$, we deduce that for given $\l>0$ 
\begin{equation}\label{eq:properties-candi-Upsilon}
   \begin{aligned}
        &\partial_t {\cal Q}_{\widehat{\Upsilon}_t^{(j,\l)}}(t, {\cal Y}_t^\l) + {\cal L}{\cal Q}_{\widehat{\Upsilon}_t^{(j,\l)}}(t,{\cal Y}_t^\l)+\e_{\widehat{\Upsilon}_t^{(j,\l)}} {\cal Y}_t^\l - L_{\widehat{\Upsilon}_t^{(j,\l)}} =0,\vspace{1mm}\\
    &\left[{\cal Q}_{\widehat{\Upsilon}_{t+}^{(j,\l)}}(t,{\cal Y}_t^\l)-{\cal Q}_{\widehat{\Upsilon}_{t}^{(j,\l)}}(t,{\cal Y}_t^\l)-\zeta_0 {\cal Y}_t^\l\right]\Delta (\widehat{\Upsilon}_t^{(j,\l)})^+ =0,\vspace{1mm}\\
    &\left[{\cal Q}_{\widehat{\Upsilon}_{t+}^{(j,\l)}}(t,{\cal Y}_t^\l)-{\cal Q}_{\widehat{\Upsilon}_{t}^{(j,\l)}}(t,{\cal Y}_t^\l)-\zeta_1 {\cal Y}_t^\l\right]\Delta (\widehat{\Upsilon}_t^{(j,\l)})^- =0,
   \end{aligned}
\end{equation}
for all $t\in[0,T]$, $\mathbb{P}$-a.s..

Now, we can state a verification theorem that provides the proof that ${\cal Q}_j$, which is a solution to the system of VIs in \eqref{eq:HJB-1} and \eqref{eq:HJB-2}, is indeed the same as the solution to the optimal switching problem in \eqref{eq:OSP}.
\begin{theorem}[Verification Theorem]\label{verification theorem}
    For $j=0,1$ and $\l>0$, 
    $$
    J_S(j, \l) = {\cal Q}_j (0,\l),\;\;\;\forall\;\l>0,
    $$
    where $({\cal Q}_0,{\cal Q}_1)$ is the unique strong solution of the PDE in \eqref{de:Q0-and-Q1}. Moreover, the job indicator process $\widehat{\Upsilon}_t^{(j,\l)}$ defined in \eqref{eq:optimal_job_state} is the solution to the optimal switching problem in \eqref{eq:OSP}, i.e.,
    \begin{align*}
        J_S(j,\l) =& \mathbb{E}\left[\int_0^T e^{-\b t}\left(\left({\cal Y}_t^\l \e_0-L_0\right)(1-\widehat{\Upsilon}_t^{(j,\l)})+\left({\cal Y}_t^\l \e_1-L_1 \right)\widehat{\Upsilon}_t^{(j,\l)} \right)dt\right]\\-& \mathbb{E}\left[\sum_{0 \leq t \le T}e^{-\b t}[\z_0{\cal Y}_t^\l (\Delta \widehat{\Upsilon}_t^{(j,\l)})^+ + \z_1{\cal Y}_t^{\l}(\Delta \widehat{\Upsilon}_t^{(j,\l)})^- ]\right]
    \end{align*}
\end{theorem}
\begin{proof}

Let us denote $\Gamma_t$ by 
\begin{equation}
\begin{aligned}
        \Gamma_t:=&\int_0^t e^{-\b s}\left(\left({\cal Y}_s^\l \e_0-L_0\right)(1-\Upsilon_s)+\left({\cal Y}_s^\l \e_1-L_1 \right)\Upsilon_s \right)ds\\
        -& \left(\zeta_0\sum_{0 \leq s \le t}e^{-\beta s}{\cal Y}_s^\l(\Delta \Upsilon_s^{\lambda})^+ +\zeta_1\sum_{0 \leq s\le t}e^{-\beta s}{\cal Y}_s^\l(\Delta \Upsilon_s^{\lambda})^-\right)+ e^{-\b t} {\cal Q}_{\Upsilon_t}(t,{\cal Y}_t^\l).
\end{aligned}
\end{equation}

Since ${\cal Q}_0,{\cal Q}_1\in  W^{1,2}_{p,loc}({\cal N}_T)\cap C^{0,1}(\overline{\cal N}_T)$ with any $p\geq1$, we can apply It\^{o}-Tanaka formula to $e^{-\b t}{\cal Q}(\Upsilon_t,{\cal Y}_t^\l)$ for any $\Upsilon\in \Psi_j$. Hence, we deduce that 
\begin{equation}
    \begin{aligned}
        d\Gamma_t=& e^{-\b t}\left(\left({\cal Y}_t^\l \e_0-L_0\right)(1-\Upsilon_t)+\left({\cal Y}_t^\l \e_1-L_1 \right)\Upsilon_t \right)dt+e^{-\b t}d{\cal Q}_{\Upsilon_t}(t,{\cal Y}_t^\l) - \b {\cal Q}_{\Upsilon_t}(t,{\cal Y}_t^\l)dt \\
        +&e^{-\b t}\left[{\cal Q}_{\Upsilon_{t+}}(t,{\cal Y}_t^\l)-{\cal Q}_{\Upsilon_{t}}(t,{\cal Y}_t^\l)-\zeta_0{\cal Y}_t^\l\right]{\rm\bf I}_{\{\D \Upsilon_t =1\}}+e^{-\b t}\left[{\cal Q}_{\Upsilon_{t+}}(t,{\cal Y}_t^\l)-{\cal Q}_{\Upsilon_{t}}(t,{\cal Y}_t^\l)-\zeta_1{\cal Y}_t^\l\right]{\rm\bf I}_{\{\D \Upsilon_t =-1\}}\\
        =&e^{-\b t}\left(\partial_t{\cal Q}_{\Upsilon_t}(t,{\cal Y}_t^\l)+{\cal L}{\cal Q}_{\Upsilon_t}(t,{\cal Y}_t^\l)+ \e_{\Upsilon_t}{\cal Y}_t^\l - L_{\Upsilon_t}\right)dt - \t {\cal Y}_t^\l \partial_{\l}{\cal Q}_{\Upsilon_t}(t,{\cal Y}_t^\l)dB_t\\
        +&e^{-\b t}\left[{\cal Q}_{\Upsilon_{t+}}(t,{\cal Y}_t^\l)-{\cal Q}_{\Upsilon_{t}}(t,{\cal Y}_t^\l)-\zeta_0{\cal Y}_t^\l\right]{\rm\bf I}_{\{\D \Upsilon_t =1\}}+e^{-\b t}\left[{\cal Q}_{\Upsilon_{t+}}(t,{\cal Y}_t^\l)-{\cal Q}_{\Upsilon_{t}}(t,{\cal Y}_t^\l)-\zeta_1{\cal Y}_t^\l\right]{\rm\bf I}_{\{\D \Upsilon_t =-1\}}.
    \end{aligned}
\end{equation}
This yields that 
\begin{equation}
\begin{aligned}\label{eq:Ito}
    \Gamma_T=& {\cal Q}_j(0, \l) + \underbrace{\int_0^T e^{-\b t} \left(\partial_t{\cal Q}_{\Upsilon_t}(t,{\cal Y}_t^\l)+{\cal L}{\cal Q}_{\Upsilon_t}(t,{\cal Y}_t^\l)+\e_{\Upsilon_t}{\cal Y}_t^\l -L_{\Upsilon_t}\right)dt}_{\bf (I)}  \\
    +&\underbrace{\sum_{0\le t \le T}e^{-\b t}\left[{\cal Q}_{\Upsilon_{t+}}(t,{\cal Y}_t^\l)-{\cal Q}_{\Upsilon_{t}}(t,{\cal Y}_t^\l)-\zeta_0{\cal Y}_t^\l\right](\Delta \Upsilon_t)^+}_{\bf (II)} \\
    +&\underbrace{\sum_{0\le t \le T}e^{-\b t}\left[{\cal Q}_{\Upsilon_{t+}}(t,{\cal Y}_t^\l)-{\cal Q}_{\Upsilon_{t}}(t,{\cal Y}_t^\l)-\zeta_1{\cal Y}_t^\l\right](\Delta \Upsilon_t)^-}_{\bf (III)} -\underbrace{\t\int_0^T e^{-\b t}{\cal Y}_t^\l \partial_{\l}{\cal Q}_{\Upsilon_t}(t,{\cal Y}_t^\l)dB_t}_{\bf (IV)}.
\end{aligned}
\end{equation}

Since $({\cal Q}_0, {\cal Q}_1)$ satisfies the system of VIs \eqref{eq:HJB-1} and \eqref{eq:HJB-2}, 
\begin{equation}\label{eq:expectation1}
    \mathbb{E}\left[{\bf (I)}\right]\le 0,\;\mathbb{E}\left[{\bf (II)}\right]\le 0,\;\;\mbox{and}\;\;\mathbb{E}\left[{\bf (III)}\right]\le 0.
\end{equation}
Moreover, the boundedness of $\partial_\l {\cal Q}_0$ and $\partial_\l {\cal Q}_1$ in Theorem \ref{property-Q} implies that 
$$
\mathbb{E}\left[\int_0^T \left(\t e^{-\b t}{\cal Y}_t^\l \partial_\l {\cal Q}_{\Upsilon_t}(t,{\cal Y}_t^\l)\right)^2 dt\right]<\infty.
$$
Hence, 
\begin{equation}\label{eq:expectation2}
    \mathbb{E}\left[{\bf (IV)}\right]=0. 
\end{equation}

By taking the expectation on the both side of the equation \eqref{eq:Ito}, it follows from \eqref{eq:expectation1} and \eqref{eq:expectation2} that 
\begin{equation}\label{eq:inequality-sub-op}
    \begin{aligned}
           {\cal Q}_j(0,\l)\ge \mathbb{E}[\Gamma_T]=&\mathbb{E}\left[\int_0^T e^{-\b t}\left(\left({\cal Y}_t^\l \e_0-L_0\right)(1-\Upsilon_t)+\left({\cal Y}_t^\l \e_1-L_1 \right)\Upsilon_t \right)dt\right.\\
        -&\left. \left(\zeta_0\sum_{0 \leq t \le T}e^{-\beta t}{\cal Y}_t^\l(\Delta \upsilon_t^{\lambda})^+ +\zeta_1\sum_{0 \leq t\le T}e^{-\beta t}{\cal Y}_t^\l(\Delta \Upsilon_t^{\lambda})^-\right)\right],
    \end{aligned}
\end{equation}
where we have used the fact ${\cal Q}_j(T,\l)=0$ for all $\l>0$.

Since the inequality in \eqref{eq:inequality-sub-op} holds for all $\Upsilon\in\Psi_j$, we deduce that 
\begin{equation}\label{eq:Q>P}
    {\cal Q}_j(0,\l) \ge J_S(j,\l).
\end{equation}

On the other hand, when substituting $\widehat{\Upsilon}_t^{(j,\l)}\in \Phi_j$ (see Lemma \ref{lem:finite_number}) for $\Upsilon_t$ into \eqref{eq:Ito}, it follows from the properties of $\widehat{\Upsilon}^{(j,\l)}$ as stated in \eqref{eq:properties-candi-Upsilon} that
\begin{equation}
     \mathbb{E}\left[{\bf (I)}\right]=\mathbb{E}\left[{\bf (II)}\right]=\mathbb{E}\left[{\bf (III)}\right]=\mathbb{E}\left[{\bf (IV)}\right]= 0.
\end{equation}
This yields that 
\begin{equation}\label{eq:Q<P}
    \begin{aligned}
    {\cal Q}_j(0,\l) = & \mathbb{E}\left[\int_0^T e^{-\b t}\left(\left({\cal Y}_t^\l \e_0-L_0\right)(1-\widehat{\Upsilon}_t^{(j,\l)})+\left({\cal Y}_t^\l \e_1-L_1 \right)\widehat{\Upsilon}_t^{(j,\l)} \right)dt\right]\\-& \mathbb{E}\left[\sum_{0 \leq t \le T}e^{-\b t}[\z_0{\cal Y}_t^\l (\Delta \widehat{\Upsilon}_t^{(j,\l)})^+ + \z_1{\cal Y}_t^{\l}(\Delta \widehat{\Upsilon}_t^{(j,\l)})^- ]\right]\\
    \le& J_S(j,\l).
    \end{aligned}
\end{equation}
By two inequalities \eqref{eq:Q<P} and \eqref{eq:Q>P}, we conclude that for $j=0,1,$
$$J_S(j,\l)={\cal Q}_j(0,\l).$$
Consequently, $\widehat{\Upsilon}^{(j,\l)}$ is the optimal.
\end{proof}

For given $j\in\{0,1\}$, let us denote $\widehat{\cal Q}_j(t,\l)$ by 
\begin{equation}\label{eq:dual-t-value}
    \widehat{\cal Q}_j(t,\l):= {\cal Q}_R(t,\l) + {\cal Q}_j(t,\l).
\end{equation}
Note that for all $\l>0$
\begin{equation*}
    J(j,\l) = \widehat{\cal Q}_j(0,\l)
\end{equation*}
and
\begin{equation}\label{eq:dual-expectation}
    \begin{aligned}
        J(j,\l)=&\mathbb{E}\left[\int_0^T e^{-\b t}\left(\widetilde{U}_1(t,{\cal Y}_t^\l)+\left({\cal Y}_t^\l \e_0-L_0\right)(1-\widehat{\Upsilon}_t^{(j,\l)})+\left(\widetilde{U}_1(t,{\cal Y}_t^\l)+{\cal Y}_t^\l \e_1-L_1 \right)\widehat{\Upsilon}_t^{(j,\l)} \right)dt\right]\\+&\mathbb{E}[e^{-\b T}\widetilde{U}_2(T,{\cal Y}_T^\l)]- \mathbb{E}\left[\sum_{0 \leq t \le T}e^{-\b t}[\z_0{\cal Y}_t^\l (\Delta \widehat{\Upsilon}_t^{(j,\l)})^+ + \z_1{\cal Y}_t^{\l}(\Delta \widehat{\Upsilon}_t^{(j,\l)})^- ]\right].
    \end{aligned}
\end{equation}

From Proposition \ref{pro:unconstrained}, Theorem \ref{property-Q}, and Lemma \ref{lem:limit-Q0Q1}, we can summarize the properties of $\widehat{\cal Q}_0$ and $\widehat{\cal Q}_1$ in the following lemma.
\begin{lem}\label{lem:properties-Q}
    For $j\in\{0,1\}$, $\widehat{\cal Q}_j(t,\l)$ satisfies the following properties.
    \begin{itemize}
        \item[(a)] $\widehat{\cal Q}_0,\widehat{\cal Q}_1\in W^{1,2}_{p,loc}({\cal N}_T)\cap C^{0,1}(\overline{\cal N}_T)$ and 
$$\widehat{\cal Q}_0\in C^\infty\left(\;\overline{\mathcal{WR}_0^{(t,\l)}}\;\right)\cap C^\infty\left(\;\overline{\mathcal{SR}_0^{(t,\l)}}\;\right),\;\;
\widehat{\cal Q}_1\in C^\infty\left(\;\overline{\mathcal{WR}_1^{(t,\l)}}\;\right)\cap C^\infty\left(\;\overline{\mathcal{SR}_1^{(t,\l)}}\setminus\{(T,0)\}\;\right),$$
        \item[(b)] $\partial_{\l\l}\widehat{\cal Q}_0,\,\partial_{\l\l}\widehat{\cal Q}_1>0$ in ${\cal N}_T$.
        \item[(c)]
        $$
        \lim_{\l \to 0+}\partial_\l \widehat{\cal Q}_1(t,\l)=-\infty,\;\;\lim_{\l \to 0+}\partial_\l \widehat{\cal Q}_0(t,\l)=-\infty,\;\;\lim_{\l \to +\infty} \partial_{\l} \widehat{\cal Q}_1(t,\l)= \frac{\epsilon_1(1-e^{-r(T-t)})}{r}.
        $$
        and 
        $$
        \lim_{\l \to +\infty} \partial_{\l} \widehat{\cal Q}_0(t,\l)= \left[\frac{\epsilon_1(1-e^{-r(T-t)})}{r}-\zeta_0\right]{\bf 1}_{\{t\in[0,T-T_1)\}}+\left[\dfrac{\e_0(1-e^{-r(T-t)})}{r}\right]{\bf 1}_{\{t\in[T-T_1,T]\}}.
        $$
    \end{itemize}
\end{lem}

\subsection{Duality Theorem and Optimal Strategies}

In this subsection, we establish the duality theorem and derive the optimal switching, consumption, and investment strategies.

The following lemma is crucial for establishing the duality theorem.
\begin{lem}\label{lem:budget-equality} For given $j\in\{0,1\}$ and $\l>0$, the following equality holds.
    \begin{equation}
      \begin{aligned}
           -\partial_{\l}J(j,\l)=&\mathbb{E}\left[\int_0^T{\cal H}_t \left({\cal I}_1(t,{\cal Y}_t^\l) - \epsilon_0 (1-\widehat{\Upsilon}_t^{(t,\l)}) - \epsilon_1 \widehat{\Upsilon}_t^{(t,\l)}\right)dt +\sum_{0 \leq t \le T}[\z_0{\cal H}_t(\Delta \widehat{\Upsilon}_t^{(t,\l)})^+ + \z_1{\cal H}_t(\Delta \widehat{\Upsilon}_t^{(t,\l)})^- ]\right.\\&\;\;\;\;\;+\left. {\cal H}_T {\cal I}_2(T,{\cal Y}_T^\l) \right],
      \end{aligned}
    \end{equation}
    where ${\cal I}_i(t,\cdot)$ are the inverse function of $\partial_c U_i(t,\cdot)$, and $\widehat{\Upsilon}^{(t,\l)}$ is defined in \eqref{eq:optimal_job_state}.
\end{lem}
\begin{proof}

For given $j\in\{0,1\}$ and $\l>0$, let us denote ${\cal G}_j(\xi)$ by
\begin{equation*}
    {\cal G}_j(\xi):=J(j,\xi)-\partial_\l J(j,\l)\xi.
\end{equation*}
It follows that 
\begin{equation*}
    {\cal G}_j'(\xi)=\partial_\l J(j,\xi)-\partial_\l J(j,\l)\;\;\mbox{and}\;\;{\cal G}_j''(\xi)=\partial_{\l\l}J(j,\xi).
\end{equation*}
Since $J(j,\xi)$ is strictly convex in $\xi>0$ (see Lemma \ref{lem:properties-Q}), $\xi=\l>0$ is a unique minimizer of ${\cal G}_j(\xi)$, i.e., 
\begin{equation}\label{eq:minimizer}
    J(j,\xi)-\partial_\l J(j,\l)\xi \ge J(j,\l)-\partial_\l J(j,\l)\l\;\;\mbox{for all}\;\;\xi>0.
\end{equation}

For sufficiently small $\varepsilon>0$, let us denote $\l^{\pm\varepsilon}$ by 
\begin{equation*}
    \l^{\pm\varepsilon}:=\l \pm \varepsilon.
\end{equation*}
Since 
\begin{equation*}
    J(j,\l)= \sup_{\Upsilon \in \Phi_j}{\cal J}(j,\l;\Upsilon)={\cal J}(j,\l;\widehat{\Upsilon}^{(t,\l)}),
\end{equation*}
it follows from \eqref{eq:minimizer} that 
\begin{equation}
\begin{aligned}
    {\cal J}(j,\l^{\pm\varepsilon};\widehat{\Upsilon}^{(j,\l^{\pm\varepsilon})}) +(-\partial_{\l}J(j,\l))\l^{\pm\varepsilon}=&J(j,\l^{\pm\varepsilon}) +(-\partial_{\l}J(j,\l))\l^{\pm\varepsilon}\\
    \ge&J(j,\l) +(-\partial_{\l}J(j,\l))\l \\
    \ge&{\cal J}(j,\l;\widehat{\Upsilon}^{(j,\l^{\pm\varepsilon})}) +(-\partial_{\l}J(j,\l))\l.
\end{aligned}
\end{equation}
This yields that 
\begin{footnotesize}
\begin{equation*}
    \dfrac{{\cal J}(j,\l;\widehat{\Upsilon}^{(j,\l^{+\varepsilon})})-{\cal J}(j,\l^{+\varepsilon};\widehat{\Upsilon}^{(j,\l^{+\varepsilon})})}{\varepsilon}\le -\partial_{\l}J(j,\l)\;\;\mbox{and}\;\; \dfrac{{\cal J}(j,\l;\widehat{\Upsilon}^{(j,\l^{-\varepsilon})})-{\cal J}(j,\l^{-\varepsilon};\widehat{\Upsilon}^{(j,\l^{-\varepsilon})})}{-\varepsilon}\ge -\partial_{\l}J(j,\l).
\end{equation*}
\end{footnotesize}
That is,
\begin{footnotesize}
    \begin{equation*}
        \begin{aligned}
            0\le& \pm\left(-\partial_{\l}J(j,\l)-\mathbb{E}\left[\int_0^T{\cal H}_t\left(\dfrac{\widetilde{U}_1(t,{\cal Y}_t^\l)-\widetilde{U}_1(t,{\cal Y}_t^{\l^{\pm\varepsilon}})}{{\cal Y}_t^{\l}-{\cal Y}_t^{\l^{\pm\varepsilon}}}-\e_0 (1-\Upsilon_t^{(j,\l^{\pm\varepsilon})})-\e_1\Upsilon_t^{(j,\l^{\pm\varepsilon})}\right)dt\right.\right.\\+&\left.\left.{\cal H}_T\left(\dfrac{\widetilde{U}_2(T,{\cal Y}_T^\l)-\widetilde{U}_2(T,{\cal Y}_T^{\l^{\pm\varepsilon}})}{{\cal Y}_T^{\l}-{\cal Y}_T^{\l^{\pm\varepsilon}}}\right)+\zeta_0\sum_{0\le t\le T}{\cal H}_t(\D \Upsilon_t^{(j,\l^{\pm\varepsilon})})^+ +\zeta_1\sum_{0\le t\le T}{\cal H}_t(\D \Upsilon_t^{(j,\l^{\pm\varepsilon})})^-\right]\right). 
        \end{aligned}
    \end{equation*}
\end{footnotesize}

By Assumption \ref{as:utility}, there exist positive constants $\widehat{C}_1$ and $\widehat{C}_2$ such that 
$$
{\cal I}_1(t,\l)<\widehat{C}_1(1+\l^{-\frac{1}{\k_1}})\;\;\;\;\mbox{and}\;{\cal I}_2(t,\l)<\widehat{C}_2(1+\l^{-\frac{1}{\k_2}})\;\;\forall\;\;(t,\l)\in[0,T]\times(0,\infty).
$$
Moreover, it follows from Lemma \ref{lem:finite_number} that for any $\l>0$
\begin{equation*}    \mathbb{E}\left[\sup\limits_{\l\in[\l-\varepsilon,\l+\varepsilon]}\sum_{0\le t\le T}e^{-\b t}\left[\zeta_0{\cal Y}_t^\l (\D \widehat{\Upsilon}_t^{(j,\l)})^++\zeta_1{\cal Y}_t^\l (\D \widehat{\Upsilon}_t^{(j,\l)})^-\right]\right] <\infty.
\end{equation*}

Hence, letting $\varepsilon\downarrow 0$, the {\it dominated convergence theorem} implies that 
\begin{footnotesize}
    \begin{equation*}
  \begin{aligned}
      -\partial_{\l}J(j,\l)=&\mathbb{E}\left[\int_0^T{\cal H}_t \left({\cal I}_1(t,{\cal Y}_t^\l) - \epsilon_0 (1-\widehat{\Upsilon}_t^{(j,\l)}) - \epsilon_1 \widehat{\Upsilon}_t^{(j,\l)}\right)dt +{\cal H}_T {\cal I}_2(T,{\cal Y}_T^\l)\right.\\&\left.\;\;\;\;\;+\sum_{0 \leq t \le T}[\z_0{\cal H}_t(\Delta \widehat{\Upsilon}_t^{(j,\l)})^+ + \z_1{\cal H}_t(\Delta \widehat{\Upsilon}_t^{(j,\l)})^- ] \right]. 
  \end{aligned}
\end{equation*}
\end{footnotesize}
\end{proof}

Now, we are ready to state the duality theorem.
\begin{theorem}[Duality Theorem and Optimal Strategy]\label{thm:main}
Let $j \in \{0, 1\}$ and $w$ satisfying \eqref{eq:initial:as} be given. 
\begin{itemize}
    \item[(a)] The following duality relationship is established:
    \begin{equation}
        V(j, w) = \inf_{\l>0} \left[J(j,\l)+\l w\right].
    \end{equation}
    Moreover, there exists a unique $\l^*=\l^*(j,w,T)$ such that 
    \begin{equation}
         w= -\partial_{\l}J(j,\l^*)\;\;\mbox{and}\;\;V(j,w)=J(j,\l^*)+\l^* w.
    \end{equation}
    \item[(b)] The optimal job-switching $\Upsilon_t^{j,*}$, consumption $c_t^*$, and investment $\pi_t^*$ are given by 
    \begin{equation}
        \Upsilon_t^{j,*}=\widehat{\Upsilon}_t^{(j,\l)},\;\;c_t^*={\cal I}_1(t,{\cal Y}_t^{ *}),\;\;\mbox{and}\;\;\pi_t^*=\dfrac{\t}{\s}{\cal Y}_t^{^*} \partial_{\l\l}\widehat{\cal Q}_{\Upsilon_t^*}(t,{\cal Y}_t^*),
    \end{equation}
    where ${\cal Y}_t^*:={\cal Y}_t^{\l^*}$ and $\widehat{\cal Q}$ is defined in \eqref{eq:dual-t-value}.
    
    Moreover, the wealth $W_t^*=W_t^{c^*,\pi^*,\Upsilon^{j,*}}$ corresponding to the optimal strategey $(c^*,\pi^*,\Upsilon^{j,*})$ is given by 
    \begin{equation}
        W_t^* =-\partial_{\l}\widehat{\cal Q}_{\Upsilon_t^*}(t,{\cal Y}_t^*).
    \end{equation}
\end{itemize}
\end{theorem}
\begin{proof}
Since $J(j,\l)=\widehat{\cal Q}_j(0,\l)$, it follows from Lemma \ref{lem:properties-Q} that 
\begin{equation}
   \begin{aligned}
        &\lim_{\l \to \infty}\left(-\partial_{\l}J(0,\l)\right) = -\dfrac{\e_0(1-e^{-rT})}{r}+\zeta_0,\;\; \lim_{\l \to \infty}\left(-\partial_{\l}J(1,\l)\right) = -\dfrac{\e_1(1-e^{-rT})}{r},\\
        &\lim_{\l \to 0+}\left(-\partial_{\l}J(0,\l)\right) = \infty,\;\; \lim_{\l \to 0+}\left(-\partial_{\l}J(1,\l)\right) = \infty.
   \end{aligned}
\end{equation}
Since $J(j,\l)$ is strictly convex in $\l>0$, for given $w$ satisfying \eqref{eq:initial:as}, there exists a unique $\l^* = \l^*(j,w,T)$ such that 
$$
w=-\partial_{\l}J(j,\l^*).
$$
By Lemma \ref{lem:budget-equality}, we have 
\begin{equation}
    \begin{aligned}
        w=-\partial_{\l}J(j,\l^*)=&\mathbb{E}\left[\int_0^T{\cal H}_t \left({\cal I}_1(t,{\cal Y}_t^{\l^*}) - \epsilon_0 (1-\widehat{\Upsilon}_t^{(j,\l^*)}) - \epsilon_1 \widehat{\Upsilon}_t^{(j,\l^*)}\right)dt+{\cal H}_T {\cal I}_2(T,{\cal Y}_T^{\l^*})  \right.\\+&\left.\sum_{0 \leq t \le T}[\z_0{\cal H}_t(\Delta \widehat{\Upsilon}_t^{(j,\l^*)})^+ + \z_1{\cal H}_t(\Delta \widehat{\Upsilon}_t^{(j,\l^*)})^- ]\right]\\
        =&\mathbb{E}\left[\int_0^T{\cal H}_t \left(c_t^* - \epsilon_0 (1- {\Upsilon}_t^{j,*}) - \epsilon_1  {\Upsilon}_t^{j,*}\right)dt+{\cal H}_T {\cal I}_2(T,{\cal Y}_T^{*})  \right.\\+&\left.\sum_{0 \leq t \le T}[\z_0{\cal H}_t(\Delta  {\Upsilon}_t^{j,*})^+ + \z_1{\cal H}_t(\Delta  {\Upsilon}_t^{j,*})^- ]\right],
    \end{aligned}
\end{equation}
where ${\cal Y}_t^*={\cal Y}_t^{\l^*}$, $c_t^*={\cal I}_1(t,{\cal Y}_t^*)$, and $\Upsilon_t^{j,*}=\widehat{\Upsilon}_t^{(j,\l^*)}$.

By (b) in Proposition \ref{pro:static-budget-cost}, there exists an investment strategy $\pi_t^*$ such that $(c^*,\pi^*,\Upsilon^{j,*})\in{\cal A}(j,w)$. Moreover, the corresponding wealth $W_t^*=W_t^{c^*,\pi^*,\Upsilon^{j,*}}$ is given by 
\begin{footnotesize}
    \begin{equation}
    \begin{aligned}
        {\cal H}_tW_t^*=&\mathbb{E}\left[\int_t^T{\cal H}_s \left(c_s^* - \epsilon_0 (1- {\Upsilon}_s^{j,*}) - \epsilon_1  {\Upsilon}_s^{j,*}\right)ds+{\cal H}_T {\cal I}_2(T,{\cal Y}_T^{*}) +\sum_{t \leq s \le T}[\z_0{\cal H}_s(\Delta  {\Upsilon}_s^{j,*})^+ + \z_1{\cal H}_s(\Delta  {\Upsilon}_s^{j,*})^- ]\right].
    \end{aligned}
\end{equation}
\end{footnotesize}
with $W_T^* ={\cal I}_2(T,{\cal Y}_T^*).$

By the {\it strong Markov property}, we easily deduce that
\begin{equation}
    W_t^*=-\partial_{\l}\widehat{\cal Q}_{\Upsilon_t^{j,*}}(t,{\cal Y}_t^*).
\end{equation}

On the other hand, 
\begin{footnotesize}
    \begin{equation}
    \begin{aligned}
        \l^* w=&\mathbb{E}\left[\int_0^T e^{-\b t}{\cal Y}_t^* \left(c_t^* - \epsilon_0 (1- {\Upsilon}_t^{j,*}) - \epsilon_1  {\Upsilon}_t^{j,*}\right)dt+e^{-\b T}{\cal Y}_T^* {\cal I}_2(T,{\cal Y}_T^{*})  +\sum_{0 \leq t \le T}e^{-\b t}[\z_0{\cal Y}_t^*(\Delta  {\Upsilon}_t^{j,*})^+ + \z_1{\cal Y}_t^*(\Delta  {\Upsilon}_t^{j,*})^- ]\right]\\
        =&\mathbb{E}\left[\int_0^T e^{-\b t} \left(U_1(t,c_t^*) -L_0(1- {\Upsilon}_t^{j,*}) - L_1  {\Upsilon}_t^{j,*}\right)dt +e^{-\b T}U_2(T,W_T^*)\right]-J(j,\l^*).
    \end{aligned}
\end{equation}
\end{footnotesize}
Thus, it follows from the {\it weak-duality} in \eqref{eq:weak-duality} that 
\begin{equation}
    \begin{aligned}
        &\mathbb{E}\left[\int_0^T e^{-\b t} \left(U_1(t,c_t^*) -L_0(1- {\Upsilon}_t^{j,*}) - L_1  {\Upsilon}_t^{j,*}\right)dt +e^{-\b T}U_2(T,W_T^*)\right]\\
        =&J(j,\l^*)+\l^*w=\inf_{\l>0}\left[J(j,\l)+\l w\right]\ge V(j,w).
    \end{aligned}
\end{equation}
Since $(c^*,\pi^*,\Upsilon^{j,*})\in{\cal A}(j,w)$, we can conclude that 
\begin{footnotesize}
    \begin{equation}
\begin{aligned}
        V(j,w)=&\sup_{(c,\pi,\Upsilon)\in{\cal A}(j,w)}\mathbb{E}\left[\int_0^T e^{-\b t} \left(U_1(t,c_ ) -L_0(1- {\Upsilon}_t ) - L_1  {\Upsilon}\right)dt + e^{-\b T}U_2(T,W_T^{c,\pi,\Upsilon})\right]\\
        =&\mathbb{E}\left[\int_0^T e^{-\b t} \left(U_1(t,c_t^*) -L_0(1- {\Upsilon}_t^{j,*}) - L_1  {\Upsilon}_t^{j,*}\right)dt +e^{-\b T}U_2(T,W_T^*)\right].
\end{aligned}
\end{equation}
\end{footnotesize}
Hence, we have just proved that the duality theorem holds and the strategy $(c^*,\pi^*,\Upsilon_t^{j,*})$ is {\it optimal.} 

Note that the dynamics of the wealth $W_t^*$ is given by
$$dW_t^*=\left[rW_t^* + (\mu-r)\pi_t^* - c_t^* + \left(\e_0(1-\Upsilon_t^{j,*})+\e_1\Upsilon_t^{j,*}\right) \right]dt+\s \pi_t^* dB_t-\zeta_0(\Delta \Upsilon_t^{j,*})^+ -\zeta_1(\Delta \Upsilon_t^{j,*})^-.$$
Applying It\^{o}-Tanaka formula to $W_t^*=-\partial_{\l}\widehat{\cal Q}_{\Upsilon_t^{j,*}}(t,{\cal Y}_t^*)$ and comparing the coefficients of $dB_t$ in the two different expressions of $dW_t^*$, we can easily derive that 
\begin{equation*}
    \pi_t^*=\dfrac{\t}{\s}{\cal Y}_t^{^*} \partial_{\l\l}\widehat{\cal Q}_{\Upsilon_t^*}(t,{\cal Y}_t^*).
\end{equation*}
\end{proof}

We define two regions ${\cal R}_T^{0}$ and ${\cal R}_T^{1}$ in the $(t,w)$-coordinate as follows:
\begin{footnotesize}
    \begin{equation}
    \begin{aligned}
        {\cal R}_T^{0}:=&\left\{(t,w)\;:\;0\le t<T,\;\left(-\frac{\epsilon_1(1-e^{-r(T-t)})}{r}+\zeta_0\right){\bf I}_{\{t\in[0,T-T_1)\}}+\left(-\dfrac{\e_0(1-e^{-r(T-t)})}{r}\right){\bf I}_{\{t\in[T-T_1,T]\}}<w \right\},\vspace{2mm}\\
        {\cal R}_T^{1}:=&\left\{(t,w)\;:\;0\le t<T,\; -\dfrac{\e_1(1-e^{-r(T-t)})}{r} <w \right\}.
    \end{aligned}
\end{equation}
\end{footnotesize}
By Theorem \ref{thm:main}, for each $j\in\{0,1\}$, the function $(-\partial_{\l}\widehat{\cal Q}_j(t,\l)):{\cal N}_T\to {\cal R}_T^{j}$ is bijective. Moreover, {\it the optimal job-switching boundaries} $w_0(t)$ and $w_1(t)$ in the $(t,w)$-coordinate can be defined as 
\begin{equation}\label{eq:two-boundaries-wealth}
\begin{aligned}
    w_0(t):=&-\partial_{\l}\widehat{\cal Q}_0(t,\L_0(t))\;\;\mbox{for}\;\;t\in[0,T-T_1),\\
    w_1(t):=&-\partial_{\l}\widehat{\cal Q}_1(t,\L_1(t))\;\;\mbox{for}\;\;t\in[0,T).
\end{aligned}
\end{equation}

In terms of the optimal job-switching boundaries $w_0(t)$ and $w_1(t)$, we also define the following four regions in the $(t,w)$-coordinate:
\begin{equation}\label{de:domain-t-w}
    \begin{aligned}
        \mathcal{WR}_0^{(t,w)}:=& \{(t,w)\in{\cal R}_T^0\;:\; w_0(t)<w\},\;\;
        \mathcal{SR}_0^{(t,w)}:=\{(t,w)\in{\cal R}_T^0\;:\;  w\le w_0(t)\},\vspace{2mm}\\
        \mathcal{WR}_1^{(t,w)}:=& \{(t,w)\in{\cal R}_T^1\;:\;w<w_1(t)\},\;\;
        \mathcal{SR}_1^{(t,w)}:= \{(t,w)\in{\cal R}_T^1\;:\; w\ge w_1(t)\}.
    \end{aligned}
\end{equation}

By Theorem \ref{property-Q} and Lemma \ref{lem:properties-Q}, we directly obtain the following lemma.
\begin{lem}
   The optimal switching boundaries, $w_0(t)$ for $t\in [0,T-T_1)$ and $w_1(t)$ for $t\in [0,T)$, are continuous. Furthermore, 
\begin{equation}
    \lim_{t\to (T-T_1)-}w_0(t) = -\e_0\dfrac{1-e^{-rT_1}}{r}\;\;\mbox{and}\;\;\lim_{t\to T-}w_1(t)=\infty.
\end{equation}
\end{lem}

\begin{figure}[ht]
	\centering
    	\subfigure[$w_0(t)$]{\includegraphics[scale=0.45]{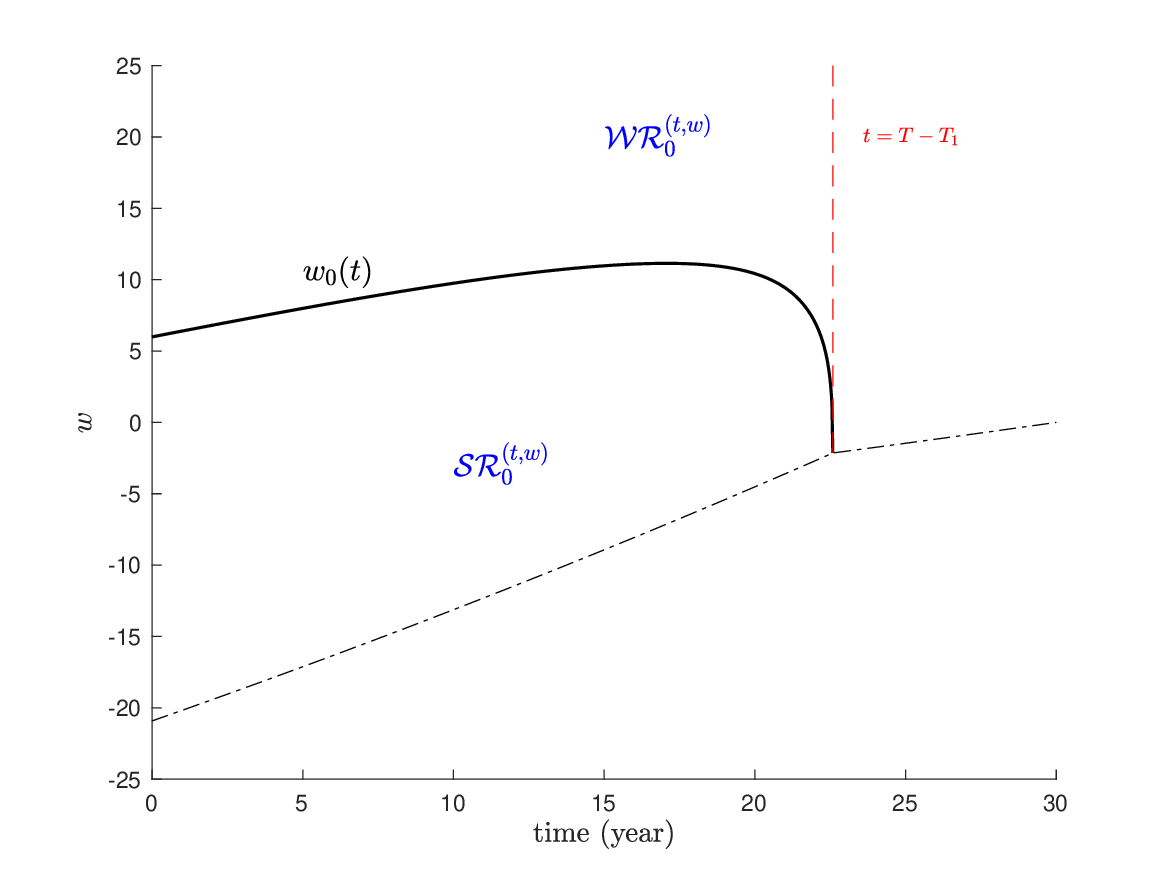}}
	\subfigure[$w_1(t)$]{\includegraphics[scale=0.45]{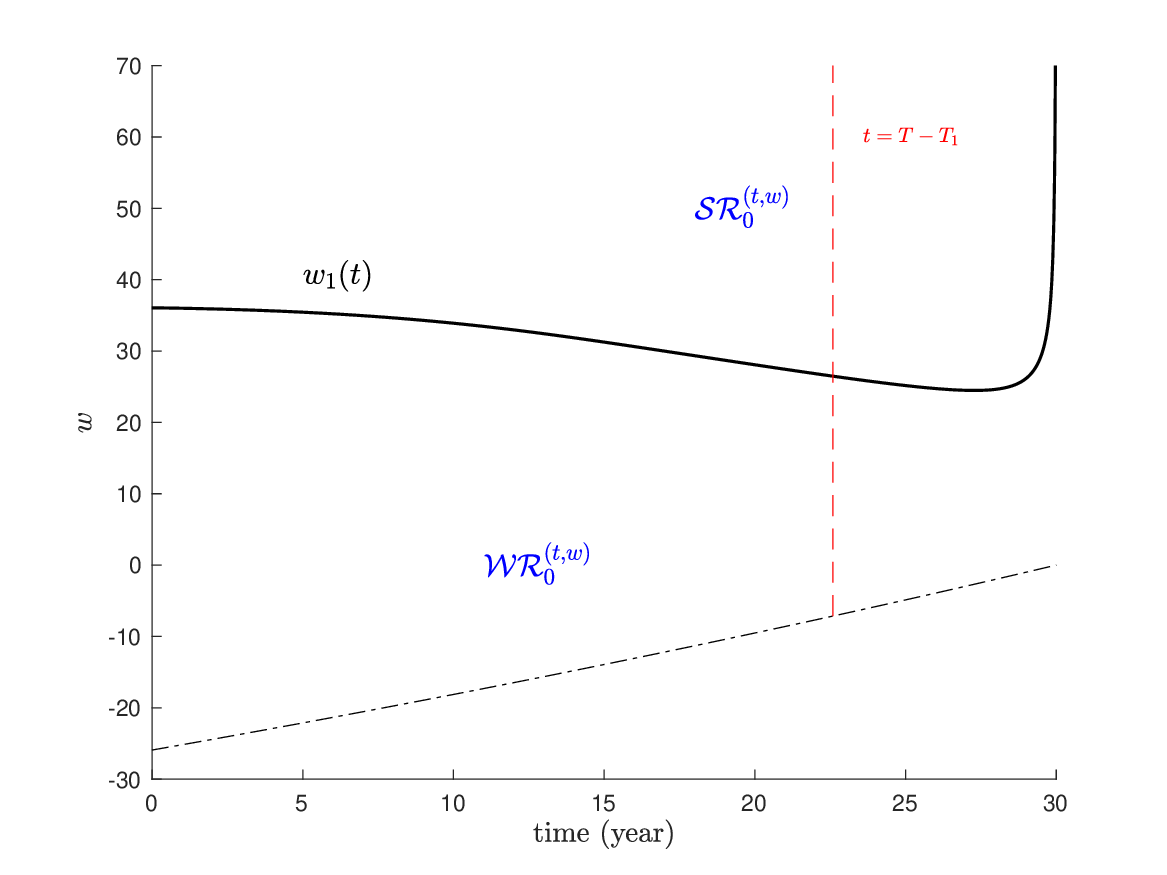}}
	\caption{The four regions $\mathcal{WR}_0^{(t,w)}$, $\mathcal{SR}_0^{(t,w)}$, $\mathcal{WR}_1^{(t,w)}$, and $\mathcal{SR}_1^{(t,w)}$ in $(t,w)$-domain \label{fig:regions-w}}
\end{figure}

Figure \ref{fig:regions-w} represents the optimal job-switching boundaries defined in \eqref{eq:two-boundaries-wealth} and the four regions defined in \eqref{de:domain-t-w}. In Job ${\cal D}_0$, as the agent's wealth decreases and approaches the job-switching boundary $w_0(t)$, he/she switch to Job ${\cal D}_1$, which offers a higher income but comes with a cost $\zeta_0$. However, if the remaining time until the mandatory retirement date is less than or equal to $T_1$, job switching does not occur because the cost $\zeta_0$ is greater than the increase in the present value of future income obtained by switching to Job ${\cal D}_1$. Conversely, when the agent is in ${\cal D}_1$, if his/her wealth increases sufficiently to hit the job-switching boundary $w_1(t)$, he/she switches to ${\cal D}_0$ to reduce the disutility caused by labor, even if it means incurring a cost $\zeta_1$. However, if there is very little time remaining until the mandatory retirement date, it is more advantageous to retire and eliminate utility rather than paying the cost to switch jobs and reduce disutility. Therefore, job switching does not occur in this case.

\section{Numerical Results}\label{sec:numerical}

In this section, our objective is to present numerical results concerning the optimal consumption and investment strategy within each job, denoted as ${\cal D}_0$ and ${\cal D}_1$, which were obtained in the previous section. Specifically, we examine how job-switching costs and mandatory retirement dates affect the optimal consumption and investment strategy in each job.

Throughout this section, we make the assumption that the agent has constant relative risk aversion and the felicity functions $U_1$ and $U_2$ are given by
\begin{equation}\label{eq:CRRA}
    U_1(c) =\dfrac{c^{1-\g}}{1-\g}\;\;\mbox{and}\;\;U_2(T,w)=\left(\dfrac{1-e^{-K(T_D-T)}}{K}\right)^\g\dfrac{w^{1-\g}}{1-\g}\;\;\;\g>0,\;\g\ne1.
\end{equation}
Here, $\g$ is the coefficient of relative risk aversion, $T_D>T$ is the expected death time, and $K$ is the Merton coefficient defined as 
$$
K:=r+\dfrac{\b-r}{\g} +\dfrac{\g-1}{\g^2}\dfrac{\t^2}{2}.
$$
\begin{rem}
    The felicity function $U_2(T,w)$ for the bequest is the value function of the following classical Merton problem (see \citet{M69,M71}):
    \begin{equation}
        U_2(T,w)=\sup_{(c,\pi)\in{\cal A}_R(T,w)}\mathbb{E}_T\left[\int_T^{T_D}e^{-\b(t-T)}U_1(c_t)dt\right],
    \end{equation}
    where ${\cal A}_R(T,w)$ is the set of all admissible consumption and investment strategies such that (i) $c\ge 0$ and $\pi$ are $\mathbb{F}$-progressively measurable processes satisfying $\int_T^{T_D}c_t dt <\infty$ a.s., $\int_T^{T_D}\pi_t^2 dt <\infty$, a.s., (ii) the wealth $W_t^{c,\pi}$ corresponding to the strategy $(c,\pi)$ follows $dW_t^{c,\pi} = [rW_t^{c,\pi}+(\m-r)\pi_t -c_t]dt+\s\pi_t dB_t\;\;\mbox{for}\;\;t\in [T,T_D]$ with $W_T^{c,\pi}=w$, (iii) $W_t^{c,\pi}\ge 0$ for $t\in [T,T_D]$, a.s..  

    Indeed, $U_2(T,w)$ represents the expected utility achievable by optimizing consumption and investment from the mandatory retirement date to the expected time of death, given that the agent's wealth is $w$ at the retirement date.
\end{rem}

Under the felicity functions in \eqref{eq:CRRA}, the convex conjugate function $\widetilde{U}_i$ are given by 
\begin{equation}
    \widetilde{U}_1(\l) =\dfrac{\g}{1-\g}\l^{-\frac{1-\g}{\g}}\;\;\mbox{and}\;\;\widetilde{U}_2(T,\l)=\dfrac{1-e^{-K(T_D-T)}}{K}\dfrac{\g}{1-\g}\l^{-\frac{1-\g}{\g}}.
\end{equation}
It follows that 
\begin{equation}
    \begin{aligned}
        {\cal Q}_R(t,\l) =\dfrac{1-e^{-K(T_D-t)}}{K}\dfrac{\g}{1-\g}\l^{-\frac{1-\g}{\g}}.
    \end{aligned}
\end{equation}

In the following proposition, we not only provide the solution to the optimal switching problem but also present an integral equation that the two free boundaries $\L_0(t)$ and $\L_1(t)$ satisfy.
\begin{pro}\label{pro:integral-equations}
    ${\cal Q}_0$ and ${\cal Q}_1$ have the following integral equation representations:
    \begin{footnotesize}
        \begin{equation}
        \begin{aligned}
            {\cal Q}_0(t,\l)=&\dfrac{1-e^{-r(T-t)}}{r}\e_0\l -\dfrac{1-e^{-\b(T-t)}}{\b}L_0+\left\{
           (\e_1-\e_0-r\zeta_0)\l\int_t^{T-T_1} e^{-r (s-t)}{\bf N}\left(d_+(s-t,\dfrac{\l}{\Lambda_0(s)})\right)ds\right.\\ -&\left.(L_1-L_0)\int_t^{T-T_1} e^{-\b(s-t)}{\bf N}\left(d_-(s-t,\dfrac{\l}{\Lambda_0(s)})\right)ds\right\}{\rm\bf I}_{\{t\in[0,T-T_1)\}}\;\;\;\mbox{for}\;\;\;0<\l <\L_0(t),\\
            {\cal Q}_1(t,\l)=&\dfrac{1-e^{-r(T-t)}}{r}\e_1\l -\dfrac{1-e^{-\b(T-t)}}{\b}L_1-
           (\e_1-\e_0+r\zeta_1)\l\int_t^{T} e^{-r (s-t)}{\bf N}\left(-d_+(s-t,\dfrac{\l}{\Lambda_1(s)})\right)ds\\ +&(L_1-L_0)\int_t^{T} e^{-\b(s-t)}{\bf N}\left(-d_-(s-t,\dfrac{\l}{\Lambda_1(s)})\right)ds\;\;\;\mbox{for}\;\;\;\l>\L_1(t),
        \end{aligned}
    \end{equation}
    \end{footnotesize}
    and 
    \begin{footnotesize}
        \begin{equation}
            \begin{aligned}
            {\cal Q}_0(t,\l)=&{\cal Q}_1(t,\l)-\zeta_0 \l\;\;\mbox{for}\;\;\l \ge \L_0(t), \\
            {\cal Q}_1(t,\l)=&{\cal Q}_0(t,\l)-\zeta_1 \l\;\;\mbox{for}\;\;0<\l \le \L_1(t),
        \end{aligned}
        \end{equation}
    \end{footnotesize}
    where ${\bf N}(\cdot)$ is the standard normal distribution function,  $d_{+}(t,\l)$ and $d_{-}(t,\l)$ are given by 
    \begin{equation*}
        d_{\pm}(t,\l):=\dfrac{\log{\l}+(\b-r\pm\frac{1}{2}\t^2)t}{\t\sqrt{t}}.
    \end{equation*}
    Moreover, the two free boundaries $\L_0(t)$ and $\L_1(t)$ satisfy the following coupled integral equations:
    \begin{equation}\label{eq:coupled-IE}
        \begin{aligned}
            \zeta_0 \L_0 (t) &= {\cal Q}_1(t,\L_0(t))-{\cal Q}_0(t,\L_0(t))\;\;\mbox{for}\;\;t\in[0,T-T_1),\vspace{2mm}\\
            -\zeta_1 \L_1 (t) &= {\cal Q}_1(t,\L_1(t))-{\cal Q}_0(t,\L_1(t))\;\;\mbox{for}\;\;t\in[0,T).
        \end{aligned}
    \end{equation}
\end{pro}
\begin{proof}
    Note that 
    \begin{equation*}
    \begin{cases}
    &-\left(\partial_t {\cal Q}_0 +{\cal L}{\cal Q}_0 +\e_0 \l -L_0\right)=\left[(\e_1-\e_0-r\zeta_0)\l-(L_1-L_0)\right] {\rm\bf I}_{\{\l\ge \L_0(t)\}},\vspace{2mm}\\ 
    &-\left(\partial_t {\cal Q}_1 +{\cal L}{\cal Q}_1 +\e_1 \l -L_1\right)=\left[(L_1-L_0)-(\e_1-\e_0+r\zeta_1)\l\right] {\rm\bf I}_{\{0<\l\le \L_1(t)\}}, \vspace{2mm}\\
    &{\cal Q}_0(T,\l) = {\cal Q}_1(T,\l)=0,\quad\l\in(0,+\infty).
    \end{cases}
\end{equation*}

Let us temporarily denote ${\cal Y}_s^{t,\l}$ by
\begin{equation}
    {\cal Y}_s^{t,\l} := \l e^{\b(s-t)}{\cal H}_s/{\cal H}_t.
\end{equation}

By applying It\^{o}'s lemma to $e^{-\b (s-t)}{\cal Q}_0(s,{\cal Y}_s^{t,\l})$, we have 
\begin{equation*}
    \begin{aligned}
        d\left(e^{-\b (s-t)}{\cal Q}_0(s,{\cal Y}_s^{t,\l} )\right)=e^{-\b(s-t)} \left(\partial_s {\cal Q}_0(s,{\cal Y}_s^{t,\l})+{\cal L}{\cal Q}_0(s,{\cal Y}_s^{t,\l})\right)ds -\t e^{-\b (s-t)} {\cal Y}_s^{t,\l} \partial_\l {\cal Q}_0(s,{\cal Y}_s^{t,\l})dB_s,
    \end{aligned}
\end{equation*}
or 
\begin{equation}\label{eq:Ito-Q0}
    {\cal Q}_0(t,\l) = -\int_t^T e^{-\b (s-t)} \left(\partial_s {\cal Q}_0(s,{\cal Y}_s^{t,\l})+{\cal L}{\cal Q}_0(s,{\cal Y}_s^{t,\l})\right)ds + \t\int_t^T e^{-\b (s-t)}{\cal Y}_s^{t,\l} \partial_{\l}{\cal Q}_0(s,{\cal Y}_s^{t,\l})dB_s.
\end{equation}

Since $\partial_\l {\cal Q}_0$ is bounded (see Theorem \ref{property-Q}), it is easy to check that 
\begin{equation*}
    \mathbb{E}_t\left[\int_t^T \left(e^{-\b(s-t)}{\cal Y}_s^{t,\l} \partial_{\l}{\cal Q}_0(s,{\cal Y}_s^{t,\l}\right)^2 ds\right]<\infty.
\end{equation*}
That is, 
\begin{equation*}
    \mathbb{E}_t \left[\int_t^T e^{-\b(s-t)}{\cal Y}_s^{t,\l} \partial_{\l}{\cal Q}_0(s,{\cal Y}_s^{t,\l} ) dB_s\right]=0.
\end{equation*}

Taking the conditional expectation at time $t$ on the both side of \eqref{eq:Ito-Q0} yields that 
\begin{equation}\label{eq:integral_Q0-expectation}
    \begin{aligned}
        {\cal Q}_0(t,\l)=&-\mathbb{E}_t\left[\int_t^T e^{-\b (s-t)} \left(\partial_s {\cal Q}_0(s,{\cal Y}_s^{t,\l})+{\cal L}{\cal Q}_0(s,{\cal Y}_s^{t,\l})\right)ds\right]\\
=&\mathbb{E}_t\left[\int_t^T e^{-\b(s-t)}\left\{\e_0{\cal Y}_s^{t,\l} -L_0+\left[(\e_1-\e_0-r\zeta_0){\cal Y}_s^{t,\l}-(L_1-L_0)\right]{\rm\bf I}_{\{{\cal Y}_s^{t,\l} \ge \L_0(s)\}}\right\}ds\right]\\
=&\mathbb{E}_t\left[\int_t^{T-T_1} e^{-\b(s-t)}\left((\e_1-\e_0-r\zeta_0){\cal Y}_s^{t,\l}-(L_1-L_0)\right){\rm\bf I}_{\{{\cal Y}_s^{t,\l} \ge \L_0(s)\}}ds\right]{\rm\bf I}_{\{t\in[0,T-T_1)\}}\\
+&\dfrac{1-e^{-r(T-t)}}{r}\e_0\l -\dfrac{1-e^{-\b(T-t)}}{\b}L_0,
    \end{aligned}
\end{equation}
where we have used the fact that $\L_0(t)=+\infty$ for $t\in[T-T_1,T)$.

Similarly,
\begin{equation}\label{eq:integral_Q1-expectation}
    \begin{aligned}
        {\cal Q}_1(t,\l) =&\mathbb{E}_t\left[\int_t^T e^{-\b(s-t)}\left((L_1-L_0)-(\e_1-\e_0+r\zeta_1){\cal Y}_s^{t,\l}\right){\rm\bf I}_{\{ 0<{\cal Y}_s^{t,\l} \le \L_1(s)\}} ds\right]\\
        +&\dfrac{1-e^{-r(T-t)}}{r}\e_1\l -\dfrac{1-e^{-\b(T-t)}}{\b}L_1.
    \end{aligned}
\end{equation}
By applying the equalities (A.7) and (A.8) in \citet{JEON2022125542} to the equations \eqref{eq:integral_Q0-expectation} and \eqref{eq:integral_Q1-expectation}, we can easily obtain the integral equation representations for ${\cal Q}_0$ and ${\cal Q}_1$ as provided in the theorem's statement.

Since ${\cal P}(t,\l)={\cal Q}_1(t,\l)-{\cal Q}_0(t,\l)$ is continuously differentiable function of $\l$ (see Theorem \ref{theorem for free boundaires}), we can directly obtain that 
\begin{equation*}
        \begin{aligned}
            \zeta_0 \L_0 (t) &= {\cal Q}_1(t,\L_0(t))-{\cal Q}_0(t,\L_0(t))\;\;\mbox{for}\;\;t\in[0,T-T_1),\vspace{2mm}\\
            \zeta_1 \L_1 (t) &= {\cal Q}_1(t,\L_1(t))-{\cal Q}_0(t,\L_1(t))\;\;\mbox{for}\;\;t\in[0,T).
        \end{aligned}
    \end{equation*}

\end{proof}

For $j\in\{0,1\}$, let us denote ${\cal W}_j(t,\l)$ and $\Pi_j(t,\l)$ by 
\begin{equation}
    {\cal X}_j(t,\l) = -\partial_{\l}\widehat{\cal Q}_j(t,\l)\;\;\mbox{and}\;\;\Pi_j(t,\l)=\dfrac{\t}{\s}\l\partial_{\l\l}\widehat{\cal Q}_j(t,\l).
\end{equation}
Since $\widehat{\cal Q}_j(t,\l)={\cal Q}_R(t,\l)+{\cal Q}_j(t,\l)$, it follows from Proposition \ref{pro:integral-equations} that 
  \begin{footnotesize}
        \begin{equation}\label{eq:wealth-IE}
        \begin{aligned}
            {\cal W}_0(t,\l)=&\dfrac{1-e^{-K(T_D-t)}}{K}\l^{-\frac{1}{\g}}-\dfrac{1-e^{-r(T-t)}}{r}\e_0 -\left\{
           (\e_1-\e_0-r\zeta_0)\int_t^{T-T_1} e^{-r (s-t)}{\bf N}\left(d_+(s-t,\dfrac{\l}{\Lambda_0(s)})\right)ds\right.\\ +&\left.(\e_1-\e_0-r\zeta_0)\int_t^{T-T_1} e^{-r (s-t)}{\bf n}\left(d_+(s-t,\dfrac{\l}{\Lambda_0(s)})\right)\dfrac{1}{\t\sqrt{s-t}}ds\right.\\-&\left.(L_1-L_0)\int_t^{T-T_1} e^{-\b(s-t)}{\bf n}\left(d_-(s-t,\dfrac{\l}{\Lambda_0(s)})\right)\dfrac{1}{\l\t\sqrt{s-t}}ds\right\}{\rm\bf I}_{\{t\in[0,T-T_1)\}},\\
            {\cal W}_1(t,\l)=&\dfrac{1-e^{-K(T_D-t)}}{K}\l^{-\frac{1}{\g}}-\dfrac{1-e^{-r(T-t)}}{r}\e_1 +
           (\e_1-\e_0+r\zeta_1)\int_t^{T} e^{-r (s-t)}{\bf N}\left(-d_+(s-t,\dfrac{\l}{\Lambda_1(s)})\right)ds\\-&(\e_1-\e_0+r\zeta_1)\int_t^{T} e^{-r (s-t)}{\bf n}\left(-d_+(s-t,\dfrac{\l}{\Lambda_1(s)})\right)\dfrac{1}{\t\sqrt{s-t}}ds\\+&(L_1-L_0)\int_t^{T} e^{-\b(s-t)}{\bf n}\left(-d_-(s-t,\dfrac{\l}{\Lambda_1(s)})\right)\dfrac{1}{\l\t\sqrt{s-t}}ds
                 \end{aligned}
    \end{equation}
  \end{footnotesize}
 and
    \begin{footnotesize}
         \begin{equation}\label{eq:investment-IE}
        \begin{aligned}
            \Pi_0(t,\l)=&\dfrac{\t}{\s}\dfrac{1-e^{-K(T_D-t)}}{\g K}\l^{-\frac{1}{\g}}-\dfrac{\t}{\s}\left\{
           (\e_1-\e_0-r\zeta_0)\int_t^{T-T_1} e^{-r (s-t)}{\bf n}\left(d_+(s-t,\dfrac{\l}{\Lambda_0(s)})\right)\dfrac{d_-(s-t,\dfrac{\l}{\L_0(s)})}{\t^2(s-t)}ds \right.\\-&\left.(L_1-L_0)\int_t^{T-T_1} e^{-\b(s-t)}{\bf n}\left(d_-(s-t,\dfrac{\l}{\Lambda_0(s)})\right)\dfrac{d_+(s-t,\dfrac{\l}{\L_0(s)})}{\l\t^2(s-t)}ds\right\}{\rm\bf I}_{\{t\in[0,T-T_1)\}},\\
           \Pi_1(t,\l)=&\dfrac{\t}{\s}\dfrac{1-e^{-K(T_D-t)}}{\g K}\l^{-\frac{1}{\g}}-\dfrac{\t}{\s}\left\{
           (\e_1-\e_0+r\zeta_1)\int_t^{T} e^{-r (s-t)}{\bf n}\left(-d_+(s-t,\dfrac{\l}{\Lambda_1(s)})\right)\dfrac{d_-(s-t,\dfrac{\l}{\L_1(s)})}{\t^2(s-t)}ds\right.\\-&\left.(L_1-L_0)\int_t^{T} e^{-\b(s-t)}{\bf n}\left(-d_-(s-t,\dfrac{\l}{\Lambda_1(s)})\right)\dfrac{1}{\l\t\sqrt{s-t}}\dfrac{d_+(s-t,\dfrac{\l}{\L_1(s)})}{\l\t^2(s-t)}ds\right\},
        \end{aligned}
    \end{equation}
    \end{footnotesize}
where ${\bf n}(\cdot)$ is the probability density function of the standard normal distribution.
    
From Theorem \ref{thm:main}, we have 
$$
c_t^*=({\cal Y}_t^*)^{-\frac{1}{\g}},\;\;W_t^{*}={\cal W}_{\Upsilon_t^{j,*}}(t,{\cal Y}_t^*)\;\;\mbox{and}\;\;\pi_t^{*}=\Pi_{\Upsilon_t^{j,*}}(t,{\cal Y}_t^*)
$$
We can obtain precise numerical solutions for the coupled integral equations in \eqref{eq:coupled-IE} satisfied by $\Lambda_0(t)$ and $\Lambda_1(t)$ by utilizing the recursive integration method proposed by \citet{Huang}. Based on the numerical solutions obtained for $\Lambda_0(t)$ and $\Lambda_1(t)$ in this manner, we can accurately and easily compute the integral equations in \eqref{eq:wealth-IE} and \eqref{eq:investment-IE} governing wealth and investment within each job ${\cal D}_i\;(j=0,1)$ by simply employing the trapezoidal rule. For more detailed process of solving the coupled integral equations using the recursive integration method, please refer to \citet{JEON201973}, \citet{JEON2019101049}, \citet{JEON2021113508}. 

Figure \ref{fig:boundary} illustrates the impact of increasing job-switching cost $\zeta_j$ on the optimal job-switching boundary $w_j(t)$ for each $j\in\{0,1\}$. As can be seen in the figure, as the cost increases, the boundary $w_0(t)$ for deciding to switch to job ${\cal D}_1$ decreases, while the boundary $w_1(t)$ for deciding to switch to job ${\cal D}_0$ increases. In other words, as job-switching costs rise, the decision to switch is delayed further.

Figure \ref{fig:strategy-0} and Figure \ref{fig:strategy-1} present the impact of the costs associated with switching from one job to another within each respective job on the consumption and investment strategy. In job ${\cal D}_0$, we can observe that as the cost of switching to ${\cal D}_1$ increases, both consumption and investment decrease. However, in job ${\cal D}_1$, as the cost of switching to ${\cal D}_0$ increases, investment decreases while consumption increases.

Figure \ref{fig:strategy-0-T} and Figure \ref{fig:strategy-1-T} depict the impact of the mandatory retirement date $T$ on the agent's optimal consumption and investment in each respective job. In job ${\cal D}_0$, as the mandatory retirement date increases, the agent increases optimal investment, while consumption rises at lower wealth levels and declines at higher wealth levels with the increasing mandatory retirement date $T$. In job ${\cal D}_0$, as the mandatory retirement date ($T$) increases, the agent increases optimal investment, while consumption rises at lower wealth levels and decreases at higher wealth levels with the increasing $T$. In job ${\cal D}_1$, with the increasing mandatory retirement date ($T$), both consumption and investment rise.

\begin{figure}[!ht]
	\centering
    	\subfigure[$w_0(t)$]{\includegraphics[scale=0.6]{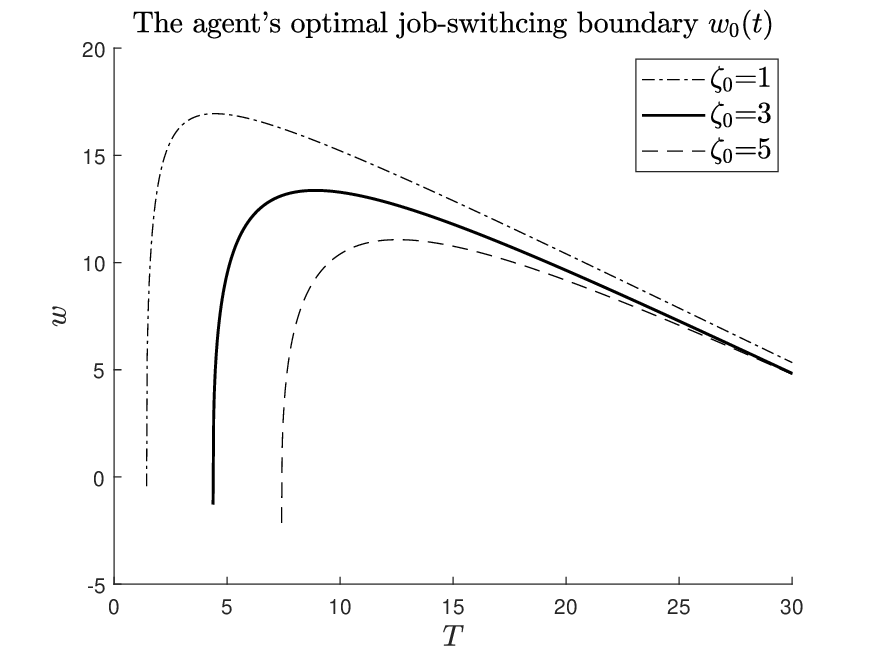}}
	\subfigure[$w_1(t)$]{\includegraphics[scale=0.6]{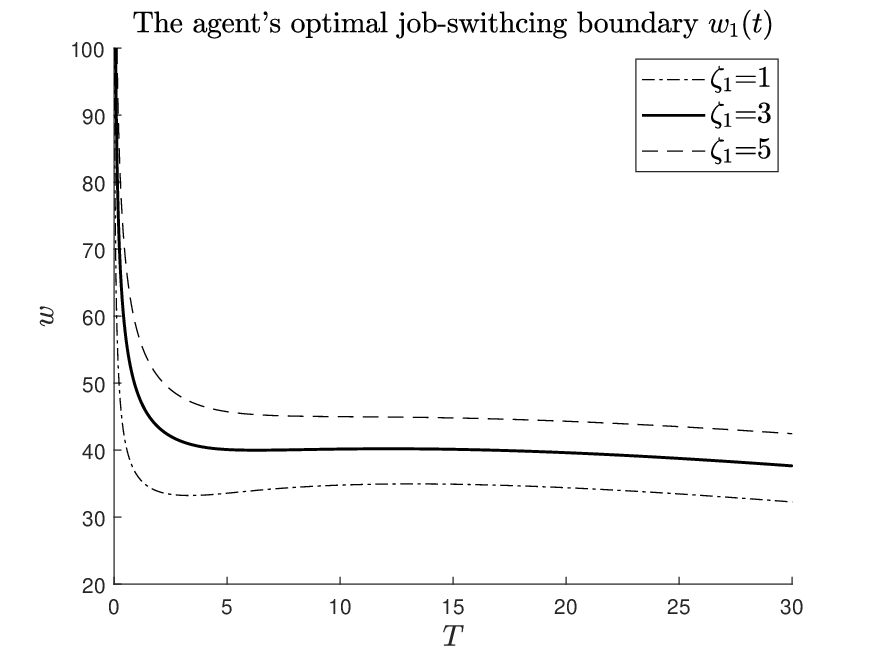}}
	\caption{The optimal job-switching boundaries $w_0(t)$ and $w_1(t)$. The parameters are given by $\b=0.02$, $r=0.01$, $\mu=0.07$, $\s=0.2$, $\e_0=0.3$, $\e_1=1$, $L_0=0.5$, $L_1=1$, $T=30$, $T_2=20$, $\zeta_0=3$, $\zeta_1=1$, and $\g=3$.\label{fig:boundary}}
\end{figure}

\begin{figure}[!ht]
	\centering
    	\subfigure[$c_t$ and $W_t$]{\includegraphics[scale=0.6]{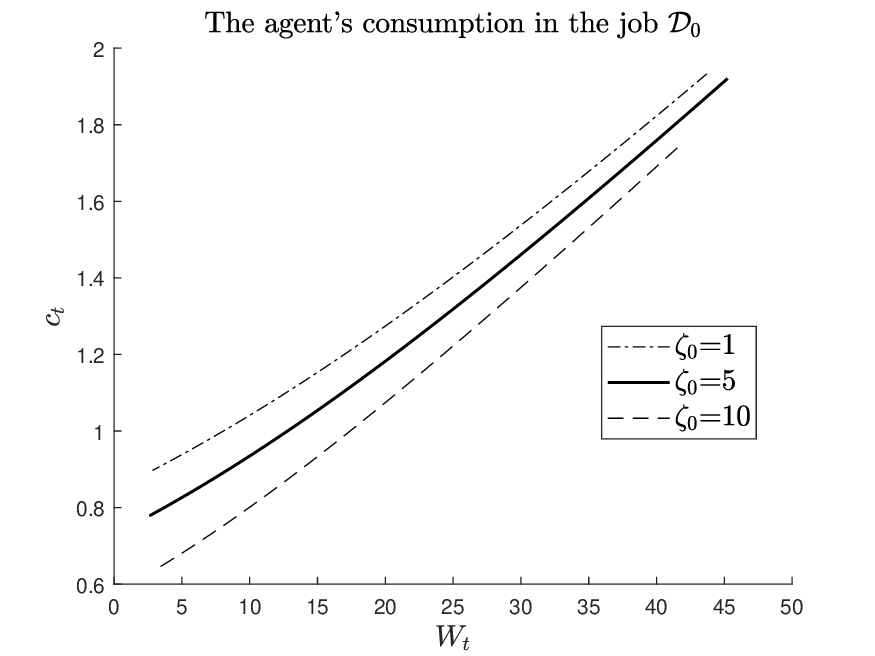}}
	\subfigure[$\pi_t$ and $W_t$]{\includegraphics[scale=0.6]{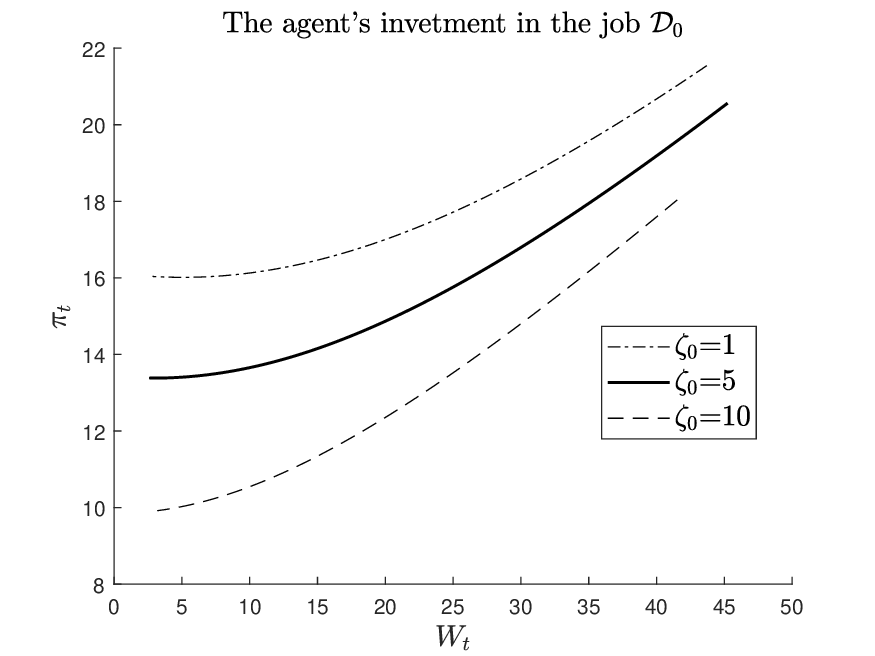}}
	\caption{The optimal consumption and investment strategy in the job ${\cal D}_0$. The parameters are given by $\b=0.02$, $r=0.01$, $\mu=0.07$, $\s=0.2$, $\e_0=0.3$, $\e_1=1$, $L_0=0.5$, $L_1=1$, $T=30$, $T_2=20$, $\zeta_0=3$, $\zeta_1=1$, and $\g=3$.\label{fig:strategy-0}}
\end{figure}
\begin{figure}[!ht]
	\centering
    	\subfigure[$c_t$ and $W_t$]{\includegraphics[scale=0.6]{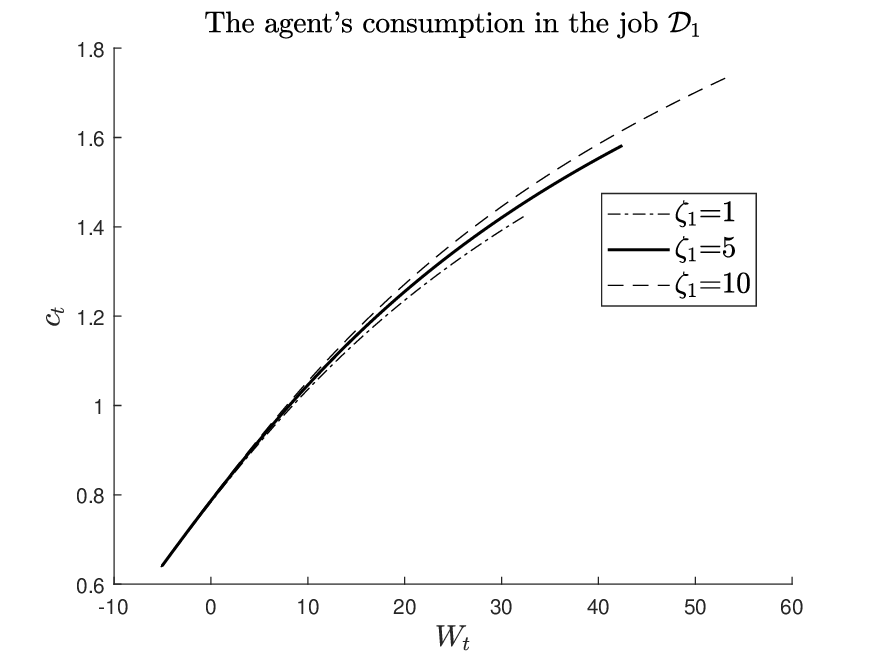}}
	\subfigure[$\pi_t$ and $W_t$]{\includegraphics[scale=0.6]{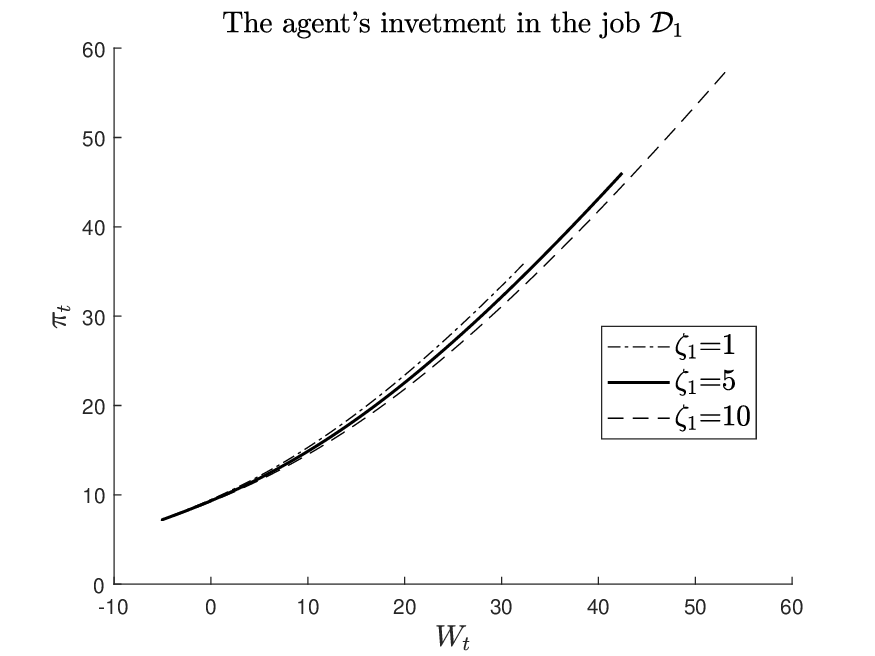}}
	\caption{The optimal consumption and investment strategy in the job ${\cal D}_1$. The parameters are given by $\b=0.02$, $r=0.01$, $\mu=0.07$, $\s=0.2$, $\e_0=0.3$, $\e_1=1$, $L_0=0.5$, $L_1=1$, $T=30$, $T_2=20$, $\zeta_0=3$, $\zeta_1=1$, and $\g=3$.\label{fig:strategy-1}}
\end{figure}

\begin{figure}[!ht]
	\centering
    	\subfigure[$c_t$ and $W_t$]{\includegraphics[scale=0.6]{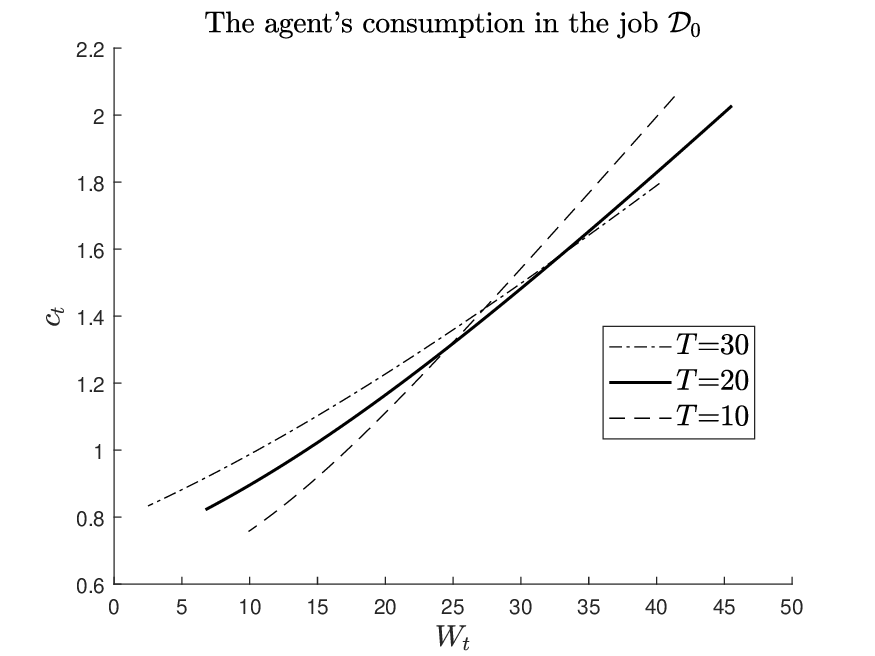}}
	\subfigure[$\pi_t$ and $W_t$]{\includegraphics[scale=0.6]{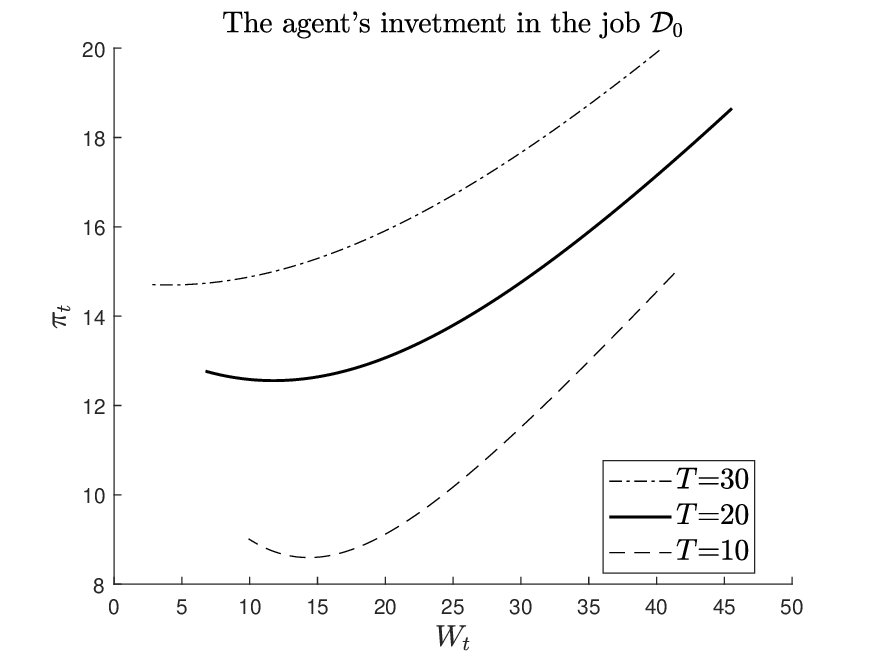}}
	\caption{The optimal consumption and investment strategy in the job ${\cal D}_0$. The parameters are given by $\b=0.02$, $r=0.01$, $\mu=0.07$, $\s=0.2$, $\e_0=0.3$, $\e_1=1$, $L_0=0.5$, $L_1=1$, $T_2=20$, $\zeta_0=3$, $\zeta_1=1$, and $\g=3$.\label{fig:strategy-0-T}}
\end{figure}

\begin{figure}[!ht]
	\centering
    	\subfigure[$c_t$ and $W_t$]{\includegraphics[scale=0.6]{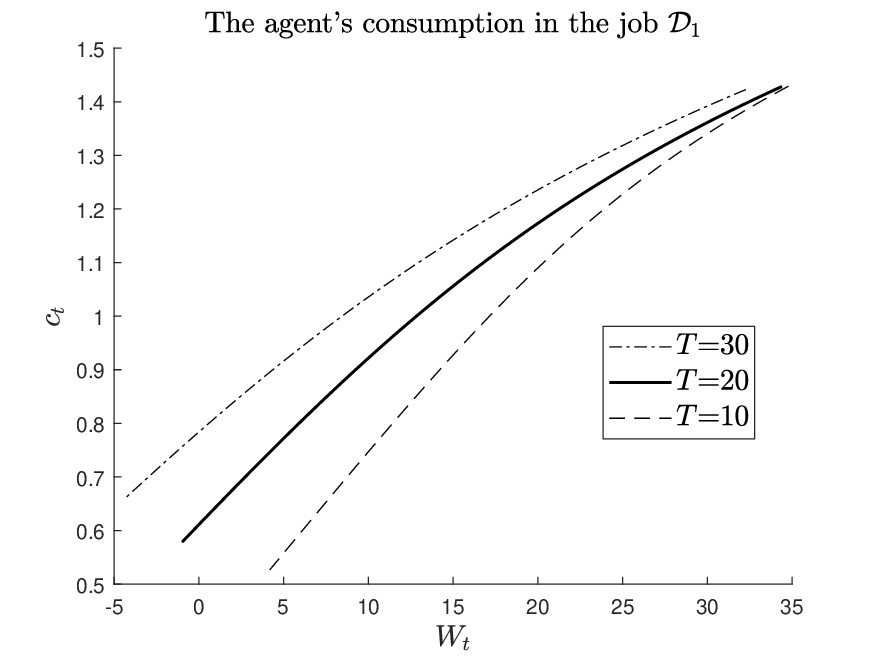}}
	\subfigure[$\pi_t$ and $W_t$]{\includegraphics[scale=0.6]{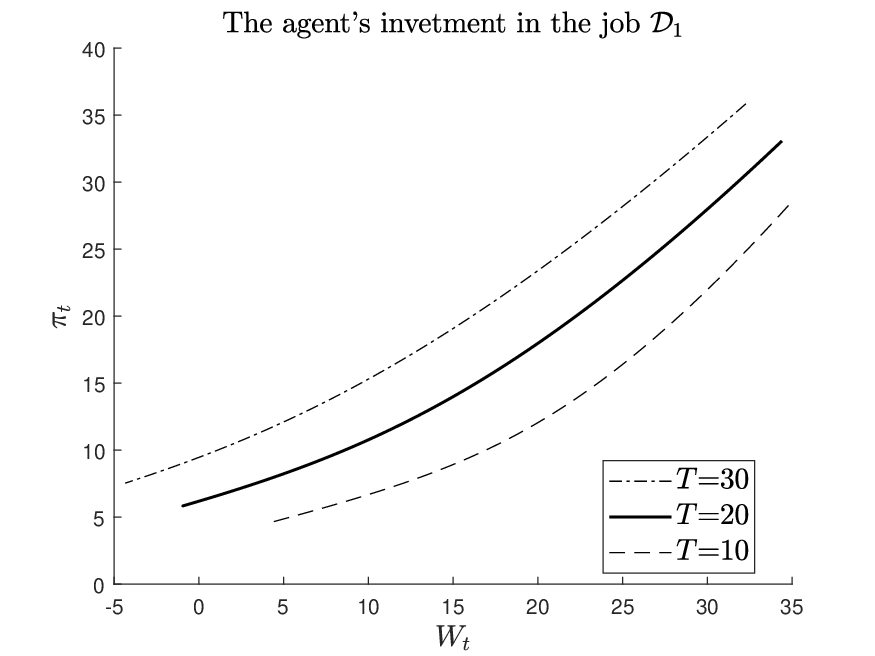}}
	\caption{The optimal consumption and investment strategy in the job ${\cal D}_0$. The parameters are given by $\b=0.02$, $r=0.01$, $\mu=0.07$, $\s=0.2$, $\e_0=0.3$, $\e_1=1$, $L_0=0.5$, $L_1=1$, $T_2=20$, $\zeta_0=3$, $\zeta_1=1$, and $\g=3$.\label{fig:strategy-1-T}}
\end{figure}

\section{Concluding Remarks}\label{sec:conclusion}

Our study in this paper focuses on the optimal consumption and investment problem for an agent, considering both mandatory retirement and job-switching costs. The agent's optimization problem encompasses two key features: finite-horizon optimal switching and stochastic control. Furthermore, to the best of our knowledge, our paper is the first to incorporate both job-switching costs and a mandatory retirement date into the optimal consumption and investment problem. Using the dual-martingale approach, we define a dual problem represented as a finite-horizon pure optimal switching problem.

From this optimal switching problem, we derive a parabolic double obstacle problem and employ non-standard techniques, combined with PDE theory, to thoroughly examine not only the solutions to this double obstacle problem but also the analytical properties of the related two free boundaries. Furthermore, we construct the optimal job-switching strategy using these two free boundaries. This approach allows us to recover the solution to the optimal switching problem from the solution to the double obstacle problem. Through the duality theorem, we not only characterize optimal strategies but also derive their integral equation representations, enabling us to provide numerical solutions.

Our model can be extended to incorporate various factors, such as ambiguity, non-Markovian environments, multiple jobs, partial information, and more. These extensions will be explored in future research.

\bibliographystyle{apalike}
\bibliography{Yang-Jeon-230923}

\end{document}